%% file: main.tex
\documentclass{article}

 \PassOptionsToPackage{numbers, compress, sort}{natbib}
\usepackage[final]{neurips_2022}

\usepackage[utf8]{inputenc} 
\usepackage[T1]{fontenc}    
\usepackage{hyperref}       
\usepackage{url}            
\usepackage{booktabs}       
\usepackage{amsfonts}       
\usepackage{nicefrac}       
\usepackage{microtype}      
\usepackage{xcolor}         
\usepackage{multirow}
\usepackage{graphicx}
\usepackage{mathtools}
\input{math.tex}

\title{Faster Stochastic Algorithms for Minimax Optimization under Polyak--{\L}ojasiewicz Conditions}

\author{%
	Lesi Chen \\
	School of Data Science \\ 
	Fudan University \\
	\texttt{lschen19@fudan.edu.cn}
	\And
 	Boyuan Yao\\
 	School of Data Science \\ 
	Fudan University \\
 	\texttt{byyao19@fudan.edu.cn}
 	\And
	Luo Luo\thanks{The corresponding author} \\
	School of Data Science \\ 
	Fudan University \\ 
	\texttt{luoluo@fudan.edu.cn}
}

\begin{document}

\maketitle

\begin{abstract}

This paper considers stochastic first-order algorithms for minimax optimization under Polyak--{\L}ojasiewicz (PL) conditions. 
We propose SPIDER-GDA for solving the finite-sum problem of the form $\min_x \max_y f(x,y)\triangleq \frac{1}{n} \sum_{i=1}^n f_i(x,y)$, where the objective function $f(x,y)$ is $\mu_x$-PL in $x$ and $\mu_y$-PL in $y$; and each $f_i(x,y)$ is $L$-smooth. We prove SPIDER-GDA could find an $\epsilon$-optimal solution within ${\mathcal O}\left((n + \sqrt{n}\,\kappa_x\kappa_y^2)\log (1/\epsilon)\right)$ stochastic first-order oracle (SFO) complexity, which is better than the state-of-the-art method whose SFO upper bound is ${\mathcal O}\big((n + n^{2/3}\kappa_x\kappa_y^2)\log (1/\epsilon)\big)$, where $\kappa_x\triangleq L/\mu_x$ and $\kappa_y\triangleq L/\mu_y$.
For the ill-conditioned case, we provide an accelerated algorithm to reduce the computational cost further. It achieves $\tilde{{\mathcal O}}\big((n+\sqrt{n}\,\kappa_x\kappa_y)\log (\kappa_y/\epsilon) \log(1/\epsilon)\big)$ SFO upper bound when $\kappa_y \gtrsim \sqrt{n}$. Our ideas can also be applied to a more general setting where the objective function only satisfies the PL condition for one variable. Numerical experiments validate the superiority of proposed methods.

\end{abstract}

\section{Introduction}
 
This paper focuses on the smooth minimax optimization problem of the form
\begin{align}\label{prob:main}
    \min_{x \in \BR^{d_x}} \max_{y \in \BR^{d_y}} f(x,y) \triangleq \frac{1}{n} \sum_{i=1}^n f_i(x,y),
\end{align}
which covers a lot of important applications in machine learning such as reinforcement learning~\cite{du2017stochastic,wai2018multi}, AUC maximization~\cite{guo2020communication,liu2019stochastic,ying2016stochastic}, 
imitation learning~\cite{cai2019global,nouiehed2019solving}, robust optimization ~\cite{duchi2019variance}, causal inference~\cite{meinshausen2018causality}, game theory~\cite{carmon2019variance,nash1953two} and so on.

We are interested in the minimax problems under PL conditions~\cite{nouiehed2019solving,yang2020global,doan2022convergence,yue2023lower}.
The PL condition~\cite{polyak1963gradient} was originally proposed to relax the strong convexity in minimization problem that is sufficient for achieving the global linear convergence rate for first-order methods. 
This condition has been successfully used to analyze the convergence behavior for overparameterized neural networks~\cite{liu2022loss}, robust phase retrieval~\cite{sun2018geometric}, and many other machine learning models~\cite{karimi2016linear}. 
There are many popular minimax formulations that only satisfy PL conditions but lack strong convexity (or strong concavity).
The examples include PL-game~\cite{nouiehed2019solving}, robust least square~\cite{yang2020global}, deep AUC maximization~\cite{liu2019stochastic} and generative adversarial imitation learning of LQR~\cite{cai2019global,nouiehed2019solving}. 

\citet{yang2020global} showed that the alternating gradient descent ascent
(AGDA) algorithm linearly converges to the saddle point when the objective function satisfies the two-sided PL condition. 
They also proposed the SVRG-AGDA method for the finite-sum problem (\ref{prob:main}), which could find $\epsilon$-optimal solution within $\fO\big((n+ n^{2/3} \kappa_x\kappa_y^2)\log(1/\epsilon)\big)$ stochastic first-order oracle (SFO) calls,\footnote{The original analysis ~\cite{yang2020global} provided an SFO upper bound $\fO\big((n+ n^{2/3} \max \{\kappa_x^3, \kappa_y^3 \})\log(1/\epsilon)\big)$, which can be refined to $\fO\big((n+ n^{2/3} \kappa_x\kappa_y^2)\log(1/\epsilon)\big)$ by some little modification in the proof.} where $\kappa_x$ and $\kappa_y$ are the condition numbers with respect to PL condition for $x$ and $y$ respectively. The variance reduced technique in the SVRG-AGDA leads to better a convergence rate than full batch AGDA whose SFO complexity is $\fO\big(n\kappa_x\kappa_y^2\log(1/\epsilon)\big)$.
However, there are still some open questions left. 
Firstly, \citet{yang2020global}'s theoretical analysis heavily relies on the alternating update rules. It remains interesting whether a simultaneous version of GDA (or its stochastic variants) also has similar convergence results. 
Secondly, it is unclear whether the SFO upper bound obtained by SVRG-AGDA can be improved by designing more efficient algorithms.

Under the one-sided PL condition, we desire to find the stationary point of $g(x)\triangleq \max_{y\in\BR^{d_y}}f(x,y)$, since the saddle point may not exist. \citet{nouiehed2019solving} proposed the Multi-Step GDA that achieves the $\epsilon$-stationary point within $\fO(\kappa_y^2 L \epsilon^{-2} \log(\kappa_y/\epsilon))$ number of full gradient iterations. A similar complexity is also proved by AGDA~\cite{yang2020global}.
Recently, \citet{yang2022faster} proposed the Smoothed-AGDA that improves the upper bound into $\fO(\kappa_y L \epsilon^{-2})$.
Both the Multi-Step GDA~\cite{nouiehed2019solving} and Smoothed-AGDA~\cite{yang2020catalyst} can be extended to the online setting, where we only assume the existence of an unbiased stochastic gradient with bounded variance.  But to the best of our knowledge, the formulation (\ref{prob:main}) with finite-sum structure has not been explored by prior works.

In this paper, we introduce a variance reduced first-order method, called SPIDER-GDA, which constructs the gradient estimator by stochastic recursive gradient~\cite{fang2018spider}, and the iterations are based on simultaneous gradient descent ascent~\cite{lin2020gradient}.
We prove that SPIDER-GDA could achieve $\epsilon$-optimal solution of the two-sided PL problem of the form (\ref{prob:main}) within $\fO\big((n+ \sqrt{n}\,\kappa_x\kappa_y^2)\log(1/\epsilon)\big)$ SFO calls, 
which has better dependency on $n$ than SVRG-AGDA~\cite{yang2020global}. We also provide an acceleration framework to improve first-order methods for solving ill-conditioned minimax problems under PL conditions. 
The accelerated SPIDER-GDA (AccSPIDER-GDA) could achieve $\epsilon$-optimal solution within $\fO\big((n+ \sqrt{n}\,\kappa_x\kappa_y)\log(\kappa_y/\epsilon) \log (1/\epsilon) \big)$ SFO calls when $\kappa_y \gtrsim \sqrt{n}$, which is the best known SFO upper bound for this problem.
We summarize our main results and compare them with related work in Table~\ref{tbl: two-side-PL}. Without loss of generality, we always suppose $\kappa_x \gtrsim \kappa_y$.
Furthermore, the proposed algorithms also work for minimax problem with one-sided PL condition. We present the results for this case in Table~\ref{tbl: one-side-PL}.

\begin{table}[t] 
\centering
\caption{We present the comparison of SFO complexities under two-sided PL condition. Note that \citet{yang2020global} named their stochastic algorithm as variance-reduced-AGDA (VR-AGDA). Here we call it SVRG-AGDA to distinguish with other variance reduced algorithms.}
\label{tbl: two-side-PL}
\begin{tabular}{ccc} 
\hline
Algorithm & Complexity & Reference \\
\hline\addlinespace
GDA/AGDA & $\fO\left(n\kappa_x \kappa_y^2 \log\left( 1/\epsilon\right)\right)$ & Theorem \ref{thm: GDA}, \cite{yang2020global} \\\addlinespace
SVRG-AGDA & $\fO\left( (n+ n^{2/3} \kappa_x \kappa_y^2) \log \left( 1/\epsilon\right)\right)$ &  \cite{yang2020global} \\\addlinespace
SVRG-GDA & $\fO\left( (n+ n^{2/3} \kappa_x \kappa_y^2) \log \left( 1/\epsilon\right)\right)$ & Theorem \ref{thm: SVRG-GDA} \\\addlinespace
SPIDER-GDA & $\fO\left( (n+ \sqrt{n} \kappa_x \kappa_y^2) \log \left( 1/\epsilon\right)\right)$ & Theorem \ref{thm: SPIDER-GDA}  \\\addlinespace
AccSPIDER-GDA & 
$
{\begin{cases}
\fO\left(\sqrt{n} \kappa_x \kappa_y \log \left(\kappa_y/\epsilon\right) \log (1/\epsilon)\right), &  \sqrt{n} \lesssim \kappa_y; \\[0.1cm]
\fO\left(n \kappa_x \log \left(\kappa_y/\epsilon\right) \log (1/\epsilon) \right), & \kappa_y \lesssim \sqrt{n} \lesssim \kappa_x \kappa_y; \\[0.1cm]
\fO\left((n+  \sqrt{n} \kappa_x \kappa_y^2) \log \left(1/\epsilon\right)\right), & \kappa_x \kappa_y \lesssim \sqrt{n}. 
\end{cases}}
$
& Theorem \ref{thm: Catalyst-GDA} \\\addlinespace
\hline
\end{tabular}
\end{table}

\begin{table}[t]
\centering
\caption{We present the comparison of SFO complexities under one-sided PL condition.} \label{tbl: one-side-PL}
\begin{tabular}{ccc} 
\hline
Algorithm & Complexity & Reference \\
\hline\addlinespace
Multi-Step GDA & $\fO(n \kappa_y^2 L \epsilon^{-2} \log(\kappa_y/\epsilon))$ & \cite{nouiehed2019solving} \\ \addlinespace
GDA/AGDA & $\fO\left(n\kappa_y^2L\epsilon^{-2}\right) $ & Theorem \ref{thm: GDA-one-side}, \cite{yang2020global} \\\addlinespace
Smooothed-AGDA & $\fO\left(n \kappa_y L\epsilon^{-2}\right)$ & \cite{yang2022faster} \\ \addlinespace 
SVRG-GDA & $\fO\left(n+  n^{2/3}\kappa_y^2L\epsilon^{-2} \right) $ & Theorem \ref{thm: SVRG-GDA-one}  \\\addlinespace
SPIDER-GDA & $ \fO\left(n+ \sqrt{n} \kappa_y^2L \epsilon^{-2} \right)$ & Theorem \ref{thm: SPIDER-GDA-one}  \\\addlinespace
AccSPIDER-GDA & 
$
\begin{cases}
\fO\left(\sqrt{n} \kappa_y L \epsilon^{-2} \log(\kappa_y/\epsilon)\right), & \sqrt{n} \lesssim \kappa_y; \\[0.15cm]
\fO\left(n L\epsilon^{-2} \log(\kappa_y/\epsilon)\right), & \kappa_y \lesssim \sqrt{n} \lesssim \kappa_y^2; \\[0.15cm]
\fO\left(n+\sqrt{n}\kappa_y^2L\epsilon^{-2}\right), & \kappa_y^2  \lesssim \sqrt{n}. \end{cases}
$
& Theorem \ref{thm: Catalyst-GDA-one-side} \\\addlinespace
\hline
\end{tabular}
\end{table}

\section{Related Work}

The minimax optimization problem (\ref{prob:main}) can be viewed as the following minimization problem
\begin{align*}
    \min_{x\in\BR^{d_x}} \bigg\{g(x)\triangleq \max_{y\in\BR^{d_y}}f(x,y)\bigg\}.
\end{align*}
A natural way to solve this problem is the Multi-step GDA~\cite{lin2020gradient,nouiehed2019solving,rafique2018non,luo2020stochastic} that contains double-loop iterations in which the outer loop can be regarded as running inexact gradient descent on $g(x)$ and the inner loop finds the approximate solution to $\max_{y\in\BR^{d_y}} f(x,y)$ for a given $x$. Another class of methods is the two-timescale (alternating) GDA algorithm~\cite{lin2020gradient,yang2020global,xian2021faster,doan2022convergence} that only has single-loop iterations which updates two variables with different stepsizes. The two-timescale GDA can be implemented more easily and typically performs better than Multi-Step GDA empirically~\cite{lin2020gradient}.
Its convergence rate also can be established by analyzing function $g(x)$ but the analysis is more challenging than the Multi-Step GDA.  

The variance reduction is a popular technique to improve the efficiency of stochastic optimization algorithms~\cite{schmidt2017minimizing,defazio2014saga,allen2016improved,johnson2013accelerating,shalev2013stochastic,mairal2015incremental,allen2018katyusha,fang2018spider,JMLR:v21:19-248,wang2019spiderboost,reddi2016stochastic,allen2016variance,palaniappan2016stochastic,chavdarova2019reducing,zhou2018stochastic,zhang2013linear,huang2022accelerated,nguyen2017sarah,li2021page}. 
It is shown that solving nonconvex minimization problems with stochastic recursive gradient estimator~\cite{fang2018spider,JMLR:v21:19-248,wang2019spiderboost,zhou2019faster,huang2022accelerated} has the optimal SFO complexity.  
In the context of minimax optimization, the variance reduced algorithms also obtain the best-known SFO complexities in several settings~\cite{luo2021near,han2021lower,alacaoglu2021stochastic,luo2020stochastic,yang2020global,vladislav2021accelerated}. Specifically, the (near) optimal SFO algorithm for several convex-concave minimax problem has been proposed~\cite{luo2021near,han2021lower}, but the optimality for the more general case is still unclear~\cite{luo2020stochastic,yang2020global}. 

The Catalyst acceleration~\cite{lin2015universal} is a useful approach to reduce the computational cost of ill-conditioned optimization problems, which is based on a sequence of inexact proximal point iterations. 
\citet{lin2020near} first introduced Catalyst into minimax optimization. Later, \citet{yang2020catalyst,luo2021near,vladislav2021accelerated} designed the accelerated stochastic algorithms for convex-concave and nonconvex-concave problems. Concurrently with our work, \citet{yang2022faster} also applied this technique to the one-sided PL setting.

\section{Notation and Preliminaries} \label{sec: set-up}

First of all, we present the definition of saddle point.
\begin{dfn}
We say $(x^{\ast},y^{\ast})\in\BR^{d_x}\times\BR^{d_y}$ is a saddle point of function $f:\BR^{d_x}\times\BR^{d_y}\to\BR$ if it holds that
$f(x^*,y) \leq f(x^*,y^*) \leq f(x,y^*)$
for any $x\in\BR^{d_x}$ and $y\in\BR^{d_y}$.
\end{dfn}

Then we formally define the Polyak--{\L}ojasiewicz (PL) condition~\cite{polyak1963gradient} as follows. 

\begin{dfn}\label{asm: PL}
We say a differentiable function $h:\BR^d\to\BR$ satisfies $\mu$-PL for some $\mu>0$ if
$\Vert \nabla h(z) \Vert^2 \ge 2 \mu \big(h(z) - \min_{z'\in\BR^d} h(z')\big)$
holds for any $z\in\BR^d$.
\end{dfn}

Note that the PL condition does not require strong convexity and it can be satisfied even if the function is nonconvex. For instance, the function $h(z) = z^2 + 3 \sin^2(z)$ ~\cite{karimi2016linear}.


We are interested in the finite-sum minimax optimization problem (\ref{prob:main}) under following assumptions.

\begin{asm} \label{asm: L-smooth}
We suppose each component $f_i:\BR^{d_x}\times\BR^{d_y}\to\BR$ is $L$-smooth, i.e., there exists a constant $L>0$ such that
$\Vert \nabla f_i(x,y) - \nabla f_i(x',y') \Vert^2 \le L^2\big(\Vert x - x' \Vert^2 + \Vert y - y' \Vert^2\big)$ holds for any $x,x' \in \BR^{d_x}$ and $y,y' \in \BR^{d_y}$.
\end{asm}

\begin{asm} \label{asm: two-PL}
We suppose the differentiable function $f:\BR^{d_x}\times\BR^{d_y} \rightarrow \BR$ satisfies two-sided PL condition, i.e., there exist constants $\mu_x>0$ and $\mu_y>0$ such that $f(\cdot, y)$ is $\mu_x$-PL for any $y\in\BR^{d_y}$ and $-f(x,\cdot)$ is $\mu_y$-PL for any $x\in\BR^{d_x}$.
\end{asm}


Under Assumption \ref{asm: L-smooth} and \ref{asm: two-PL}, we define the condition numbers of problem (\ref{prob:main}) with respect to PL conditions for $x$ and $y$ as $\kappa_x \triangleq {L}/{\mu_x}$ and $\kappa_y \triangleq {L}/{\mu_y}$ 
respectively. 

We also introduce the following assumption for the existence of saddle points.

\begin{asm}[\citet{yang2020global}] \label{asm: exist}
We suppose the function $f:\BR^{d_x}\times\BR^{d_y} \rightarrow \BR$ has at least one saddle point $(x^*,y^*)$. We also suppose that for any fixed $y\in\BR^{d_y}$, the problem $\min_{x \in \BR^{d_x}} f(x, y)$ has a nonempty solution set and a finite optimal value; and for any fixed $x\in\BR^{d_x}$, the problem $\max_{y \in \BR^{d_y}} f(x, y)$ has a nonempty solution set and a finite optimal value.
\end{asm}

The goal of solving minimax optimization under the two-sided PL condition is finding an $\epsilon$-optimal solution or $\epsilon$-saddle point that is defined as follows.

\begin{dfn} \label{dfn: epsilon-appox}
We say $x$ is an $\epsilon$-optimal solution of problem (\ref{prob:main}) if it holds that $g(x)-g(x^{\ast})\le \epsilon$, where $g(x)=\max_{y\in\BR^{d_y}}f(x,y)$.
\end{dfn}

We do not assume the existence of saddle points for the problems under the one-sided PL condition. 
In such case, it is guaranteed that $g(x)\triangleq \max_{y\in\BR^{d_y}}f(x,y)$ is differentiable \cite[Lemma A.5]{nouiehed2019solving} and we target to find an $\epsilon$-stationary point of $g(x)$.


\begin{dfn} \label{dfn: epsilon-stationary}
If the function $g:\BR^{d_x}\to\BR$ is differentiable, we say $x$ is an $\epsilon$-stationary point of $g$ if it holds that $\Vert\nabla g(x)\Vert\leq\epsilon$.
\end{dfn}


\section{A Faster Algorithm for the Two-Sided PL Condition} \label{sec: SPIDER-for-PL}

We first consider the two-sided PL conditioned minimax problem of the finite-sum form (\ref{prob:main}) under Assumption~\ref{asm: L-smooth}, \ref{asm: two-PL} and \ref{asm: exist}.
We propose a novel stochastic algorithm, which we refer to as SPIDER-GDA. The detailed procedure of our method is presented in Algorithm~\ref{alg: SPIDER-GDA}. SPIDER-GDA constructs the stochastic recursive gradient estimators~\cite{fang2018spider,nguyen2017sarah} as follows:
\begin{align*}
G_x(x_{t,k},y_{t,k}) =& \frac{1}{B } \sum_{i \in S_x} \big(\nabla_x f_i(x_{t,k},y_{t,k}) - \nabla_x f_i(x_{t,k-1}, y_{t,k-1})  +G_x(x_{t,k-1},y_{t,k-1})\big), \\
G_y(x_{t,k},y_{t,k}) =& \frac{1}{B} \sum_{i \in S_y} \big(\nabla_y f_i(x_{t,k},y_{t,k}) - \nabla_y f_i(x_{t,k-1}, y_{t,k-1}) +G_y(x_{t,k-1},y_{t,k-1})\big).
\end{align*}
It simultaneously updates two variables $\vx$ and $\vy$ by estimators $G_x$ and $G_y$ with different stepsizes $\tau_x=\Theta(1/(\kappa_y^2L))$ and $\tau_y=\Theta(1/L)$ respectively. \citet{luo2020stochastic,xian2021faster} have studied the SPIDER-type algorithm for nonconvex-strongly-concave problem and showed it converges to the stationary point of $g(x)\triangleq \max_{y\in\BR^{d_y}}f(x,y)$ sublinearly. However, solving the problem minimax problems under the two-sided PL condition desires a stronger linear convergence rate, which leads to our theoretical analysis being different from previous works.


We measure the convergence of SPIDER-GDA by the following Lyapunov function
\begin{align*}
\fV(x,y) \triangleq g(x) - g(x^{\ast}) + \frac{\lambda \tau_x}{\tau_y}\big(g(x) - f(x,y) \big),
\end{align*}
where $x^*\in\argmin_{x\in\BR^{d_x}}g(x)$ and $\lambda=\Theta(\kappa_y^2)$.
In Lemma \ref{lem: key-lem-for-Spider} we establish recursion for $\fV_{t,k}$ as
\begin{align*}
    \BE[\fV (x_{t,K}, y_{t,K})] \le \mathbb{E}\left[ \mathcal{V} (x_{t,0},y_{t,0})   - \frac{\tau_x}{2} \sum_{k=0}^{K-1} \Vert \nabla g(x_k) \Vert^2 - \frac{\lambda \tau_x}{4} \sum_{k=0}^{K-1} \Vert \nabla_y f(x_k,y_k) \Vert^2\right]
\end{align*}
by setting $M =B = \sqrt{n}$. Using the facts that $g(\,\cdot\,)$ is $\mu_x$-PL (Lemma \ref{lem: g(x)-PL}) and that $-f(x,\,\cdot\,)$ is $\mu_y$-PL (Assumption \ref{asm: two-PL}), we can show by setting $K = \Theta(\kappa_x \kappa_y^2)$, the value of $\tilde \fV(\tilde x_t,\tilde y_t)$ would shrink by $1/2$ each time the restart mechanism is triggered (Line 16 in Algorithm \ref{alg: SPIDER-GDA}).
Below, we formally provide the convergence result for SPIDER-GDA, and its detailed proof is shown in Appendix \ref{apx: SPIDER}.

\begin{thm} \label{thm: SPIDER-GDA} Under Assumption \ref{asm: L-smooth}, \ref{asm: two-PL} and \ref{asm: exist}, we run Algorithm \ref{alg: SPIDER-GDA} with $M= B = \sqrt{n}$ , $\tau_y = 1/(5L), \lambda = 32 L^2 / \mu_y^2$,  $\tau_x = \tau_y /  (24 \lambda) $, $ K = \lceil 2/ (\mu_x \tau_x) \rceil$ and $T = \lceil \log(1/\epsilon)\rceil$. Then the output $(\tilde x_T, \tilde y_T)$ satisfies $g(\tilde x_T) - g(x^{\ast}) \le \epsilon$ and $g(\tilde x_T) - f(\tilde x_T,\tilde y_T) \le 24 \epsilon$ in expectation; and it takes no more than $\fO\big((n + \sqrt{n} \kappa_x \kappa_y^2) \log(1/\epsilon)\big)$ SFO calls.
\end{thm}

Our results provide an SFO upper bound of $\fO( (n + \sqrt{n} \kappa_x \kappa_y^2) \log(1/\epsilon))$ for finding an $\eps$-optimal solution that is better than the complexity $\fO((n + n^{2/3}\kappa_x \kappa_y^2) \log(1/\epsilon))$ derived from SVRG-AGDA~\cite{yang2020global}.
It is possible to use SVRG-type~\cite{johnson2013accelerating,zhang2013linear} estimators to replace the stochastic recursive estimators in Algorithm~\ref{alg: SPIDER-GDA}, which yields the algorithm SVRG-GDA.
We can prove that SVRG-GDA also has $\fO((n + n^{2/3}\kappa_x \kappa_y^2) \log(1/\epsilon))$ SFO upper bound that matches the theoretical result of SVRG-AGDA.
We provide the details in Appendix \ref{apx: SVRG-GDA}.


\begin{algorithm*}[t]  
\caption{SPIDER-GDA $(f, (x_0,y_0), T,K,M, B, \tau_x, \tau_y)$} 
\begin{algorithmic}[1] \label{alg: SPIDER-GDA}
\STATE $\tilde x_0 = x_0, \tilde y_t = y_0 $\\[0.15cm]
\STATE \textbf{for} $t = 0,1,\dots, T-1$ \textbf{do} \\[0.15cm]
\STATE \quad $ x_{t,0} =  \tilde x_t, y_{t,0} = \tilde y_t$ \\[0.15cm]
\STATE \quad \textbf{for} $k = 0, 1, \dots, K-1$ \textbf{do}\\[0.15cm]
\STATE \quad \quad \textbf{if} $\mod(k,M) =0$ \textbf{then} \\[0.15cm]
\STATE \quad \quad \quad $G_x(x_{t,k},y_{t,k}) = \nabla_x f(x_{t,k},y_{t,k})$ \\[0.15cm]
\STATE \quad \quad \quad $G_y(x_{t,k},y_{t,k}) = \nabla_y f(x_{t,k},y_{t,k})$ \\[0.15cm]
\STATE \quad \quad \textbf{else} \\[0.15cm]
\STATE \quad \quad \quad Draw mini-batches $S_x$ and $S_y$ independently with both sizes of $B$. \\[0.15cm]
\STATE \quad \quad \quad {\small$G_x(x_{t,k},y_{t,k}) =\dfrac{1}{B } \sum_{i \in S_x} [\nabla_x f_i(x_{t,k},y_{t,k})- \nabla_x f_i(x_{t,k-1} , y_{t,k-1} )  +G_x(x_{t,k-1},y_{t,k-1})]$}\\[0.15cm]
\STATE \quad \quad \quad {\small$G_y(x_{t,k},y_{t,k}) = \dfrac{1}{B} \sum_{i \in S_y} [\nabla_y f_i(x_{t,k},y_{t,k})- \nabla_y f_i(x_{t,k-1} , y_{t,k-1} ) +G_y(x_{t,k-1},y_{t,k-1})]$} \\[0.15cm]
\STATE \quad \quad \textbf{end if} \\[0.15cm]
\STATE \quad \quad $x_{t,k+1} = x_{t,k} - \tau_x G_x(x_{t,k},y_{t,k})$ \\[0.15cm]
\STATE \quad \quad $y_{t,k+1} = x_{y,k} + \tau_y G_y(x_{t,k},y_{t,k})$ \\[0.15cm]
\STATE \quad  \textbf{end for} \\[0.15cm]
\STATE \quad Choose $(\tilde x_{t+1},\tilde y_{t+1})$ from $\{  (x_{t,k}, y_{t,k} )\}_{k=0}^{K-1} $ uniformly at random. \\ [0.15cm]
\STATE \textbf{end for}  \\[0.15cm]
\STATE \textbf{return} $(\tilde x_T, \tilde y_T)$ 
\end{algorithmic}
\end{algorithm*}

\section{Further Acceleration with Catalyst} \label{sec: Catalyst-for-acc}

\begin{algorithm*}[t]  
\caption{AccSPIDER-GDA} 
\begin{algorithmic}[1] \label{alg: Catalyst-GDA}
    \STATE $u_ 0  = x_0$ \\[0.15cm]
    \STATE \textbf{for} $k = 0, 1, \dots, K-1$ \textbf{do}\\[0.15cm]
    \STATE \quad $(x_{k+1}, y_{k+1}) = \text{SPIDER-GDA}\left( f(x,y) + \dfrac{\beta}{2} \Vert x - u_k \Vert^2,(x_k,y_k),T_k,K,M,B,\tau_x,\tau_y\right)$ \\[0.15cm]
    \STATE $\quad u_{k+1} = x_{k+1} + \gamma(x_{k+1} - x_k)$ \\ [0.15cm]
    \STATE\textbf{end for} \\[0.15cm]
    \STATE \textbf{Option I}  (two-sided PL): \textbf{return} $(x_K,y_K)$  \\[0.15cm]
    \STATE \textbf{Option II} (one-sided PL): \textbf{return} $(\hat x,\hat y)$ chosen uniformly at random from $\{(x_k,y_k)\}_{k=0}^{K-1}$
\end{algorithmic}
\end{algorithm*}

Both the proposed SPIDER-GDA (Algorithm~\ref{alg: SPIDER-GDA}) and existing SVRG-AGDA~\cite{yang2020global} have complexities that more heavily depend on the condition number of $y$ than the condition number of $x$. It is natural to ask whether we can make the dependency of two condition numbers balanced like the results in the strongly-convex-strongly-concave case~\cite{lin2020near,luo2020stochastic,vladislav2021accelerated}. In this section, we show that it is possible by introducing the Catalyst acceleration.



We proposed the accelerated SPIDER-GDA (AccSPIDER-GDA) in Algorithm \ref{alg: Catalyst-GDA} for reducing the computational cost further.
Each iteration of the algorithm solves the following sub-problem by SPIDER-GDA (Algorithm \ref{alg: SPIDER-GDA}):
\begin{align}\label{prob:sub}
    \min_{x \in \BR^{d_x}} \max_{y \in \BR^{d_y}} F_k(x,y) \triangleq \min_{x \in \BR^{d_x}} \left\{g(x) + \frac{\beta}{2} \Vert x - u_k \Vert_2^2\right\}.
\end{align}
The resulting algorithm AccSPIDER-GDA has the following convergence result if the sub-problem can attain the required accuracy.

\begin{lem}\label{lem: outer-converge}
Under Assumption  \ref{asm: L-smooth}, \ref{asm: two-PL} and \ref{asm: exist}, we run Algorithm~\ref{alg: Catalyst-GDA} by $\beta = 2L$, $\gamma = 0$ and the appropriate setting for the sub-problem solver such that, for all $k \ge 1$, it finds $(x_{k+1},y_{k+1})$ satisfying
\begin{align} \label{dfn:delta-two}
\E\left[\max_{y \in \BR^{d_y}} F_k(x_{k+1}, y) - \min_{x \in \BR^{d_x}} F_k(x, y_{k+1})\right] \le 
   \frac{\mu_y^2 \epsilon}{12L^2} \triangleq \delta,
\end{align}
where $\delta>0$ denotes the required precision of the sub-problem.
Then it holds that
\begin{align*}
    \BE[g(x_k)  - g(x^{\ast})]
    &\le \left(1-\frac{\mu_x}{4\beta+ \mu_x} \right)^k \big(g(x_0) - g(x^{\ast})\big) + \frac{\epsilon}{2}.
\end{align*}
\end{lem}
The setting $\beta = \Theta(L)$ in Lemma \ref{lem: outer-converge} guarantees the sub-problem (\ref{prob:sub}) has condition number of the order $\fO(1)$ for $x$. It is more well-conditioned on $x$, we prefer to address the following equivalent problem
\begin{align} \label{sub-minimax}
    \max_{y \in \BR^{d_y}} \min_{ x \in \BR^{d_x}} F_k(x,y) = - \min_{y \in \BR^{d_y}} \max_{ x \in \BR^{d_x}} \left\{ -F_k(x,y)\right\}.
\end{align}

The above strong duality under PL conditions are shown in appendix.
Since (\ref{sub-minimax}) is a minimax problem satisfying the two-sided PL condition, we can apply SPIDER-GDA to solve it. 


\begin{lem} \label{sub:eps-saddle}
Under Assumption  \ref{asm: L-smooth}, \ref{asm: two-PL} and \ref{asm: exist}, if we use Algorithm \ref{alg: SPIDER-GDA} to solve each sub-problem $\max_{y \in \BR^{d_y}} \min_{ x \in \BR^{d_x}} F_k(x,y) $ (\ref{prob:sub}) with $\beta = 2L$,
$M = B = \sqrt{n}$, $\tau_x = 1/(15L)$, $\lambda = 288$, $\tau_y = \tau_x / (24 \lambda)$, $K = \lceil 2 / (\mu_y \tau_y) \rceil$, $T_k = \lceil \log(1/\delta_k) \rceil$, then for all $k \ge 1$ it holds that
\begin{align*}
\BE\left[\max_{y \in \BR^{d_y}} F_k(x_{k+1}, y) - \min_{x \in \BR^{d_x}} F_k(x, y_{k+1})\right]
\le 200 \kappa_y^2 \delta_k \BE \left[\max_{y \in \BR^{d_y}} F_k(x_{k}, y) - \min_{x \in \BR^{d_x}} F_k(x, y_{k})\right].  
\end{align*}
\end{lem}


For a short summary, Lemma \ref{lem: outer-converge} means Algorithm \ref{alg: Catalyst-GDA} requires $\fO(\kappa_x \log(1/\epsilon))$ numbers of inexact proximal point iterations to find an $\epsilon$-optimal solution of the problem. And Lemma \ref{lem: sub-probelm} tells us that each sub-problem can be solved within an SFO complexity of $\fO\left(n + \sqrt{n} \kappa_y) \log(1/\delta_k)\right)$. Thus, the total complexity for AccSPIDER-GDA becomes $\fO( (n \kappa_x + \sqrt{n} \kappa_x \kappa_y ) \log(1/\epsilon) \log(1/\delta_k))$. Our next step is to specify $\delta_k$, which would lead to the total SFO complexity of the algorithm.

\begin{thm} \label{thm: Catalyst-GDA}
Under Assumption  \ref{asm: L-smooth}, \ref{asm: two-PL}, and \ref{asm: exist}, if we  let $\gamma = 0, \beta = 2L$ and use Algorithm \ref{alg: SPIDER-GDA} to solve each sub-problem $\max_{y \in \BR^{d_y}} \min_{ x \in \BR^{d_x}} F_k(x,y)$ (\ref{prob:sub}) with
$M, B, \tau_x, \tau_y,K $ defined as Lemma~\ref{sub:eps-saddle} and $T_k = \lceil \log(1/\delta_k) \rceil$, where 
\begin{align} \label{dfn:delta_k}
    \delta_k = 
    \begin{cases}
    \dfrac{1}{400 \kappa_y^2 } \min \left\{ \dfrac{1}{24}, \dfrac{\mu_y^2 \delta}{64 L^2 \Vert x_k - x_{k-1} \Vert^2}  \right \}, & k \ge 1;\\[0.25cm]
    \dfrac{\delta}{2400 \kappa_y^4 ( g(x_0) - g(x^*) + g(x_0) - f(x_0,y_0))}, & k =0,
    \end{cases}
\end{align}
and $\delta$ is followed by the definition in (\ref{dfn:delta-two}). Then Algorithm \ref{alg: Catalyst-GDA} can return $x_K$ such that $g(x_K) - g(x^{\ast}) \le \epsilon$ in expectation with no more than $\fO( (n \kappa_x + \sqrt{n} \kappa_x \kappa_y) \log(1/\epsilon) \log(\kappa_y /\epsilon))$ SFO calls.
\end{thm}

Lemma \ref{lem: outer-converge} does not rely on the choice of sub-problem solver, we can apply the acceleration framework in Algorithm~\ref{alg: Catalyst-GDA} by replacing SPIDER-GDA with other algorithms. We summarize the SFO complexities for the acceleration of different algorithms in Table \ref{tbl: diff-acc}. 
\begin{table}[t]
\centering
\caption{Accelerated results for different methods under the two-sided PL condition. 
We use the notation $\tilde \fO(\cdot)$ to hide the logarithmic factors of $\kappa_y$ and $1/\epsilon$.}
{
\begin{tabular}{ccc}
\hline 
Method & Before Acceleration & After Acceleration   \\
\hline \addlinespace
GDA & $ \tilde \fO(n \kappa_x \kappa_y^2 )$  & $\tilde \fO\left(n \kappa_x \kappa_y\right)$\\[0.3cm]
SVRG-GDA & $\tilde \fO((n + n^{2/3}\kappa_x \kappa_y^2))$  & 
$
\begin{cases}
\tilde \fO\left(n^{2/3} \kappa_x \kappa_y\right), & n^{1/3} \lesssim \kappa_y; \\[0.15cm]
\tilde \fO \left(n \kappa_x \right), & \kappa_y \lesssim n^{1/3} \lesssim \kappa_x \kappa_y; \\[0.15cm]
\text{no acceleration}, & \kappa_x \kappa_y \lesssim n^{1/3}.
\end{cases}
$
\\\addlinespace
SPIDER-GDA & $\tilde \fO\left((n + \sqrt{n} \kappa_x \kappa_y^2) \right)$ & $
\begin{cases}
\tilde \fO\left(\sqrt{n} \kappa_x \kappa_y \right), &  \sqrt{n} \lesssim \kappa_y; \\[0.15cm]
\tilde \fO \left( n \kappa_x \right), & \kappa_y \lesssim \sqrt{n} \lesssim \kappa_x \kappa_y; \\[0.15cm]
\text{no acceleration}, & \kappa_x \kappa_y \lesssim \sqrt{n}. 
\end{cases}
$  \\ \addlinespace
 \hline
\end{tabular}}
\label{tbl: diff-acc}
\end{table}

\section{Extension to One-Sided PL Condition} \label{sec: extension}
In this section, we show the idea that SPIDER-GDA and its  Catalyst acceleration also work for one-sided PL conditions. 
We relax Assumption \ref{asm: two-PL} and \ref{asm: exist} to the following one.

\begin{asm} \label{asm: one-PL}
We suppose that $-f(x,\cdot)$ is $\mu_y$-PL for any $x\in\BR^{d_x}$;
the problem $\max_{y\in\BR^{d_y}}f(x,y)$ has a nonempty solution set and an optimal value ; $g(x) \triangleq \max_{y\in\BR^{d_y}}f(x,y)$ is lower bounded, i.e., we have $g^{\ast} = \inf_{x \in \BR^{d_x}} g(x) > -\infty$.
\end{asm}

We first show that the SFO complexity of SPIDER-GDA outperforms SVRG-GDA \footnote{The complexity for finding an $\epsilon$-stationary point of SVRG-GDA is presented in Appendix \ref{apx: one-side}.}  by a factor of $\mathcal{O}(n^{1/6})$ in Theorem \ref{thm: SPIDER-GDA-one}. 
 
\begin{thm} \label{thm: SPIDER-GDA-one}
Under Assumption \ref{asm: L-smooth} and \ref{asm: one-PL} , Let $T = 1$ and $M,B,\tau_x,\tau_y,\lambda$ as defined in Theorem~\ref{thm: SPIDER-GDA} and $ K = \lceil 64 / (\tau_x \epsilon^2 )\rceil$, then  Algorithm \ref{alg: SPIDER-GDA} can guarantee the output $\hat x$ to satisfy  $\Vert \nabla g(\hat x) \Vert \le \epsilon$ in expectation with no more than $\fO(n+ { \sqrt{n} \kappa_y^2 L}{\epsilon^{-2}})$ SFO calls.
\end{thm}

The AccSPIDER-GDA also performs better than SPIDER-GDA under one-sided PL conditions for ill-conditioned problems. In the following lemma, we show that AccSPIDER-GDA could find an approximate stationary point if we solve the sub-problem sufficiently accurately.

\begin{lem} \label{lem: outer-convergence-one-side}
Under Assumption \ref{asm: L-smooth} and \ref{asm: one-PL}, if for all $k \ge 1$ the point $(x_{k+1},y_{k+1})$ satisfies
\begin{align} \label{dfn:delta-one}
\E\left[\max_{y \in \BR^{d_y}} F_k(x_{k+1}, y) - \min_{x \in \BR^{d_x}} F_k(x, y_{k+1})\right] \le \frac{\epsilon^2}{184 \kappa_y^2 L} \triangleq \delta.  
\end{align}
 Let $\beta=2L$, then for the output $(\hat x,\hat y)$ of Algorithm \ref{alg: Catalyst-GDA}, it holds true that
\begin{align*}
    \BE \Vert \nabla g(\hat x) \Vert^2
    \le  \frac{8 \beta(g(x_0) - g^{\ast})}{K} +\frac{\epsilon^2}{2}.
\end{align*}
\end{lem}

Compared with SPIDER-GDA, the analysis of AccSPIDER-GDA is more complicated since the precision $\delta_k$ at each round is different. By choosing the parameters of the algorithm carefully, we obtain the following result.

\begin{thm} \label{thm: Catalyst-GDA-one-side}
Under Assumption \ref{asm: L-smooth} and \ref{asm: one-PL}, if we run Algorithm~\ref{alg: Catalyst-GDA} by $\gamma = 0, \beta = 2L$  and use Algorithm \ref{alg: SPIDER-GDA} to solve each sub-problem $\max_{y \in \BR^{d_y}} \min_{ x \in \BR^{d_x}} F_k(x,y)$ (\ref{prob:sub}) with $M,B,\tau_x,\tau_y,\lambda,K$ and $T_k$ (dependent on $\delta$) as in Theorem \ref{thm: Catalyst-GDA}
and $\delta$ is followed by the definition in Lemma \ref{lem: outer-convergence-one-side}, then Algorithm \ref{alg: Catalyst-GDA} can find $\hat x$ such that $\Vert \nabla g(\hat x) \Vert \le \epsilon$ in expectation within $\fO( (n  + \sqrt{n} \kappa_y) L \epsilon^{-2} \log(\kappa_y/\epsilon) )$ SFO calls.
\end{thm}

We can directly set $\beta=0$ for Algorithm~\ref{alg: Catalyst-GDA} in the case of very large $n$, which makes AccSPIDER-GDA reduce to SPIDER-GDA. The summary and comparison of the complexities under the one-sided PL condition are both shown in Table \ref{tbl: one-side-PL}.
Besides, the algorithms of GDA and SVRG-GDA also can be accelerated with the Catalyst framework and we present the corresponding results in Table \ref{tbl: diff-acc-one}.

\begin{table}[t]
\centering
\caption{Acceleration for different methods under one-sided PL condition.
}
{\begin{tabular}{ccc}
\hline
Method & Before Acceleration & After Acceleration   \\
\hline \addlinespace
GDA & $\fO\left(n \kappa_y^2L\epsilon^{-2}\right)$  & $\fO\left(n \kappa_y L\epsilon^{-2} \log(\kappa_y/\epsilon)\right)$\\[0.2cm]
SVRG-GDA & $ \fO\left(n + n^{2/3} \kappa_y^2 L\epsilon^{-2} \right)$  & 
$
\begin{cases}
 \fO\left(n^{2/3} \kappa_y L\epsilon^{-2}  \log(\kappa_y/\epsilon) \right), &  n^{1/3} \lesssim \kappa_y; \\[0.15cm]
 \fO \left( n L\epsilon^{-2}  \log(\kappa_y/\epsilon) \right), & \kappa_y \lesssim n^{1/3} \lesssim \kappa_y^2; \\[0.15cm]
\text{no acceleration}, & \kappa_y^2 \lesssim n^{1/3}. 
\end{cases}
$
\\ \addlinespace
SPIDER-GDA & $\fO\left(n + {\sqrt{n} \kappa_y^2 L}{\epsilon^{-2}}\right)$ & $
\begin{cases}
\fO\left({\sqrt{n} \kappa_y L}{\epsilon^{-2}}  \log(\kappa_y/\epsilon) \right), & \sqrt{n} \lesssim \kappa_y; \\[0.15cm]
\fO \left( {n L}{\epsilon^{-2}}  \log(\kappa_y/\epsilon)\right), & \kappa_y \lesssim \sqrt{n} \lesssim \kappa_y^2; \\[0.15cm]
\text{no acceleration}, & \kappa_y^2  \lesssim \sqrt{n}. 
\end{cases}
$  \\ \addlinespace
 \hline
\end{tabular}}
\label{tbl: diff-acc-one}
\end{table}

\section{Experiments} \label{sec: exp}

In this section, we conduct the numerical experiments to show the advantage of proposed algorithms and the source code is available\footnote{~\url{https://github.com/TrueNobility303/SPIDER-GDA}}. We consider the following two-player Polyak--{\L}ojasiewicz game:
\begin{align*}
\min_{x\in\BR^{d}}\max_{y\in\BR^{d}} f(x,y) \triangleq \frac{1}{2} x^\top P x  - \frac{1}{2}y^\top Q  y + x^\top R  y,
\end{align*}
where
\begin{align*}
P = \frac{1}{n}\sum_{i=1}^n p_i p_i^\top, \quad
Q = \frac{1}{n}\sum_{i=1}^n q_i q_i^\top \quad \text{and} \quad
R = \frac{1}{n}\sum_{i=1}^n r_i r_i^\top.
\end{align*}
We independently sample $p_i$, $q_i$ and $r_i$ from $\fN(0, \Sigma_P) $, $\fN(0, \Sigma_Q)$ and $\fN(0, \Sigma_R)$ respectively. We set the covariance matrix $\Sigma_P$ as the form of $U D U^\top$ such that $U\in\BR^{d\times r}$ is column orthogonal matrix and $D\in\BR^{r\times r}$ is diagonal with $r<d$. The diagonal elements of $D$ are distributed uniformly in the interval $[\mu, L]$ with $0<\mu<L$. The matrix $\Sigma_Q$ is set in a similar way to $\Sigma_P$. We also let $\Sigma_R = 0.1 V V^\top$, where each element of $V\in\BR^{d\times d}$ is sampled from $\fN(0, 1)$ independently. Since the covariance matrices $\Sigma_P$ and $\Sigma_Q$ are rank-deficient, it is guaranteed that both $P$ and $Q$ are singular. Hence, the objective function is not strongly-convex nor strongly-concave, but it satisfies the two-sided PL-condition~\cite{karimi2016linear}. We set $n = 6000, d= 10$, $r=5$, $L=1$ for all experiments; and let $\mu$ be $10^{-5}$ and $10^{-9}$ for two different settings.

We compare the proposed SPIDER-GDA (Algorithm \ref{alg: SPIDER-GDA}) and AccSPIDER-GDA (Algorithm \ref{alg: Catalyst-GDA}) with the baseline algorithm SVRG-AGDA~\cite{yang2020global}. We let $B=1$ and $M=n$ for all of these algorithms and both of the stepsizes for $x$ and $y$ are tuned from $\{10^{-1}, 10^{-2}, 10^{-3}, 10^{-4}, 10^{-5}\}$. For AccSPIDER, we set $\beta = L/(20n)$ and $\gamma = 0.999$. 
We present the results of the number of SFO calls against the norm of gradient and the distance to the saddle point in Figure \ref{fig: dataset-1} and Figure \ref{fig: dataset-2}. It is clear that our algorithms outperform baselines.

\begin{figure}[htbp]\centering
\begin{tabular}{cc}
\includegraphics[scale=0.44]{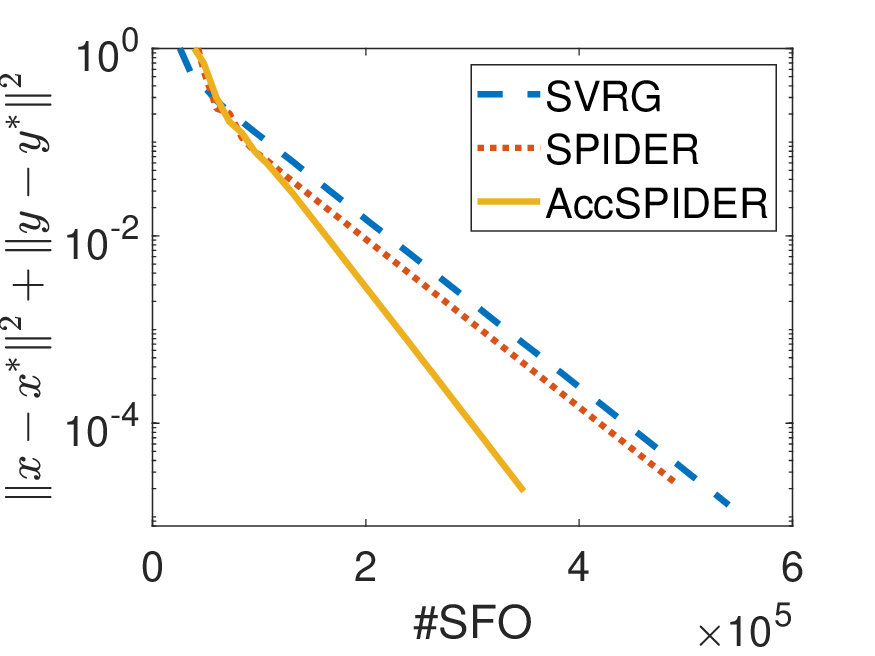} &
\includegraphics[scale=0.44]{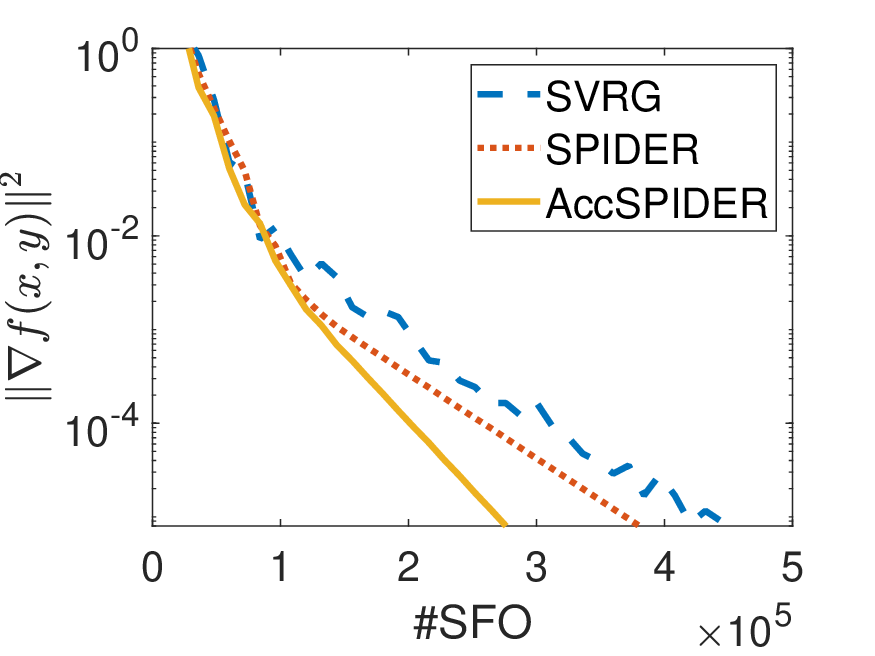} \\
(a) Distance to saddle point & 
(b) Norm of gradient  
\end{tabular}
\caption{The comparison for the case of $\mu = 10^{-5}$}
    \label{fig: dataset-1}
\end{figure}

\begin{figure}[htbp]\centering
\begin{tabular}{cc}
\includegraphics[scale=0.44]{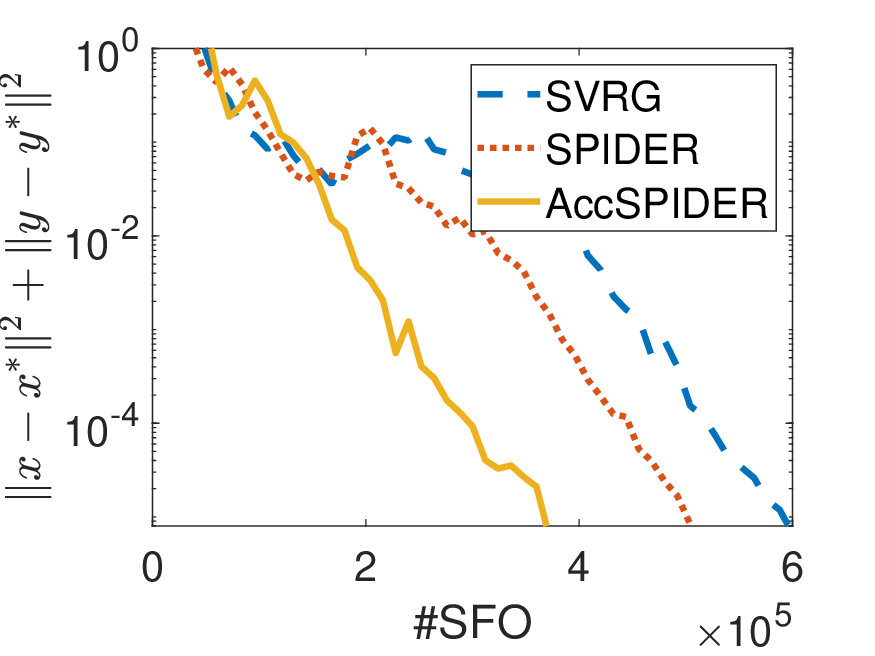} &
\includegraphics[scale=0.44]{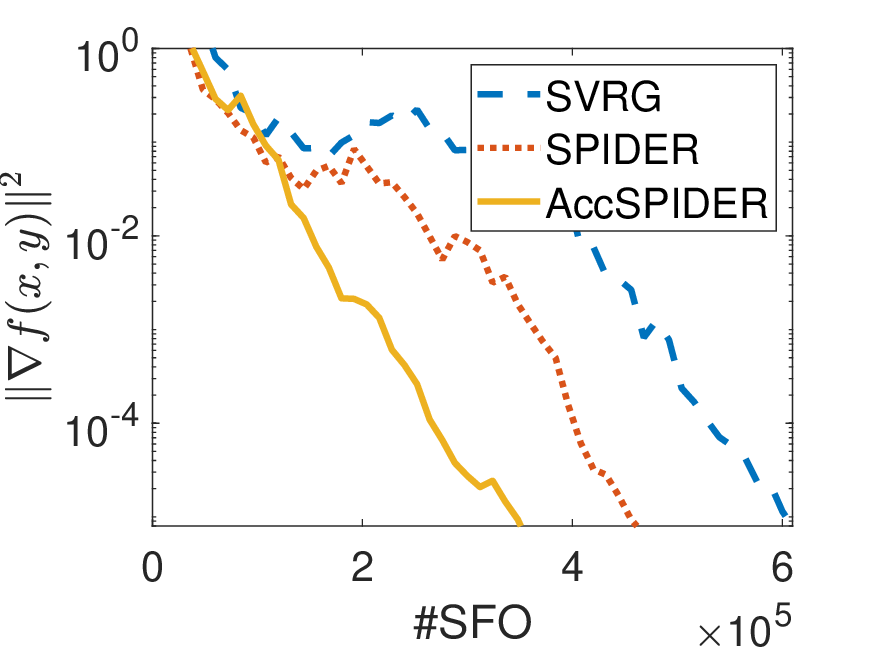} \\
(a) Distance to saddle point & 
(b) Norm of gradient  
\end{tabular}
\caption{The comparison for the case of $\mu = 10^{-9}$}
\label{fig: dataset-2}
\end{figure}

\section{Conclusion and Future Work} \label{sec: discucssion}

In this paper, we have investigated stochastic optimization for PL conditioned minimax problem with the finite-sum objective. 
We have proposed the SPIDER-GDA algorithm, which reduces the dependency of the sample numbers in SFO complexity. Moreover, we have introduced a Catalyst scheme to accelerate our algorithm for solving ill-conditioned problems. 
We improve the SFO upper bound of the state-of-the-art algorithms for both two-sided and one-sided PL conditions.

However, the optimality of SFO algorithms for the PL conditioned minimax problem is still unclear. It is interesting to construct the lower bound for verifying the tightness of our results. It is also possible to extend our algorithm to the online setting.

\section*{Acknowledgements}
The authors would like to thank Yunyan Bai for pointing out some mistakes in the proof of Theorem~\ref{thm: SPIDER-GDA}.
This work is supported by the National Natural Science Foundation of China (No. 62206058) and the Shanghai Sailing Program (22YF1402900).

\bibliographystyle{plainnat}
\bibliography{reference}

\begin{enumerate}

\item For all authors...
\begin{enumerate}
  \item Do the main claims made in the abstract and introduction accurately reflect the paper's contributions and scope?
    \answerYes{}
  \item Did you describe the limitations of your work?
    \answerYes{See Section \ref{sec: discucssion}.}
  \item Did you discuss any potential negative societal impacts of your work?
    \answerNo{The purpose of this work is for to provide a better understanding of GDA on a class of nonconvex-nonconcave minimax optimization.} 
  \item Have you read the ethics review guidelines and ensured that your paper conforms to them?
    \answerYes{}
\end{enumerate}

\item If you are including theoretical results...
\begin{enumerate}
  \item Did you state the full set of assumptions of all theoretical results?
    \answerYes{See Section \ref{sec: set-up}.}
        \item Did you include complete proofs of all theoretical results?
    \answerYes{See Appendix for details.}
\end{enumerate}

\item If you ran experiments...
\begin{enumerate}
  \item Did you include the code, data, and instructions needed to reproduce the main experimental results (either in the supplemental material or as a URL)?
    \answerYes{We include the codes In the supplemental materials.}
  \item Did you specify all the training details (e.g., data splits, hyperparameters, how they were chosen)?
    \answerYes{See Section \ref{sec: exp}.}
\item Did you report error bars (e.g., with respect to the random seed after running experiments multiple times)?
    \answerNo{We wants to compare the training dynamic, and different trials may cause different numerical results, which can not be observed clearly in one graph.}
    
    \item Did you include the total amount of compute and the type of resources used (e.g., type of GPUs, internal cluster, or cloud provider)?
    \answerNo{The experiments is certainly simple
and easy to run under CPUs.}

\end{enumerate}

\item If you are using existing assets (e.g., code, data, models) or curating/releasing new assets...
\begin{enumerate}
  \item If your work uses existing assets, did you cite the creators?
    \answerYes{}
  \item Did you mention the license of the assets?
    \answerYes{}
  \item Did you include any new assets either in the supplemental material or as a URL?
    \answerYes{}
  \item Did you discuss whether and how consent was obtained from people whose data you're using/curating?
    \answerNo{These datasets are common.}
  \item Did you discuss whether the data you are using/curating contains personally identifiable information or offensive content?
    \answerNo{These datasets are common.}
\end{enumerate}

\item If you used crowdsourcing or conducted research with human subjects...
\begin{enumerate}
  \item Did you include the full text of instructions given to participants and screenshots, if applicable?
    \answerNA{}
  \item Did you describe any potential participant risks, with links to Institutional Review Board (IRB) approvals, if applicable?
    \answerNA{}
  \item Did you include the estimated hourly wage paid to participants and the total amount spent on participant compensation?
    \answerNA{}
\end{enumerate}

\end{enumerate}


\newpage
\appendix

\input{appendix.tex}

\end{document}

%% file: math.tex
\usepackage{amsmath}\allowdisplaybreaks
\usepackage{amsfonts,bm}
\usepackage{amssymb}
\usepackage{algorithm}
\usepackage{algorithmic}

\def\eqref#1{equation~(\ref{#1})}

\def\1{\bf{1}}

\newcommand{\norm}[1]{\left\| #1 \right\|_2}

\def\eps{{\varepsilon}}

\def\vx{{\bf{x}}}
\def\vy{{\bf{y}}}

\def\fA{{\mathcal{A}}}
\def\fB{{\mathcal{B}}}

\def\fN{{\mathcal{N}}}
\def\fO{{\mathcal{O}}}

\def\fV{{\mathcal{V}}}

\def\BE{{\mathbb{E}}}

\def\BR{{\mathbb{R}}}

\newcommand{\E}{\mathbb{E}}

\DeclareMathOperator*{\argmax}{arg\,max}
\DeclareMathOperator*{\argmin}{arg\,min}

\usepackage{amsthm}
\theoremstyle{plain}
\newtheorem{thm}{Theorem}[section]
\newtheorem{dfn}{Definition}[section]

\newtheorem{lem}{Lemma}[section]
\newtheorem{asm}{Assumption}[section]

\newtheorem{cor}{Corollary}[section]

\makeatletter
\def\Ddots{\mathinner{\mkern1mu\raise\p@
\vbox{\kern7\p@\hbox{.}}\mkern2mu
\raise4\p@\hbox{.}\mkern2mu\raise7\p@\hbox{.}\mkern1mu}}
\makeatother

\makeatletter
\newcommand*{\rom}[1]{\expandafter\@slowromancap\romannumeral #1@}
\makeatother

%% file: appendix.tex
\section{Some Useful Facts}

In this section, we provide some facts which are useful in the following proofs. 

First of all, we define three notations of optimality. 
\begin{dfn}  \label{dfn: three-points} 
 We say $(x^{\ast},y^{\ast})$ is a saddle point of function $f$, if for all $(x,y)$, it holds that
\begin{align*}
    f(x^{\ast}, y) \le f(x^{\ast},y^{\ast}) \le f(x,y^{\ast}).
\end{align*}
We say $(x^{\ast},y^{\ast})$ is a global minimax point, if for all $x \in \BR^{d_x}, y \in \BR^{d_y} $, it holds that 
\begin{align*}
    f(x^{\ast},y) \le f(x^{\ast}, y^{\ast}) \le \max_{y' \in \BR^{d_y}} f(x,y').
\end{align*}
And we say $(x^{\ast},y^{\ast})$ is a stationary point, if it holds that
\begin{align*}
    \nabla_x f(x^{\ast},y^{\ast}) = \nabla_y f(x^{\ast}, y^{\ast}) = 0.
\end{align*}
\end{dfn}

 For general nonconvex-nonconcave minimax problem, a stationary point or a global minimax point is weaker than a saddle point, i.e. a stationary point or a global minimax point may not  be a saddle point. However, under two-sided PL condition, the above three notations are equivalent.

\begin{lem}[{\citet[Lemma 2.1]{yang2020global}}] \label{lem: three-points-equal}

Under Assumption \ref{asm: two-PL}, it holds that 
\begin{align*}
    (\text{saddle point} ) \Leftrightarrow (\text{global minimax point}) \Leftrightarrow (\text{stationary point} ).
\end{align*}

Further,  if $(x^{\ast}, y^{\ast}) $ is a saddle point of $f$, then 
\begin{align*}
    \max_{y \in \BR^{d_y}} f(x^{\ast}, y) = f(x^{\ast}, y^{\ast}) = \min_{x \in \BR^{d_x}} f(x,y^{\ast}).
\end{align*}
and vice versa.
\end{lem}

 It is well known that weak duality always holds.
 
\begin{lem}[{\citet[Theorem 1.3.1]{nesterov2018lectures}}] \label{lem: minimax}
Given a function $f$, we have 
\begin{align*}
    \max_{y \in \BR^{d_x}}  \min_{x \in \BR^{d_x }} f(x,y) \le  \min_{ x \in \BR^{d_x}} \max_{y \in \BR^{d_y}} f(x,y).
\end{align*}
\end{lem}

It is a standard conclusion that the existence of saddle points implies strong duality. Since strong duality is important for the convergence of Catalyst scheme under PL condition, we present this lemma as follows.

\begin{lem} \label{lem: saddle-point-eq}

If $(x^{\ast}, y^{\ast}) $ is a saddle point of function $f$, then $(x^{\ast}, y^{\ast}) $ is also a global minimax point and stationary point of $f$,  and it holds that
\begin{align*}
     \max_{y \in \BR^{d_y}} \min_{x \in \BR^{d_x}} f(x,y) = f(x^{\ast} ,y^{\ast}) = \min_{x \in \BR^{d_x}} \max_{y \in \BR^{d_y}}  f(x,y).
\end{align*}
\end{lem}

\begin{lem}[{\citet[Lemma A.1]{yang2020global}}] \label{lem: quadratic-growth}

Under Assumption \ref{asm: two-PL}, then $f(x,y)$ also satisfies the following quadratic growth condition, i.e. for all $x \in \BR^{d_x},y \in \BR^{d_y}$, it holds that
\begin{align*}
    f(x,y) - \min_{x \in \BR^{d_x}} f(x,y) &\ge \frac{\mu_x}{2} \Vert x^{\ast}(y) - x \Vert^2, \\
    \max_{y \in \BR^{d_y}} f(x,y) - f(x,y) &\ge \frac{\mu_y}{2} \Vert y^{\ast}(x) - y \Vert^2,
\end{align*}
where $x^{\ast}(y)$ is the projection of $x$ on the set $\argmin_{x \in \BR^{d_x}} f(x,y)$ and $y^{\ast}(x)$ is the projection of $y$ on the set of $\argmax_{y \in \BR^{d_y}} f(x,y)$. 
\end{lem}

Also, we analyze the properties of $g(x)$.
\begin{lem}[{\citet[Lemma 2.1]{yang2020global}}] \label{lem: g(x)-PL}
Under Assumption \ref{asm: two-PL}, then $g(x)$ is $\mu_x$-PL, i.e., for all $x$ we have
\begin{align*}
    \Vert \nabla g(x) \Vert^2 \ge 2 \mu_x (g(x) - g(x^{\ast})).
\end{align*}

\end{lem}

\begin{lem}[{\citet[in the proof of Theorem 3.1]{yang2020global}}] \label{lem: bound-Bk}
Under Assumption \ref{asm: two-PL} and \ref{asm: L-smooth}, then for all $x,y$ it holds true that 
\begin{align*}
    \Vert \nabla_x f(x,y) - \nabla g(x) \Vert^2 &\le \frac{2L^2}{\mu_y} (g(x) - f(x,y)).
\end{align*}
\end{lem}

The above lemma is a direct result of the quadratic growth property implied by PL condition and $L$-smooth property of function $f(x,y)$. Using the definition of $\mu_y$-PL in $y$, we can also show the relationship between $ \Vert \nabla_x f(x,y) - \nabla g(x) \Vert^2 $ and $ \Vert \nabla_y f(x,y) \Vert^2$ as follows.

\begin{lem} \label{lem: new-bound-Bk}

Under Assumption \ref{asm: two-PL} and \ref{asm: L-smooth}, then for all $x,y$ it holds true that 
\begin{align*}
    \Vert \nabla_x f(x,y) - \nabla g(x) \Vert^2 &\le \frac{L^2}{\mu_y^2} \Vert \nabla_y f(x,y) \Vert^2
\end{align*}
\end{lem}

\begin{lem}[{\citet[Lemma A.5]{nouiehed2019solving}}] \label{lem: g(x)-L-smooth}

Under Assumption \ref{asm: one-PL} and \ref{asm: L-smooth}, then $g(x) $ satisfies $(L+{L^2} /{\mu_y})$-smooth,that is, it holds for all $x,x'$ that
\begin{align*}
    \Vert \nabla g(x) - \nabla g(x') \Vert^2 \le \left (L+ \frac{L^2}{\mu_y}\right) \Vert x - x' \Vert^2.
\end{align*}
Further, noting that ${L}/ {\mu_y} \ge 1$, it implies that $g(x)$ is $({2L^2}/ {\mu_y})$-smooth.

\end{lem}

\begin{lem}[{\citet[Lemma A.3]{nouiehed2019solving}}] \label{lem: prox-argmax}
Under Assumption \ref{asm: L-smooth} and~\ref{asm: one-PL}, for any $x,x'$, any $y^\ast(x) \in \arg \max_{y \in \BR^{d_y}} f(x,y)$, there exists $y^*(x') \in \arg \max_{y' \in \BR^{d_y}} f(x',y')$ such that
\begin{align*}
    \Vert y^\ast(x) - y^{\ast}(x') \Vert^2 \le \frac{L^2}{\mu_y^2} \Vert x - x' \Vert^2.
\end{align*}
\end{lem}


\section{Two-Timescale GDA Matches AGDA}

As a warm-up, we study GDA as well as AGDA with full gradient calculation in this section. After that, it would be easy to extend the analysis to the stochastic setting.

\begin{algorithm*}[htbp] 
\caption{AGDA $(f, (x_0,y_0), K, \tau_x, \tau_y)$} 
\begin{algorithmic}[] \label{alg: AGDA}
    \STATE \textbf{for} $k = 0, 1, \dots, K-1$ \textbf{do}\\[0.15cm]
    \STATE \quad $x_{k+1} = x_k - \tau_x \nabla_x f(x_k,y_k) $ \\[0.15cm]
    \STATE \quad $y_{k+1} = y_k + \tau_y \nabla_y f(x_{k+1},y_k)$ \\[0.15cm]
    \STATE\textbf{end for} \\[0.15cm]
    \STATE \textbf{option I}  (two-sided PL): \textbf{return} $(x_K,y_K)$  \\
    \STATE \textbf{option II} (one-sided PL): \textbf{return} $(\hat x,\hat y)$ chosen uniformly at random from $\{ (x_k,y_k)\}_{k=0}^{K-1}$.
\end{algorithmic}
\end{algorithm*}

\begin{algorithm*}[htbp] 
\caption{GDA $(f, (x_0,y_0), K, \tau_x, \tau_y)$} 
\begin{algorithmic}[] \label{alg: GDA}
    \STATE \textbf{for} $k = 0, 1, \dots, K-1$ \textbf{do}\\[0.15cm]
    \STATE \quad $x_{k+1} = x_k - \tau_x \nabla_x f(x_k,y_k) $ \\[0.15cm]
    \STATE \quad $y_{k+1} = y_k + \tau_y \nabla_y f(x_{k},y_k)$ \\[0.15cm]
    \STATE\textbf{end for} \\[0.15cm]
    \STATE \textbf{option I}  (two-sided PL): \textbf{return} $(x_K,y_K)$  \\
    \STATE \textbf{option II} (one-sided PL): \textbf{return} $(\hat x,\hat y)$ chosen uniformly at random from $\{ (x_k,y_k) \}_{k=0}^{K-1}$
\end{algorithmic}
\end{algorithm*}

\subsection{Convergence under Two-Sided PL condition} \label{sec: GDA}

Under the two-sided PL condition, it is known that AGDA \cite{yang2020global} can find an $\epsilon$-optimal solution with a complexity of $\tilde \fO( n \kappa_x\kappa_y^2 \log ({1}/{\epsilon}))$ when $\kappa_x \gtrsim \kappa_y$. However, the authors left us the question that whether GDA can converge under the same setting. We answer this question affirmatively in this section. We show that the same convergence rate can be achieved by GDA  with simultaneous updates.

We define the following Lyapunov function  suggested by \citet{doan2022convergence}:
\begin{align*}
\mathcal{V}_k &= \mathcal{A}_k + \frac{\lambda \tau_x}{\tau_y} \mathcal{B}_k,
\end{align*}
where $\mathcal{A}_k =  g(x_k ) - g({x}^{\ast})$, $ \mathcal{B}_k = g(x_k) - f(x_k,y_k)$. Then we can obtain the following statement.

\begin{thm} \label{thm: GDA}
Suppose function $f(x,y)$ satisfies $L$-smooth, $\mu_x$-PL in $x$, $\mu_y$-PL in $y$. Let $\tau_y = {1}/{L}$, $\lambda  = {6L^2}/{\mu_y^2}$ and $\tau_x = {\tau_y}/{(22\lambda)}$, then the sequence $\{ (x_k,y_k) \}_{k=1}^{K}$ generated by Algorithm \ref{alg: GDA} satisfies:
\begin{align*}
    \fV_{k+1} \le \left (1 - \frac{\mu_x \tau_x}{2} \right)^k \fV_k.
\end{align*}
\begin{proof}
Since we know that $g$ is $({2L^2}/{\mu_y})$- smooth by Lemma \ref{lem: g(x)-L-smooth}, let $\tau_x \le {\mu_y}/{(2 L^2)}$, we have 
\begin{align} \label{eq: 2-1}
\begin{split}
    g(x_{k+1})  & \le g(x_k ) - g(x^{\ast}) + \nabla g(x_k)^\top(x_{k+1} - x_k) + \frac{L^2}{\mu_y} \Vert x_{k+1} - x_k \Vert^2 \\
    &\le g(x_k)  - \tau_x\nabla g(x_k)^\top \nabla_x f(x_k,y_k) + \frac{\tau_x}{2} \Vert \nabla_x f(x_k,y_k) \Vert^2 \\
    &= g(x_k)  - \frac{\tau_x}{2} \Vert \nabla g(x_k)\Vert^2 +\frac{\tau_x}{2} \Vert \nabla  g(x_k ) - \nabla_x f(x_k,y_k) \Vert^2,
\end{split}
\end{align}
which implies 
\begin{align} \label{eq: 2-2}
    \fA_{k+1} \le \fA_k - \frac{\tau_x}{2} \Vert \nabla g(x_k)\Vert^2 +\frac{\tau_x}{2} \Vert \nabla  g(x_k ) - \nabla_x f(x_k,y_k) \Vert^2.
\end{align}
Using the property of $L$-smooth, we know that the difference between $f(x_k,y_k)$ and $f(x_{k+1},y_{k+1})$ can be bounded. Noting that $\tau_x \le {1}/{L}$, we can obtain
\begin{align} \label{eq: 2-3}
\begin{split}
    f(x_k,y_k) - f(x_{k+1}, y_k) &\le  -\nabla_x f(x_k,y_k)^\top(x_{k+1} -x_k) + \frac{L}{2} \Vert x_{k+1} - x_k \Vert^2 \\
& =\tau_x \Vert \nabla_x f(x_k,y_k) \Vert^2 + \frac{ \tau_x^2 L}{2} \Vert \nabla_x f(x_k,y_k) \Vert^2  \\
&\le \frac{3 \tau_x}{2} \Vert \nabla_x f(x_k,y_k) \Vert^2.
\end{split}
\end{align}

Let $\tau_y < 1/L$, then we have 
\begin{align} \label{eq: 2-4}
\begin{split}
& f(x_{k+1},y_k) - f(x_{k+1},y_{k+1}) \\
\le& - \nabla_y f(x_{k+1},y_k)^\top(y_{k+1} -y_k) + \frac{L}{2} \Vert y_{k+1} - y_k \Vert^2 \\
\le& - \tau_y \nabla_y f(x_{k+1},y_k)^\top \nabla_y f(x_k,y_k)  + \frac{ \tau_y}{2} \Vert \nabla_y f(x_{k},y_k) \Vert^2 \\
=& - \frac{\tau_y}{2} \Vert \nabla_y f(x_{k+1},y_k) \Vert^2 + \frac{\tau_y}{2} \Vert \nabla_y f(x_k,y_k) - \nabla_y f(x_{k+1},y_k) \Vert^2 \\
\le& -\frac{\tau_y}{4} \Vert \nabla_y f(x_k,y_k) \Vert^2 + \tau_y \Vert \nabla_y f(x_k,y_k) - \nabla_y f(x_{k+1},y_k) \Vert^2 \\
\le& -\frac{\tau_y}{4} \Vert \nabla_y f(x_k,y_k) \Vert^2 + \tau_y \tau_x^2 L^2 \Vert \nabla_x f(x_k,y_k) \Vert^2 \\
\le& -\frac{\tau_y}{4} \Vert \nabla_y f(x_k,y_k) \Vert^2 + \tau_x \Vert \nabla_x f(x_k,y_k) \Vert^2,
\end{split}
\end{align}
where in the first inequality we use $f$ is $L$-smooth, and we use $\tau_y \le {1}/{L}$ in the second one and  Young's inequality of $ -\Vert a- b \Vert^2 \le \frac{1}{2} \Vert a \Vert^2 +\Vert b\Vert^2$ in the third one.

Combing (\ref{eq: 2-3}) and (\ref{eq: 2-5}), we can see that
\begin{align} \label{eq: 2-5}
    f(x_k,y_k) - f(x_{k+1},y_{k+1}) &\le - \frac{\tau_y}{4} \Vert \nabla_y f(x_k,y_k) \Vert^2 + \frac{5 \tau_x}{2} \Vert \nabla_x f(x_k,y_k) \Vert^2.
\end{align}
Now we can describe how $\fB_{k+1}$ declines compared with $\fB_k$, using (\ref{eq: 2-1}) and (\ref{eq: 2-5}), we have
\begin{align} \label{eq: 2-6}
    \begin{split}
    \fB_{k+1} 
    &= g(x_{k+1}) -g(x_k ) + g(x_k) - f(x_k,y_k) + f(x_k,y_k) - f(x_{k+1},y_{k+1}) \\
&\le \fB_k - \frac{\tau_x}{2} \Vert \nabla g(x_k)\Vert^2 +\frac{\tau_x}{2} \Vert \nabla  g(x_k ) - \nabla_x f(x_k,y_k) \Vert^2   \\
&\quad - \frac{\tau_y}{4} \Vert \nabla_y f(x_k,y_k) \Vert^2 + \frac{5 \tau_x}{2} \Vert \nabla_x f(x_k,y_k) \Vert^2. 
\end{split}
\end{align}
Using the inequality $\Vert \nabla_x f(x_k,y_k) \Vert^2 \le 2 \Vert \nabla g(x_k) \Vert^2 + 2 \Vert \nabla g(x_k) - \nabla_x f(x_k,y_k) \Vert^2$, we have
\begin{align*}
    \fB_{k+1} &\le \fB_k + \frac{9\tau_x}{2} \Vert \nabla g(x_k) \Vert^2 + \frac{11 \tau_x}{2} \Vert \nabla g(x_k) - \nabla_x f(x_k,y_k) \Vert^2 - \frac{\tau_y}{4} \Vert \nabla_y f(x_k,y_k) \Vert^2.
\end{align*}
By Lemma \ref{lem: g(x)-PL}, Lemma \ref{lem: bound-Bk} and Assumption \ref{asm: two-PL}, we have 
\begin{align} \label{eq: 2-7}
    \begin{split}
        \Vert \nabla g(x_k) \Vert^2 &\ge 2 \mu_x ( g(x_k) - g(x^{\ast})), \\
\Vert \nabla_x f(x_k,y_k) - \nabla g(x_k) \Vert^2 &\le \frac{2L^2}{\mu_y} (g(x_k) - f(x_k,y_k)), \\
\Vert \nabla_y f (x_{k} ,y_k) \Vert^2 &\ge 2 \mu_y (g(x_{k}) - f(x_{k},y_k)).
    \end{split}
\end{align}
Since we let $\tau_y = {1}/ {L}$, $\lambda  = {6L^2} /{\mu_y^2}$ and $\tau_x = {\tau_y} / {(22\lambda)}$, we can obtain
\begin{align} \label{eq: 2-8}
    \begin{split}
        \mathcal{V}_{k+1} &= \mathcal{A}_{k+1} + \frac{\lambda \tau_x}{\tau_y} \mathcal{B}_{k+1} \\
&\le \mathcal{A}_{k} + \frac{\lambda \tau_x}{\tau_y} \mathcal{B}_{k} - \left (1 - \frac{9 \lambda \tau_x}{\tau_y} \right) \frac{\tau_x}{2}\Vert  \nabla g(x_k) \Vert^2 \\
&\quad + \left (1 +\frac{11\lambda \tau_x}{\tau_y} \right )\frac{\tau_x}{2} \Vert \nabla_x f(x_k, y_k) - \nabla g(x_k) \Vert^2 - \frac{\lambda \tau_x}{4} \Vert \nabla_y f(x_k, y_k) \Vert^2 \\
&\le \mathcal{A}_{k}  - \left (1- \frac{9 \lambda \tau_x}{\tau_y} \right ) \tau_x \mu_x \mathcal{A}_k +  \frac{\lambda \tau_x}{\tau_y} \mathcal{B}_{k}+ \left (1 +\frac{11\lambda \tau_x}{\tau_y} \right) \frac{\tau_xL^2 }{\mu_y} \mathcal{B}_k - \frac{\lambda \tau_x \mu_y}{2} \mathcal{B}_k \\
&\le \left (1- \frac{\mu_x \tau_x}{2} \right ) \mathcal{A}_k +  \left (1- \frac{\mu_y \tau_y}{4} \right ) \frac{\lambda \tau_x}{\tau_y} \mathcal{B}_k \\
&\le \left (1- \frac{\mu_x \tau_x}{2} \right)  \mathcal{V}_k,
    \end{split}
\end{align}
where in the second inequality we use $ {11\lambda \tau_x} / {\tau_y}\le {1}/{2}$ by the choices of $\tau_x,\tau_y$ and $\lambda$, while we use the fact that ${ 3\tau_x L^2} / {\mu_y} \le {\lambda \tau_x \mu_y}/{2}$ in the third one and ${\mu_x \tau_x} \le {\mu_y \tau_y} / {2}$ in the last one.

\end{proof}
\end{thm} 

Now we show that the convergence of $\fV_k$ is sufficient to guarantee the convergence to an $\epsilon$-saddle point, defined as follows.

\begin{dfn}
Under Assumption \ref{asm: exist}, we call $(x,y)$ an $\epsilon$-saddle point of problem (\ref{prob:main}) if if holds that $\Vert x - x^* \Vert^2 + \Vert y - y^* \Vert^2 \le \epsilon $ for some saddle point $(x^*,y^*)$.
\end{dfn}

\begin{cor} \label{cor: GDA-saddle}

Suppose  function $f(x,y)$ satisfies $L$-smooth, $\mu_x$-PL in $x$, $\mu_y$-PL in $y$ and  $\kappa_x \gtrsim \kappa_y$. Define $\tau_x,\tau_y$ as in Lemma \ref{thm: GDA} ,then  then the sequence $\{ (x_k,y_k) \}_{k=1}^{K}$ generated by Algorithm \ref{alg: GDA} satisfies:
\begin{align} \label{eq: 11.1}
    \Vert x_k - x^{\ast} \Vert^2 + \Vert y_k - y^{\ast } \Vert^2 \le \frac{2c^k}{(1-\sqrt{c})^2}  \max \left \{ \frac{4}{\mu_x}, \frac{88}{\mu_y} \right\} \fV_0. 
\end{align}
where $c = 1 - {\mu_x \tau_x}/{2}$. Further, Algorithm \ref{alg: GDA} can find an $\epsilon$-saddle point with no more than $\mathcal{O}(n \kappa_x \kappa_y^2 \log ({ \kappa_x \kappa_y}/{\epsilon}))$ stochastic first-order oracle calls.

\begin{proof}

The proof is similar to the proof of Theorem 3.2 in \cite{yang2020global}.

By Lemma \ref{lem: quadratic-growth} and the fact that $2 \tau_x^2 L^2 \le 1$, $\tau_x \le {\mu_y}/{(2L^2)}$ and $\tau_y \le {1}/{L}$ by the choices of $\tau_x,\tau_y$, we can see that
\begin{align} \label{eq: 2-9}
    \begin{split}
        &\quad \Vert x_{k+1} - x_k \Vert^2 + \Vert y_{k+1} - y_k \Vert^2 \\
        &=
        \tau_x^2 \Vert \nabla_x f(x_k,y_k) \Vert^2 + \tau_y^2 \Vert \nabla_y f(x_k,y_k) \Vert^2 \\
        &= \tau_x^2 \Vert \nabla_x f(x_k,y_k) \Vert^2 + \tau_y^2 \Vert \nabla_y f(x_k,y_k) - \nabla_y f(x_k, y^{\ast}(x_k) \Vert^2 \\
        &\le \tau_x^2 \Vert \nabla_x f(x_k,y_k) \Vert^2  + \Vert y_k - y^{\ast}(x_k) \Vert^2 \\
        &\le 2\tau_x^2 \Vert \nabla g(x_k) \Vert^2 + 2 \tau_x^2 \Vert \nabla g(x_k) - \nabla_x f(x_k,y_k) \Vert^2 +\Vert y_k - y^{\ast}(x_k) \Vert^2 \\
        &\le 2 \Vert x_k -x^{\ast} \Vert^2 + 2 \Vert y_k - y^{\ast}(x_k) - y_k \Vert^2 \\
        &\le \frac{4}{\mu_x} \fA_k + \frac{4}{\mu_y} \fB_k \\
        &\le \max \left \{ \frac{4}{\mu_x}, \frac{88}{\mu_y} \right \} \fV_k \\
        &\le \max \left \{ \frac{4}{\mu_x}, \frac{88}{\mu_y} \right \} \left (1 - \frac{\mu_x \tau_x}{2} \right )^k \fV_0,
    \end{split}
\end{align}
where in the last inequality we use ${\lambda \tau_x}/{\tau_y} = {1}/{22}$. 
Then we have
\begin{align*}
    \Vert x_{k+1} - x_k \Vert + \Vert y_{k+1} - y_k \Vert \le \left (1 - \frac{\mu_x \tau_x}{2} \right)^{k/2}  \sqrt{2 \max \left \{ \frac{4}{\mu_x}, \frac{88}{\mu_y} \right \} \fV_0}.
\end{align*}
For $n \ge k$, we obtain
\begin{align*}
\begin{split}
    \Vert x_n - x_k \Vert + \Vert y_n  - y_k \Vert &\le \sum_{i=k}^{n-1} \Vert x_{i+1} - x_i \Vert^2 + \Vert y_{i+1} - y_i \Vert^2\\
    &\le \sqrt{2 \max \left \{ \frac{4}{\mu_x}, \frac{88}{\mu_y} \right \} \fV_0} \sum_{i=k}^{\infty} \left(1- \frac{\mu_x \tau_x}{2} \right )^{i/2} \\
    &\le \frac{c^{k/2}}{1 - \sqrt{c}} \sqrt{2 \max \left \{ \frac{4}{\mu_x}, \frac{88}{\mu_y}\right \} \fV_0},
\end{split}
\end{align*}
where $c = 1 - {\mu_x \tau_x}/{2}$. We know that when $n \rightarrow \infty$, we have $(x_n,y_n) \rightarrow (x^{\ast}, y^{\ast})$ where $(x^{\ast},y^{\ast})$ is a saddle point, Taking square on both sides completes our proof.

\end{proof}

\end{cor} 

\subsection{Convergence under One-Sided PL condition}

When $f$ is nonconvex in $x$, we have the following theorem for GDA.

\begin{thm} \label{thm: GDA-one-side}

Suppose function $f(x,y)$ satisfies $L$-smooth, $\mu_y$-PL in $y$. Let $\tau_y = {1}/{L}$, $\lambda = {4 L^2}/{\mu_y^2}$ and $\tau_x = {\tau_y}/{(18 \lambda)}$, then the sequence $\{ (x_k,y_k) \}_{k=0}^{K-1}$ generated by Algorithm \ref{alg: GDA} satisfies,~
\begin{align*}
    \frac{1}{K} \sum_{k=0}^{K-1} \Vert \nabla g(x_k) \Vert^2 \le \frac{288 L^3}{K \mu_y^2} \mathcal{V}_0.
\end{align*}
Furthermore, if we choose the output $(\hat x,\hat y)$ uniformly from $\{ (x_k,y_k) \}_{k=0}^{K-1}$, then we can get $\Vert \nabla g(\hat x) \Vert \le \epsilon$ with no more than $\mathcal{O}({n \kappa_y^2 L}/{\epsilon^2} )$ first-order oracle calls. 
\begin{proof}
Using equation (\ref{eq: 2-2}) and Lemma \ref{lem: bound-Bk} that $ \Vert \nabla g(x_k) - \nabla_x f(x_k,y_k) \Vert^2 \le {2L^2 \fB_k} /{\mu_y} $,we have
\begin{align} \label{eq: 10}
    \fA_{k+1} 
&\le \fA_k - \frac{\tau_x}{2} \Vert \nabla g(x_k)\Vert^2 +\frac{\tau_x L^2}{\mu_y} \mathcal{B}_k.
\end{align}
Further, using equation (\ref{eq: 2-6}), we have 
\begin{align} \label{eq: 11}
\begin{split}
    \fB_{k+1 } &\le \mathcal{B}_k + \frac{9\tau_x}{2} \Vert \nabla g(x_k) \Vert^2 + \frac{11 \tau_x}{2} \Vert \nabla g(x_k) - \nabla_x f(x_k,y_k) \Vert^2 - \frac{\tau_y}{4} \Vert \nabla_y f(x_k,y_k) \Vert^2 \\
    &\le \fB_k + \frac{9\tau_x}{2} \Vert \nabla g(x_k) \Vert^2 + \frac{11 \tau_x L^2}{\mu_y} \fB_k - \frac{\mu_y \tau_y}{2} \fB_k \\
    &\le (1- \frac{\mu_y \tau_y}{4}) \fB_k + \frac{9\tau_x}{2} \Vert \nabla g(x_k) \Vert^2, 
\end{split}
\end{align}
where we use Lemma \ref{lem: bound-Bk} and PL condition in $y$ in the first inequality and ${11 \tau_x L^2}/{\mu_y} \le {\mu_y \tau_y}/{4}$ by the choices of $\tau_x,\tau_y$. Thus,
\begin{align*}
\fB_k \le \left (1 - \frac{\mu_y \tau_y}{4} \right)^k \fB_0 + \frac{9 \tau_x}{2} \sum_{i=0}^{k-1} \left (1 - \frac{\mu_y \tau_y}{4} \right)^{k-i-1} \Vert \nabla g(x_i) \Vert^2.   
\end{align*}
Plugging  into (\ref{eq: 10}), 
{\small
\begin{align*}
    \mathcal{A}_{k+1} &\le \mathcal{A}_k - \frac{\tau_x}{2} \Vert \nabla g(x_k) \Vert^2 + \frac{\tau_x L^2}{\mu_y}\left (1 - \frac{ \mu_y \tau_y}{4} \right)^{k} \mathcal{B}_0  +\frac{9\tau_x^2 L^2}{2\mu_y} \sum_{i=0}^{k-1} \left (1 - \frac{ \mu_y \tau_y}{4} \right)^{k-i-1} \Vert \nabla g(x_i) \Vert^2.
\end{align*}}
Telescoping and noticing that ${18\tau_x^2 L^2}/{\tau_y \mu_y^2} \le {\tau_x}/{4}$ and $\lambda =  {4L^2}/{\mu_y^2}$, we have 
\begin{align*}
    \fA_{K+1} &\le \fA_0 - \frac{\tau_x}{2} \sum_{k=0}^{K} \Vert \nabla g(x_k) \Vert^2 + \frac{\tau_x L^2}{\mu_y} \sum_{k=0}^{K-1} \left (1 - \frac{\mu_y \tau_y}{4}\right )^k \fB_0 \\
    &\quad \quad \quad + \frac{9\tau_x^2 L^2}{2\mu_y} \sum_{k=1}^K \sum_{i=0}^{k-1} \left (1 - \frac{\mu_y \tau_y}{4}\right )^{k-i-1} \Vert \nabla g(x_i) \Vert^2 \\
    &= \fA_0 - \frac{\tau_x}{2} \sum_{k=0}^{K} \Vert \nabla g(x_k) \Vert^2 + \frac{\tau_x L^2}{\mu_y} \sum_{k=0}^{K-1} \left (1 - \frac{\mu_y \tau_y}{4} \right)^k \fB_0 \\
    &\quad \quad \quad + \frac{9\tau_x^2 L^2}{2\mu_y} \sum_{i=0}^{K-1} \sum_{k=i+1}^{K} \left (1 - \frac{\mu_y \tau_y}{4} \right)^{k-i-1} \Vert \nabla g(x_i) \Vert^2 \\
    &\le \fA_0 - \frac{\tau_x}{2} \sum_{k=0}^{K} \Vert \nabla g(x_k) \Vert^2 + \frac{\tau_x L^2}{\mu_y} \sum_{k=0}^{K-1} \left (1 - \frac{\mu_y \tau_y}{4} \right )^k \fB_0 + \frac{18 \tau_x^2 L^2}{\tau_y \mu_y^2} \sum_{i=0}^{K-1} \Vert \nabla g(x_i) \Vert^2 \\
    &\le \fA_0 - \frac{\tau_x}{2} \sum_{k=0}^{K} \Vert \nabla g(x_k) \Vert^2 + \frac{\tau_x L^2}{\mu_y} \sum_{k=0}^{K-1} \left (1 - \frac{\mu_y \tau_y}{4} \right)^k \fB_0 + \frac{18 \tau_x^2 L^2}{\tau_y \mu_y^2} \sum_{k=0}^{K} \Vert \nabla g(x_k) \Vert^2 \\
    &\le \fA_0 - \frac{\tau_x}{4} \sum_{k=0}^K \Vert \nabla g(x_k) \Vert^2 +  \frac{4 \tau_x L^2}{\tau_y \mu_y^2} \fB_0 \\
    &= \fV_0 - \frac{\tau_x}{4} \sum_{k=0}^K \Vert \nabla g(x_k) \Vert^2.
\end{align*}
Rearranging and noticing that $\fA_{K+1} \ge  0 $, we can see that 
\begin{align*}
    \frac{1}{K+1} \sum_{k=0}^{K} \Vert \nabla g(x_k) \Vert^2 \le \frac{4 \fV_0}{(K+1) \tau_x}, 
\end{align*}
which is equivalent to the desired inequality.
\end{proof}

\end{thm}

\section{Convergence of GDA with SVRG Gradient Estimators} \label{apx: SVRG-GDA}

In this section, we show the convergence rate of GDA with SVRG gradient estimators (Algorithm \ref{alg: SVRG-GDA}) can be $\fO( (n+ n^{2/3} \kappa_x \kappa_y^2 ) \log (1/\epsilon))$, improving the result of $\fO( (n+ n^{2/3} \max\{\kappa_x^3, \kappa_y^3 \} ) \log (1/\epsilon))$  by \citet{yang2020global} with SVRG-AGDA. 


\begin{algorithm*}[t] 
\caption{SVRG-GDA $(f, (x_0,y_0), T,S,M,B, \tau_x, \tau_y)$} 
\begin{algorithmic}[1] \label{alg: SVRG-GDA}
\STATE $\bar x_0 = x_0, \bar y_0 = y_0$ \\[0.15cm]
\STATE \textbf{for} $t = 0,1,\dots, T-1$ \textbf{do} \\[0.15cm]
\STATE \quad \textbf{for} $s = 0, 1, \dots, S-1$ \textbf{do}\\[0.15cm]
\STATE \quad \quad $x_{s,0} = \bar x_s, y_{s,0} = \bar y_s$ \\[0.15cm]
\STATE \quad \quad compute $\nabla_x f(\bar x_s,\bar y_s) = \dfrac{1}{n} \sum_{i=1}^n \nabla_x f_i(\bar x_s,\bar y_s) $ \\[0.15cm]
\STATE \quad \quad compute $\nabla_y f(\bar x_s,\bar y_s) = \dfrac{1}{n} \sum_{i=1}^n \nabla_y f_i(\bar x_s,\bar y_s) $ \\[0.15cm]
\STATE \quad \quad \textbf{for } $k= 0,1, \dots, M-1$ \\[0.15cm]
\STATE \quad \quad \quad Draw samples $S_x,S_y$ independently with both size $B$. \\[0.15cm]
\STATE \quad \quad \quad  $G_x(x_{s,k},y_{s,k}) =\dfrac{1}{B} \sum_{i \in S_x}[\nabla_x f_{i} (x_{s,k},y_{s,k}) - \nabla_x f_{i} (\bar x_s, \bar y_s) + \nabla_x f(\bar x_s,\bar y_s)]$ \\[0.15cm]
\STATE \quad \quad \quad $G_y(x_{s,k},y_{s,k}) = \dfrac{1}{B} \sum_{i\in S_y}[\nabla_y f_{i} (x_{s,k},y_{s,k}) - \nabla_x f_{i} (\bar x_s, \bar y_s) + \nabla_x f(\bar x_s,\bar y_s)]$ \\[0.15cm]
\STATE \quad \quad \quad $x_{s,k+1} = x_{s,k} - \tau_x G_x(x_{s,k},y_{s,k})$ \\[0.15cm]
\STATE \quad \quad \quad $y_{s,k+1} = y_{s,k} + \tau_y G_y(x_{s,k},y_{s,k})$ \\[0.15cm]
\STATE \quad \quad \textbf{end for} \\[0.15cm]
\STATE \quad \quad $\bar x_{s+1} = x_{s,M}, \bar y_{s+1} = y_{s,M}$ \\[0.15cm]
\STATE \quad \textbf{end for} \\[0.15cm]
\STATE \quad Choose $(x_t,y_t)$ from $\{ \{ (x_{s,k}, y_{s,k} )\}_{k=0}^{M-1} \}_{s=0}^{S-1}$ uniformly at random. \\ [0.15cm]
\STATE \quad $ \bar x_0 = x_t, \bar y_0 = y_t$\\ [0.15cm]
\STATE \textbf{end for} \\[0.15cm]
\STATE \textbf{return} $(x_T,y_T)$
\end{algorithmic}
\end{algorithm*}

For the innermost loop about subscript $k$ when $t$ and $s$ are both fixed,  we define the Lyapunov function:
\begin{align*}
    \mathcal{V}_{s,k} &= \mathcal{A}_{s,k} + \frac{\lambda \tau_x}{ \tau_y} \mathcal{B}_{s,k} + c_{s,k} \Vert x_{s,k} - \bar x_s \Vert^2 + d_{s,k}  \Vert y_k - \bar y_s \Vert^2,
\end{align*}
where $\fA_{s,k} = g(x_{s,k}) - g(x^{\ast}) $ and $\fB_{s,k} = g(x_{s,k}) - f(x_{s,k},y_{s,k})$ and $c_{s,k},d_{s,k}$ will be defined recursively with $c_{s,M} = d_{s,M} = 0$ in our proof. Then we can have the following lemma.

\begin{lem} \label{lem: SVRG-GDA-V_k}

Under Assumption \ref{asm: one-PL} and \ref{asm: L-smooth}, if we let $\tau_y = {\nu}/{(L n^{\alpha})}$, $\lambda  = {14L^2}/{\mu_y^2}$ and $\tau_x = {\tau_y}/{(22 \lambda)}$, where  $\nu = 1/(176(\rm{e}-1)), 0< \alpha \le 1$; let $B = 1, M = \lfloor {n^{3 \alpha/2}}/{(2\nu)} \rfloor$. Then for Algorithm \ref{alg: SVRG-GDA}, the following statement holds true:
\begin{align*}
     \BE[\fV_{s,k+1}] &\le \fV_{s,k} - \frac{\tau_x}{8} \Vert \nabla g(x_{s,k}) \Vert^2 - \frac{\lambda \tau_x}{16} \Vert \nabla_y f(x_{s,k},y_{s,k}) \Vert^2,
\end{align*}
Above, the definitions of $c_{s,k},d_{s,k}$ is given recursively with $c_{s,M} = d_{s,M} = 0$ as:
\begin{align*}
c_{s,k} &= c_{s,k+1}(1+ \tau_x \gamma_1) + \left(c_{s,k+1}\tau_x^2+ \frac{3\tau_x^2 L^2}{ \mu_y}\right ) L^2 + \left (d_{s,k+1}\tau_y^2 +\frac{\lambda \tau_x\tau_y L}{2}\right) L^2, \\
d_{s,k} &= d_{s,k+1}(1+\tau_y \gamma_2) + \left (c_{s,k+1}\tau_x^2+ \frac{3\tau_x^2 L^2}{\mu_y} \right) L^2 + \left (d_{s,k+1}\tau_y^2 +\frac{\lambda \tau_x\tau_y L}{2} \right ) L^2.
\end{align*}

\begin{proof}

Since $s$ is fixed in the lemma, we omit  subscripts of $s$ in the following proofs, then the Lyapunov function can be written as:
\begin{align*}
    \mathcal{V}_{k} &= \mathcal{A}_k + \frac{\lambda \tau_x}{ \tau_y} \mathcal{B}_k + c_{k} \Vert x_k - \bar x \Vert^2 + d_k  \Vert y_k - \bar y \Vert^2.
\end{align*}
Before the formal proof, we present some standard properties of variance reduction. We denote the stochastic gradients as:
\begin{align*}
    G_x(x_k,y_k) &= \frac{1}{B} \sum_{i \in S_x} \big(\nabla_x f_{i} (x_{k},y_{k}) - \nabla_x f_{i} (\bar x, \bar y) + \nabla_x f(\bar x,\bar y)\big)
\end{align*}
and
\begin{align*}    
    G_y(x_k,y_k) &= \frac{1}{B} \sum_{i \in S_y} \big(\nabla_y f_{i} (x_{k},y_{k}) - \nabla_y f_{i} (\bar x, \bar y) + \nabla_y f(\bar x,\bar y)\big).
\end{align*}
Then we know that the stochastic gradients satisfy unbiasedness that 
\begin{align*}
    \BE[G_x(x_k,y_k)] = \nabla_x f(x_k,y_k) \quad\text{and}\quad
    \BE[G_y(x_k,y_k)] = \nabla_y f(x_k,y_k).
\end{align*}
And we  can bound the variance of the stochastic gradients as follows:
\begin{align} \label{eq: bound-var-x}
\begin{split}
&\quad \BE\Vert G_x(x_k,y_k) - \nabla_x f(x_k,y_k) \Vert^2 \\
&= \BE\Vert \nabla_x f_{i} (x_{k},y_{k}) - \nabla_x f_{i} (\bar x, \bar y) + \nabla_x f(\bar x,\bar y) - \nabla_x f(x_k,y_k) \Vert^2 \\
&\le \BE\Vert \nabla_x f_{i} (x_{k},y_{k}) - \nabla_x f_{i} (\bar x, \bar y) \Vert^2 \\
&\le  L^2 \BE\Vert x_{k} - \bar x \Vert^2] +L^2 \BE [\Vert y_{k} - \bar y \Vert^2. 
\end{split}
\end{align}
Similarly, we have
\begin{align} \label{eq: bound-var-y}
    \BE\Vert G_y(x_k,y_k) - \nabla_y f(x_k,y_k) \Vert^2 
    &\le  L^2 \BE\Vert x_{k} - \bar x \Vert^2 +L^2 \BE \Vert y_{k} - \bar y \Vert^2.
\end{align}

Equipped with the above properties of SVRG, now we can begin our proof of Lemma \ref{lem: SVRG-GDA-V_k}. 

Since we know that $g$ is $({2L^2}/{\mu_y})$-smooth by Lemma \ref{lem: g(x)-L-smooth} and  $\tau_x \le {\mu_y}/{2 L^2}$, we have
\begin{align} \label{eq: bound-SVRG-Ak}
\begin{split}
   \BE[g(x_{k+1})] 
    & \le \BE\left[g(x_k )  + \nabla g(x_k)^\top(x_{k+1} - x_k) + \frac{2L^2}{\mu_y} \Vert x_{k+1} - x_k \Vert^2\right] \\
    &\le \BE\left[g(x_k)  - \tau_x\nabla g(x_k)^\top G_x(x_k,y_k) + \frac{\tau_x^2 L^2}{\mu_y} \Vert G_x(x_k,y_k) \Vert^2\right] \\
    &\le \BE\left[g(x_k)  - \tau_x\nabla g(x_k)^\top \nabla_x f(x_k,y_k) + \frac{\tau_x}{2} \Vert \nabla_x f(x_k,y_k) \Vert^2 \right] \\
    &\quad +\BE\left[ \frac{\tau_x^2 L^2}{\mu_y} \Vert G_x(x_k,y_k) - \nabla_x f(x_k,y_k) \Vert^2 \right] \\
    &= \BE\left[g(x_k)  - \frac{\tau_x }{2} \Vert \nabla g(x_k)\Vert^2 +\frac{\tau_x}{2} \Vert \nabla  g(x_k ) - \nabla_x f(x_k,y_k) \Vert^2 \right] \\
    &\quad + \BE \left[\frac{\tau_x^2 L^2}{\mu_y} \Vert G_x(x_k,y_k) - \nabla_x f(x_k,y_k) \Vert^2 \right] ,
\end{split}
\end{align}
where we use $ \tau_x^2 L^2 \le \mu_y$ in the third inequality. Also, we have 
\begin{align*}
\mathbb{E}[ f(x_k,y_k) ]&\le \mathbb{E}\left[ f(x_{k+1},y_{k} ) - \nabla_x f(x_k,y_k)^\top(x_{k+1}-x_k) + \frac{L}{2} \Vert x_{k+1} - x_k \Vert^2\right] \\
&= \mathbb{E} \left[ f(x_{k+1},y_k) + \tau_x\nabla_x f(x_k,y_k)^\top G_x(x_k,y_k) + \frac{\tau_x^2L}{2} \Vert G_x(x_k,y_k ) \Vert^2\right] \\
&= \mathbb{E}\left[ f(x_{k+1},y_k)  + \tau_x  \Vert \nabla_x f(x_k,y_k) \Vert^2 + \frac{\tau_x^2L}{2} \Vert \nabla_x f(x_k,y_k) \Vert^2 \right]\\
&\quad + \BE \left[ \frac{\tau_x^2L}{2} \Vert G_x(x_k,y_k) - \nabla_x f(x_k,y_k) \Vert^2\right] \\
&\le \mathbb{E} \left[ f(x_{k+1},y_k) + \frac{3\tau_x}{2} \Vert \nabla_x f(x_k,y_k) \Vert^2 + \frac{\tau_x^2L}{2} \Vert G_x(x_k,y_k) - \nabla_x f(x_k,y_k) \Vert^2 \right],
\end{align*}
where we use the quadratic upper bound implied by $L$-smoothness in the first inequality and $\tau_y \le 1/L$ in the second one. Similarly,  
{\small \begin{align*}
\mathbb{E}[ f(x_{k+1},y_k)] &\le \mathbb{E}\left[ f(x_{k+1},y_{k+1}) - \nabla_y f(x_{k+1},y_k)^\top(y_{k+1} - y_k) + \frac{L}{2} \Vert y_{k+1} - y_k \Vert^2\right] \\
&=\mathbb{E}\left[ f(x_{k+1},y_{k+1} ) - \tau_y \nabla_y f(x_{k+1},y_k)^\top G_y(x_k,y_k) + \frac{\tau_y^2L}{2} \Vert G_y (x_k,y_k) \Vert^2\right] \\
&\le \mathbb{E}\left[ f(x_{k+1},y_{k+1} ) - \tau_y  \nabla_y f(x_{k+1},y_k)^\top  \nabla_y f(x_k,y_k) + \frac{\tau_y}{2} \Vert \nabla_y f(x_k,y_k) \Vert^2 \right]\\
&\quad + \BE \left[ \frac{\tau_y^2L}{2} \Vert G_y(x_k,y_k) - \nabla_y f(x_k,y_k) \Vert^2\right] \\
&= \mathbb{E}\left[ f(x_{k+1},y_{k+1} ) -\frac{\tau_y}{2} \Vert \nabla_y f(x_{k+1},y_{k}) \Vert^2 \right]\\
&\quad + \BE \left[\frac{\tau_y}{2} \Vert \nabla_y f(x_{k+1},y_k) - \nabla_y f(x_k,y_k) \Vert^2 + \frac{\tau_y^2L}{2} \Vert G_y(x_k,y_k) - \nabla_y f(x_k,y_k) \Vert^2\right]\\
&\le \mathbb{E}\left[ f(x_{k+1},y_{k+1} ) -\frac{\tau_y}{4} \Vert \nabla_y f(x_{k+1},y_{k}) \Vert^2 \right] \\
&\quad + \BE \left[\tau_y \Vert \nabla_y f(x_{k+1},y_k) - \nabla_y f(x_k,y_k) \Vert^2 + \frac{\tau_y^2L}{2} \Vert G_y(x_k,y_k) - \nabla_y f(x_k,y_k) \Vert^2\right]\\
&\le \mathbb{E}\left[ f(x_{k+1},y_{k+1} ) -\frac{\tau_y}{4} \Vert \nabla_y f(x_{k+1},y_{k}) \Vert^2 \right] \\
&\quad + \BE \left[\tau_y \tau_x^2 L^2\Vert G_x(x_k,y_k) \Vert^2 + \frac{\tau_y^2L}{2} \Vert G_y(x_k,y_k) - \nabla_y f(x_k,y_k) \Vert^2\right]\\
&\le\mathbb{E}\left[ f(x_{k+1},y_{k+1} ) -\frac{\tau_y}{4} \Vert \nabla_y f(x_{k+1},y_{k}) \Vert^2 + \tau_x \Vert \nabla_x f(x_k,y_k) \Vert^2 \right]\\
&\quad + \BE \left[\tau_x^2  L \Vert G_x(x_k,y_k) - \nabla_x f(x_k,y_k) \Vert^2 + \frac{\tau_y^2L}{2} \Vert G_y(x_k,y_k) - \nabla_y f(x_k,y_k) \Vert^2\right].
\end{align*}}
Above, the first and fourth inequalities are both due to $L$-smoothness; the second one follows from $\tau_y \le 1/L$; the third one uses the fact that $ -\BE [\Vert a - b \Vert^2] \le -\frac{1}{2} \BE[\Vert a \Vert^2] + \BE [\Vert b \Vert^2]$; the last one relies on $ \BE[ \Vert G_x(x_k,y_k) \Vert^2] = \BE[ \Vert \nabla_x f(x_k,y_k) \Vert^2 + \Vert G_x(x_k,y_k) - \nabla_x f(x_k,y_k) \Vert^2  ] $ and the choices of $\tau_x,\tau_y$.
Summing up the above two inequalities, we obtain 
{\small\begin{align*}
\mathbb{E}[f(x_k,y_k)] &\le \mathbb{E}\left[f(x_{k+1},y_{k+1}) + \frac{5\tau_x}{2} \Vert \nabla_x f(x_k,y_k) \Vert^2 + \frac{3\tau_x^2L}{2} \Vert G_x(x_k,y_k) - \nabla_x f(x_k,y_k) \Vert^2 \right] \\
&\quad \BE \left[- \frac{\tau_y}{4} \Vert \nabla_y f(x_k,y_k) \Vert^2 + \frac{\tau_y^2L}{2} \Vert G_y(x_k,y_k) - \nabla_y f(x_k,y_k) \Vert^2\right].
\end{align*}}
Combing with inequality (\ref{eq: bound-SVRG-Ak}), we can see that
\begin{align} \label{eq: bound-SVRG-Bk}
\begin{split}
    \BE[\fB_{k+1}] &\le \BE\left[\fB_k - \frac{\tau_x}{2} \Vert \nabla g(x_k) \Vert^2 + \frac{\tau_x}{2} \Vert \nabla g(x_k) - \nabla_x f(x_k,y_k) \Vert^2\right ] \\
&\quad + \BE\left[  - \frac{\tau_y}{4} \Vert \nabla_y f(x_k,y_k) \Vert^2 + \frac{5 \tau_x}{2} \Vert \nabla_x f(x_k,y_k) \Vert^2\right] \\
&\quad + \BE\left[ \frac{\tau_y^2L}{2} \Vert G_y(x_k,y_k) - \nabla_y f(x_k,y_k) \Vert^2 + \frac{5\tau_x^2 L^2}{2\mu_y} \Vert G_x(x_k,y_k) - \nabla_x f(x_k,y_k) \Vert^2 \right] \\
&\le \BE \left[ \fB_k + \frac{9 \tau_x}{2} \Vert \nabla g(x_k) \Vert^2 + \frac{11 \tau_x}{2} \Vert \nabla g(x_k) - \nabla_x f(x_k,y_k) \Vert^2 - \frac{\tau_y}{4} \Vert \nabla_y f(x_k,y_k) \Vert^2 \right ] \\
&\quad + \BE\left[ \frac{\tau_y^2L}{2} \Vert G_y(x_k,y_k) - \nabla_y f(x_k,y_k) \Vert^2 + \frac{5\tau_x^2 L^2}{2\mu_y} \Vert G_x(x_k,y_k) - \nabla_x f(x_k,y_k) \Vert^2 \right],
\end{split}
\end{align}
where we use $\BE\Vert \nabla_x f(x_k,y_k) \Vert^2 \le \BE\Vert \nabla g(x_k) \Vert^2 + \BE\Vert \nabla g(x_k) - \nabla_x f(x_k,y_k) \Vert^2$.
Using Young's inequality as equation (37) and (38) in \cite{yang2020global}, we have 
\begin{align*}
\begin{split}
    \BE\Vert x_{k+1} - \bar x \Vert^2 &\le \BE \left[ \tau_x^2  \Vert G_x(x_k,y_k) - \nabla_x f(x_k,y_k) \Vert^2\right] \\
    & \quad + \BE \left[ (1+ \tau_x \gamma_1) \Vert x_k - \bar x \Vert^2 + \left(\tau_x^2 + \frac{\tau_x}{\gamma_1} \right) \Vert \nabla_x f(x_k,y_k) \Vert^2 \right] ,\\
    \BE\Vert y_{k+1} - \bar y \Vert^2 &\le \BE \left[  \tau_y^2 \Vert G_y(x_k,y_k) - \nabla_y f(x_k,y_k) \Vert^2 \right] \\
    &\quad  + \BE \left[ (1+ \tau_y \gamma_2) \Vert y_k - \bar y \Vert^2 + \left(\tau_y^2 + \frac{\tau_y}{\gamma_2} \right) \Vert \nabla_y f(x_k,y_k) \Vert^2 \right],
\end{split}
\end{align*}
where $\gamma_1,\gamma_2$ are two positive constant in Young's inequality  which will be chosen later. 

Then, using equation (\ref{eq: bound-var-x}), (\ref{eq: bound-var-y}), (\ref{eq: bound-SVRG-Ak}), (\ref{eq: bound-SVRG-Bk}),   we have
\begin{align} \label{eq: 18}
\begin{split}
\BE[\mathcal{V}_{k+1}] &= \BE \left[ \fA_{k+1} + \frac{\lambda \tau_x}{\tau_y} \fB_{k+1} +c_{k+1} \Vert x_{k+1} -\bar x \Vert^2 + d_{k+1} \Vert y_{k+1} - \bar y \Vert^2 \right]\\
&\le \mathcal{A}_k + \frac{\lambda \tau_x}{\tau_y} \fB_k -  \left(1 - \frac{9 \lambda \tau_x}{\tau_y}\right ) \frac{\tau_x}{2}\Vert  \nabla g(x_k) \Vert^2 + \\
&\quad + \left (1 +\frac{11\lambda \tau_x}{\tau_y}\right )\frac{\tau_x}{2} \Vert \nabla_x f(x_k, y_k) - \nabla g(x_k) \Vert^2  \\
&\quad - \frac{\lambda \tau_x}{4} \Vert \nabla_y f(x_k,y_k) \Vert^2 + \frac{2 \tau_x^2 L^2}{\mu_y} \left(1 + \frac{5\lambda \tau_x}{4\tau_y} \right)\Vert G_x (x_k,y_k) - \nabla_x f(x_k,y_k) \Vert^2 \\
&\quad + \frac{\lambda \tau_x \tau_y L}{2} \Vert G_y(x_k,y_k) - \nabla_y f(x_k,y_k) \Vert^2\\
&\quad  + \BE [c_{k+1} \Vert x_{k+1} -\bar x \Vert^2 + d_{k+1} \Vert y_{k+1} - \bar y \Vert^2] \\
&= \fV_k - \left(1 - \frac{9 \lambda \tau_x}{\tau_y}\right ) \frac{\tau_x}{2}\Vert  \nabla g(x_k) \Vert^2 + \left (1 +\frac{11\lambda \tau_x}{\tau_y}\right )\frac{\tau_x}{2} \Vert \nabla_x f(x_k, y_k) - \nabla g(x_k) \Vert^2  \\
&\quad - \frac{\lambda \tau_x}{4} \Vert \nabla_y f(x_k, y_k) \Vert^2+ c_{k+1} \left (\tau_x^2 + \frac{\tau_x}{\gamma_1} \right ) \Vert \nabla_x f(x_k,y_k)  \Vert^2\\
&\quad + d_{k+1}\left ( \tau_y^2+\frac{\tau_y}{\gamma_2} \right ) \Vert \nabla_y f(x_k,y_k) \Vert^2  \\
&\le \fV_k - \left(1 - \frac{9 \lambda \tau_x}{\tau_y}\right ) \frac{\tau_x}{2}\Vert  \nabla g(x_k) \Vert^2 + \left (1 +\frac{11\lambda \tau_x}{\tau_y}\right )\frac{\tau_x}{2} \Vert \nabla_x f(x_k, y_k) - \nabla g(x_k) \Vert^2  \\
&\quad - \frac{\lambda \tau_x}{4} \Vert \nabla_y f(x_k, y_k) \Vert^2+ d_{k+1}\left ( \tau_y^2+\frac{\tau_y}{\gamma_2} \right ) \Vert \nabla_y f(x_k,y_k) \Vert^2 \\
&\quad +  2c_{k+1} \left(\tau_x^2 + \frac{\tau_x}{\gamma_1}\right ) \Vert \nabla g(x_k) \Vert^2 \\
&\quad + 2c_{k+1} \left (\tau_x^2 + \frac{\tau_x}{\gamma_1} \right ) \Vert \nabla_x f(x_k,y_k) - \nabla g(x_k) \Vert^2 \\ 
&\le \mathcal{V}_k -   \frac{\tau_x}{4}\Vert  \nabla g(x_k) \Vert^2 + \frac{3\tau_x}{4} \Vert \nabla_x f(x_k, y_k) - \nabla g(x_k) \Vert^2 - \frac{\lambda \tau_x}{4} \Vert \nabla_y f(x_k, y_k) \Vert^2 \\
&\quad + d_{k+1}\left ( \tau_y^2+\frac{\tau_y}{\gamma_2} \right ) \Vert \nabla_y f(x_k,y_k) \Vert^2 +  2c_{k+1} \left (\tau_x^2 + \frac{\tau_x}{\gamma_1} \right ) \Vert \nabla g(x_k) \Vert^2 \\
&\quad + 2c_{k+1} \left (\tau_x^2 + \frac{\tau_x}{\gamma_1} \right ) \Vert \nabla_x f(x_k,y_k) - \nabla g(x_k) \Vert^2, 
\end{split}
\end{align}
where the second last inequality relies on 
\begin{align*} \Vert \nabla_x f(x_k,y_k) \Vert^2 \le 2 \Vert \nabla g(x_k) \Vert^2 + 2 \Vert \nabla g(x_k) - \nabla_x f(x_k,y_k) \Vert^2;
\end{align*}
in the last inequality we use ${11\lambda \tau_x}/{\tau_y } \le {1}/{2}$ by our choices of $\lambda,\tau_x,\tau_y$; the second equality is due to the definition of $c_{k+1}, d_{k+1}$.

Now we define $e_k = \max\{c_k,d_k \}$ and we bound $e_k$ by letting $ \gamma_1 = {\lambda L}/{n^{\alpha/2}}$ and $\gamma_2 = {L}/{n^{\alpha/2}}$. Then according to  the definition of $c_k,d_k$ given by our definition, we have
\begin{align*}
    e_k &\le (1 + \tau_y \gamma_2 + \tau_y^2 L^2) e_{k+1} + \frac{3 \tau_x^2 L^4}{\mu_y} + \frac{\lambda \tau_x \tau_y L^3}{2} \\ 
    & \le (1 + \tau_y \gamma_2 + \tau_y^2 L^2) e_{k+1} + \tau_y^2L^3 \\
    &= \left (1 + \frac{\nu}{n^{3\alpha/2}} + \frac{\nu^2}{n^{2 \alpha}} \right ) + \frac{L\nu^2 }{n^{2 \alpha}} \\
    &\le \left(1 + \frac{2 \nu}{n^{3\alpha/2}} \right) e_{k+1} + \frac{L\nu^2 }{n^{2 \alpha}},
\end{align*}
where we use  $\tau_x \le \tau_y$ and $\gamma_1 \tau_x \le \gamma_2 \tau_y$ in the first; the second inequality is due to $\tau_x^2 L / \mu_y \le \tau_y^2 /6$ and $\lambda \tau_x \le \tau_y /4$ ; we plug in $\tau_y = {\nu}/{L n^{\alpha}}$ in the third line;  we use $ \nu \le 1$ in the last inequality.

Since  $M = \lfloor {n^{3 \alpha/2}}/{(2\nu)} \rfloor$ and $c_M = d_M  =0$, if we define  $\theta = {2 \nu }/{n^{3 \alpha/2}}$, then
\begin{align} \label{eq: bound-e_k}
    e_0 & \le \frac{ L \nu^2}{n^{2 \alpha}}  \frac{(1 + \theta)^M - 1}{\theta} \le \frac{L \nu ({\rm e}-1)}{ 2n^{\alpha/2}}.
\end{align}
Since $e_{k+1} \le e_k$, we know that $e_k \le e_0$, then 
\begin{align} \label{eq: bound-re-d_k}
\begin{split}
    d_{k+1}\left ( \tau_y^2+\frac{\tau_y}{\gamma_2} \right) & \le e_0 \left (\tau_y + \frac{1}{\gamma_2}\right) \tau_y \\
    &\le \frac{L \nu (\rm{e}-1)}{ 2n^{\alpha/2}}  \left (\tau_y + \frac{1}{\gamma_2}\right) \tau_y \\
    &= \frac{ \nu(\rm{e}-1)}{2} \left ( \frac{\nu}{n^{3\alpha/2}} + 1 \right) \tau_y \\
    &\le \nu(\rm{e}-1) \tau_y,
\end{split}
\end{align}
where we use $d_{k+1} \le e_0$  in the first inequality and (\ref{eq: bound-e_k}) in the second one, and in the third line we plug in $\tau_y = {\nu}/{(L n^{\alpha})}$ and $\gamma_2 = {L}/{n^{\alpha/2}}$. The last inequality follows from ${\nu}/{n^{3\alpha/2}} \le 1$.

Similarly, note that $\gamma_1 \ge \gamma_2$  and $\tau_x \le \tau_y$ by the choices of $\tau_x,\tau_y$, then we have
\begin{align} \label{eq: bound-re-c_k}
c_{k+1}\left (\tau_x^2 + \frac{\tau_x}{\gamma_1} \right) \le e_0 \left (\tau_x + \frac{1}{\gamma_1}\right) \tau_x 
\le e_0\left (\tau_y + \frac{1}{\gamma_2} \right) \tau_x \le3\nu({ \rm e}-1)\tau_x. 
\end{align}
Plugging (\ref{eq: bound-re-d_k}) and (\ref{eq: bound-re-c_k}) into (\ref{eq: 18}), 
\begin{align*} 
\begin{split}
    \BE[\mathcal{V}_{k+1}] 
&\le \mathcal{V}_k -   \frac{\tau_x}{4}\Vert  \nabla g(x_k) \Vert^2 + \frac{3\tau_x}{4} \Vert \nabla_x f(x_k, y_k) - \nabla g(x_k) \Vert^2 - \frac{\lambda \tau_x}{4} \Vert \nabla_y f(x_k, y_k) \Vert^2 \\
&\quad + \nu({\rm e}-1) \tau_y \Vert \nabla_y f(x_k,y_k) \Vert^2 +  2 \nu(\rm{e}-1) \tau_x \Vert \nabla g(x_k) \Vert^2 \\
&\quad + 2 \nu({\rm e}-1) \tau_x \Vert \nabla_x f(x_k,y_k) - \nabla g(x_k) \Vert^2.
\end{split}
\end{align*}
If we let $\nu \le {1}/{(176{(\rm e}-1))}$, then we can verify that the following statements hold true:
\begin{align*}
     \nu({\rm e}-1) \tau_y &\le \frac{\lambda \tau_x}{8}, \\
    2 \nu({\rm e}-1) \tau_x &\le \frac{\tau_x}{8}.
\end{align*}
Thus,
\begin{align*}
    \BE [\fV_{k+1}] &\le \fV_k - \frac{\tau_x}{8} \Vert \nabla g(x_k) \Vert^2 + \frac{7 \tau_x}{8} \Vert \nabla_x f(x_k,y_k) - \nabla g(x_k) \Vert^2 - \frac{\lambda \tau_x}{8} \Vert \nabla_y f(x_k,y_k) \Vert^2.
\end{align*}
Using Lemma \ref{lem: new-bound-Bk} and $\mu_y$-PL condition in $y$ and plugging in $\lambda = {14 L^2}/{\mu_y^2}$ yields
\begin{align*}
    \frac{7 \tau_x}{8} \Vert \nabla_x f(x_k,y_k) - \nabla g(x_k) \Vert^2  \le \frac{\lambda \tau_x}{16} \Vert \nabla_y f(x_k,y_k)\Vert^2. 
\end{align*}
Thus, 
\begin{align*}
    \BE[\fV_{k+1}] &\le \fV_k - \frac{\tau_x}{8} \Vert \nabla g(x_k) \Vert^2 - \frac{\lambda \tau_x}{16} \Vert \nabla_y f(x_k,y_k) \Vert^2.
\end{align*}
\end{proof}
\end{lem}

Now it is sufficient to show the convergence of SVRG-GDA.

\begin{thm}  \label{thm: SVRG-GDA}

Under Assumption \ref{asm: two-PL} and  \ref{asm: L-smooth},  if we  let $SM = \lceil {8}/{(\mu_x \tau_x)}  \rceil$, $T = \lceil \log(1/\epsilon) \rceil$ and $M,B,\tau_x,\tau_y$ defined in Lemma \ref{lem: SVRG-GDA-V_k}, then the following statement holds true for Algorithm \ref{alg: SVRG-GDA}:~
\begin{align*}
    \BE \left[\tilde \fA_{t+1}+ \frac{\lambda \tau_x}{\tau_y} \tilde \fB_{t+1} \right]\le \frac{1}{2}\left( \tilde \fA_t + \frac{\lambda \tau_x}{\tau_y} \tilde \fB_t \right),
\end{align*}
where $\tilde \fA_{t} = g(x_{t}) - g(x^{\ast}) $ and $\tilde \fB_{t} = g(x_{t}) - f(x_{t},y_{t})$. Furthermore, let $\alpha = {2}/{3}$, then it requires $\fO((n+  n^{2/3} \kappa_x \kappa_y^2) \log ({1}/{\epsilon}))$ stochastic first-order calls to achieve $g(x_T) - g(x^{\ast}) \le \epsilon$ in expectation.

\begin{proof}
By Lemma \ref{lem: SVRG-GDA-V_k} and Lemma \ref{lem: g(x)-PL} that $g$ satisfies $\mu_x$-PL in $x$ and Assumption \ref{asm: two-PL} that function $f$ satisfies $\mu_y$-PL in $y$, we have
\begin{align*}
    \BE[\mathcal{V}_{s,k+1}] \le \mathcal{V}_{s,k} - \frac{\mu_x \tau_x}{4} \mathcal{A}_{s,k} - \frac{\mu_y \tau_y}{8} \frac{\lambda \tau_x}{\tau_y} \mathcal{B}_{s,k} \le \mathcal{V}_{s,k} - \frac{\mu_x \tau_x}{4} \fV_{s,k},
\end{align*}
where in the last inequality we use $ {\mu_x \tau_x}/{4} \le {\mu_y \tau_y}/{8} $. Telescoping for $k=0,1,\dots,M-1$ and $s = 0,1,\dots,S-1$ and rearranging, we can see that in round $t$, it holds that
\begin{align*}
    \frac{1}{SM} \sum_{s=0}^{S-1} \sum_{k=0}^{M-1} \fV_{s,k} \le \frac{4}{\mu_x \tau_x SM} (\fV_{0,0} - \fV_{S,M}) \le \frac{1}{2} \fV_{0,0},
\end{align*}
where the last inequality is due to the choice of $S$.

The above inequality is exactly equivalent to what we want to prove:
\begin{align*}
    \BE \left[\tilde \fA_{t+1}+ \frac{\lambda \tau_x}{\tau_y} \tilde \fB_{t+1} \right]\le \frac{1}{2}\left( \tilde \fA_t + \frac{\lambda \tau_x}{\tau_y} \tilde \fB_t \right).
\end{align*}
Note that we have $M = \fO(n^{3\alpha/2})$ and $ S = \fO({\kappa_x \kappa_y^2}/{n^{\alpha/2}})$, then the complexity is
\begin{align*}
    \fO\left ( (n+SM + Sn ) \log \left(\frac{1}{\epsilon}\right) \right) = \fO \left( (n+ (n^{\alpha} + n^{1- \alpha/2})\kappa_x \kappa_y^2  ) \log\left(\frac{1}{\epsilon}\right)\right ). 
\end{align*}
Plugging in $\alpha = {2}/{3}$ yields the desired complexity and it can also be seen that it is also the best choice of $\alpha$.

\end{proof}
\end{thm}

\section{Proof of Section \ref{sec: SPIDER-for-PL}} \label{apx: SPIDER}

In this section, we show that SPIDER-type stochastic gradient estimators outperform SVRG-type estimators with complete proofs.
The following lemma controls the variance of gradient estimators using the recursive update formula as SPIDER.
\begin{lem}[{\citet[Modified form Lemma 1]{fang2018spider}}] \label{lem: spider-var}
In Algorithm \ref{alg: SPIDER-GDA},  it holds true that
{\small\begin{align*}
\mathbb{E}[ \Vert G_x(x_k,y_k) - \nabla_x f(x_k,y_k) \Vert^2] &\le \frac{L^2 }{B}\sum_{j=(n_k - 1)M}^k \left(\tau_x^2 \mathbb{E}[\Vert G_x(x_j,y_j) \Vert^2] + \tau_y^2 \mathbb{E}[\Vert G_y(x_j,y_j) \Vert^2] \right),\\
\mathbb{E}[ \Vert G_y(x_k,y_k) - \nabla_y f(x_k,y_k) \Vert^2] &\le \frac{L^2 }{B}\sum_{j=(n_k - 1)M}^k \left(\tau_x^2\mathbb{E}[\Vert G_x(x_j,y_j) \Vert^2] + \tau_y^2 \mathbb{E}[\Vert G_y(x_j,y_j) \Vert^2] \right), 
\end{align*}}
where  $n_k = \lceil k/M \rceil$ and $(n_k - 1)M \le k \le n_k M - 1$.
\end{lem}

We define the following Lyapunov function to measure the progress in each epoch:
\begin{align*}
    \tilde{\mathcal{V}}_{t} = \tilde {\mathcal{A}}_{t} + \frac{\lambda \tau_x}{\tau_y}\tilde{\mathcal{B}}_t,
\end{align*}
where $\tilde{\mathcal{A}}_t = g(\tilde x_t) - g(x^*)  $ and $\tilde{\mathcal{B}}_t = g(\tilde x_t) - f(\tilde x_t, \tilde y_t)$.

The following lemma describes the main convergence property of SPIDER-GDA.
\begin{lem} \label{lem: key-lem-for-Spider}
Under Assumption \ref{asm: one-PL} and \ref{asm: L-smooth}, setting all the parameters as defined in Theorem \ref{thm: SPIDER-GDA}, then it holds true that
\begin{align*}
    \BE \left[ \frac{\tau_x}{2} \Vert \nabla g(\tilde x_{t+1}) \Vert^2 + \frac{\lambda \tau_x}{4} \Vert \nabla_y f(\tilde x_{t+1},\tilde y_{t+1}) \Vert^2 \right] &\le \frac{1}{K}  \BE [ \tilde \fV_t ].
\end{align*}
\begin{proof}
First of all, we fix $t$ and analyze the inner loop. We define 
\begin{align*}
    \fV_{t,k} = \fA_{t,k} + \frac{\lambda \tau_x}{\tau_y} \fB_{t,k},
\end{align*}
where $\fA_{t,k} = g(x_{t,k}) - g(x^{\ast})$ and $\fB_{t,k} = g(x_{t,k}) - f(x_{t,k},y_{t,k}) $.

For simplification, we omit the subscripts $t$ when there is no ambiguity.
Note that $g(x)$ is $(2L^2/\mu_y)$-smooth, we have
\begin{align} \label{eq: bound-Ak-spider}
\begin{split}
    \mathbb{E} [g(x_{k+1}) ] &\le \mathbb{E}\left[ g(x_k) + \nabla g(x_k)^\top(x_{k+1} - x_k) + \frac{L^2}{\mu_y} \Vert x_{k+1} - x_k \Vert^2\right] \\
&= \mathbb{E}\left[ g(x_k)  - \tau_x\nabla g(x_k)^\top G_x(x_k,y_k) + \frac{L^2 \tau_x^2}{\mu_y} \Vert G_x(x_k,y_k) \Vert^2\right] \\
&= \mathbb{E}\left[ g(x_k)  - \frac{\tau_x}{2} \Vert \nabla g(x_k) \Vert^2 + \frac{\tau_x}{2} \Vert \nabla g(x_k) - G_x(x_k,y_k) \Vert^2 \right] \\
&\quad 
+\left(\frac{L^2 \tau_x^2}{\mu_y} - \frac{\tau_x}{2}\right) \mathbb{E} [\Vert G_x(x_k,y_k) \Vert^2] \\
&\le \mathbb{E} \left[ g(x_k)  - \frac{\tau_x}{2} \Vert \nabla g(x_k) \Vert^2 - \frac{\tau_x}{4} \Vert G_x(x_k,y_k) \Vert^2 \right] \\
&\quad +\mathbb{E}\left[ \tau_x \Vert \nabla g(x_k) - \nabla_x f(x_k,y_k) \Vert^2+ \tau_x \Vert  \nabla_x f(x_k,y_k) - G_x(x_k,y_k )\Vert^2\right].
\end{split}
\end{align}
where the last inequality uses  $\tau_x \le \mu_y / (4L^2)$.
Similarly, we can also show that
\begin{align*}
&\quad \mathbb{E}[ f(x_k,y_k)  - f(x_{k+1},y_k)] \\
&\le  \mathbb{E}\left[ - \nabla_x f(x_k,y_k)^\top (x_{k+1}-x_k) + \frac{L}{2} \Vert x_{k+1} - x_k \Vert^2\right] \\
&= \mathbb{E}\left[  \tau_x\nabla_x f(x_k,y_k)^\top G_x(x_k,y_k) + \frac{\tau_x^2 L}{2} \Vert G_x(x_k,y_k) \Vert^2\right] \\
&= \mathbb{E}\left[  \tau_x (\nabla_x f(x_k, y_k) - G_x(x_k,y_k))^\top G_x(x_k,y_k) +  \left(\frac{\tau_x^2L}{2} +\tau_x\right) \Vert G_x(x_k,y_k) \Vert^2\right] \\
&\le \mathbb{E}\left[ \frac{\tau_x}{2} \Vert  \nabla_x f(x_k,y_k)  - G_x(x_k,y_k) \Vert^2  +2 \tau_x \Vert G_x(x_k,y_k) \Vert^2\right],
\end{align*}
and 
\begin{align*}
&\quad\mathbb{E}[ f(x_{k+1},y_k) - f(x_{k+1},y_{k+1})] \\
&\le \mathbb{E}\left[ - \nabla_y f(x_{k+1},y_k)^\top (y_{k+1} - y_k) + \frac{L}{2} \Vert y_{k+1} - y_k \Vert^2\right] \\
&= \mathbb{E}\left[ - \tau_y\nabla_y f(x_{k+1},y_k)^\top G_y(x_k,y_k) + \frac{\tau_y^2L}{2} \Vert G_y(x_k,y_k) \Vert^2\right] \\
&=  \mathbb{E}\left[ -
\frac{\tau_y}{2} \Vert \nabla_y f(x_{k+1},y_k) \Vert^2 + \frac{\tau_y}{2} \Vert \nabla_y f(x_{k+1},y_k) - G_y(x_k,y_k) \Vert^2 \right] \\
&\quad 
 -\left( \frac{\tau_y}{2} - \frac{\tau_y^2L}{2} \right) \mathbb{E} \left[ \Vert G_y(x_k,y_k) \Vert^2\right] \\
&\le \mathbb{E}\left[ -
\frac{\tau_y}{4} \Vert \nabla_y f(x_{k},y_k) \Vert^2 -\left( \frac{\tau_y}{2} - \frac{\tau_y^2L}{2} \right)  \Vert G_y(x_k,y_k) \Vert^2  \right] \\
&\quad + \mathbb{E} \left[ \frac{\tau_y}{2} \Vert \nabla_y f(x_{k},y_k) - G_y(x_k,y_k) \Vert^2 + 2\tau_y \Vert \nabla_y f(x_{k+1},y_k) - \nabla_y f(x_k,y_k) \Vert^2 \right] \\
&\le \mathbb{E}\left[ -
\frac{\tau_y}{4} \Vert \nabla_y f(x_{k},y_k) \Vert^2 -\left( \frac{\tau_y}{2} - \frac{\tau_y^2L}{2} \right)  \Vert G_y(x_k,y_k) \Vert^2  \right] \\
&\quad + \mathbb{E} \left[ \frac{\tau_y}{2} \Vert \nabla_y f(x_{k},y_k) - G_y(x_k,y_k) \Vert^2 + 2\tau_y \tau_x^2 L^2 \Vert  G_x(x_k,y_k) \Vert^2 \right]
\end{align*}

Summing up the above two inequalities, we have
\begin{align*}
&\quad \mathbb{E} [f(x_k,y_k) - f(x_{k+1},y_{k+1})] \\
&\le \mathbb{E} \left[ - \frac{\tau_y}{4} \Vert \nabla_y f(x_k,y_k) \Vert^2 - \frac{\tau_y}{4} \Vert G_y(x_k,y_k) \Vert^2 + 3 \tau_x \Vert G_x (x_k,y_k) \Vert^2 \right] \\
&\quad + \mathbb{E} \left[ \frac{\tau_y}{2} \Vert \nabla_y f(x_k,y_k) - G_y(x_k,y_k) \Vert^2 + \frac{\tau_x}{2} \Vert \nabla_x f(x_k,y_k) - G_x(x_k,y_k) \Vert^2  \right] \\.
\end{align*}
Combing with inequality (\ref{eq: bound-Ak-spider}), it can be seen that
\begin{align*}
\mathbb{E}[\mathcal{B}_{k+1} ] &= \mathbb{E} [ g(x_{k+1}) -  f(x_{k+1},y_{k+1})] \\
&= \mathbb{E} [ g(x_{k+1} )- g(x_k) + g(x_k)  -f(x_k,y_k) + f(x_k,y_k)- f(x_{k+1},y_{k+1})] \\
&= \mathbb{E} \left[ \mathcal{B}_k   - \frac{\tau_y}{4} \Vert \nabla_y f(x_k,y_k) \Vert^2 - \frac{\tau_y}{4} \Vert G_y(x_k,y_k) \Vert^2 \right] \\
&\quad +\mathbb{E}  \left[  \tau_x \Vert \nabla g(x_k) - \nabla_x f(x_k,y_k) \Vert^2 +\frac{3\tau_x}{2} \Vert G_x(x_k,y_k )- \nabla_x f(x_k,y_k) \Vert^2 \right] \\
&\quad + \mathbb{E} \left[ \frac{\tau_y}{2} \Vert G_y(x_k,y_k) - \nabla_y f(x_k,y_k) \Vert^2 + 3 \tau_x \Vert G_x(x_k,y_k) \Vert^2\right].
\end{align*}
Therefore, using $24\lambda \tau_x \le \tau_y$ and inequality (\ref{eq: bound-Ak-spider}) again, we obtain
\begin{align*}
\mathbb{E}[ \fV_{k+1} ] &= \mathbb{E}\left[ \mathcal{A}_{k+1} + \frac{\lambda \tau_x}{\tau_y} \mathcal{B}_{k+1}\right] \\
&\le \BE\left[ \fA_k + \frac{\lambda \tau_x}{\tau_y} \fB_k - \frac{\tau_x}{2} \Vert \nabla g(x_k) \Vert^2 - \frac{\lambda \tau_x}{4} \Vert \nabla_y f(x_k,y_k) \Vert^2  \right] \\
&\quad + \mathbb{E}\left[ \left(  \frac{3\lambda \tau_x^2}{\tau_y} - \frac{\tau_x}{4} \right)  \Vert G_x(x_k,y_k) \Vert^2 - \frac{\lambda \tau_x}{4} \Vert G_y (x_k,y_k) \Vert^2  \right] \\
&\quad + \left( \tau_x + \frac{\lambda \tau_x}{\tau_y}\right) \mathbb{E} \left[ \Vert \nabla g(x_k) - \nabla_x f(x_k,y_k) \Vert^2\right] \\
&\quad + \left( \tau_x + \frac{3 \lambda \tau_x^2 }{2 \tau_y}  \right) \BE \left[ \Vert G_x(x_k,y_k) - \nabla_x f(x_k,y_k) \Vert^2 \right]\\
&\quad +  \frac{\lambda \tau_x}{2} \mathbb{E}\left[ \Vert G_y(x_k,y_k) - \nabla_y f(x_k,y_k) \Vert^2 \right] \\
&\le \mathbb{E}\left[ \mathcal{A}_k + \frac{\lambda \tau_x }{ \tau_y}\mathcal{B}_k - \frac{\tau_x}{2} \Vert \nabla g(x_k) \Vert^2 -\frac{\lambda \tau_x}{4} \Vert \nabla_y f(x_k,y_k) \Vert^2\right]\\
&\quad +  2 \tau_x \mathbb{E} \left[\Vert \nabla g(x_k) - \nabla_x f(x_k,y_k) \Vert^2\right] \\
&\quad +  \mathbb{E}\left[\frac{5\tau_x}{2} \Vert G_x(x_k,y_k )- \nabla_x f(x_k,y_k) \Vert^2 - \frac{\tau_x}{8} \Vert G_x(x_k,y_k) \Vert^2  \right] \\
&\quad + \mathbb{E}\left[ \frac{\lambda \tau_x}{2} \Vert G_y(x_k,y_k) - \nabla_y f(x_k,y_k) \Vert^2 - \frac{\lambda \tau_x}{4} \Vert G_y(x_k,y_k) \Vert^2
\right]. 
\end{align*}
Furthermore,
\begin{align*}
\mathbb{E}[\mathcal{V}_{k+1} ] &\le 
\mathbb{E}\left[ \mathcal{V}_k - \frac{\tau_x}{2} \Vert \nabla g(x_k) \Vert^2 -\frac{\lambda \tau_x}{4} \Vert \nabla_y f(x_k,y_k) \Vert^2 + \frac{2L^2 \tau_x}{\mu_y^2} \Vert \nabla_y f(x_k,y_k) \Vert^2\right]  \\
&\quad +  \mathbb{E}\left[\frac{5\tau_x}{2} \Vert G_x(x_k,y_k )- \nabla_x f(x_k,y_k) \Vert^2 - \frac{\tau_x}{8} \Vert G_x(x_k,y_k) \Vert^2\right] \\
&\quad + \mathbb{E}\left[ \lambda \tau_x \Vert G_y(x_k,y_k) - \nabla_y f(x_k,y_k) \Vert^2 - \frac{\lambda \tau_x}{4} \Vert G_y(x_k,y_k) \Vert^2\right] \\
&\le \mathbb{E} \left[ \mathcal{V}_k - \frac{\tau_x}{2} \Vert \nabla g(x_k) \Vert^2 -\frac{\lambda \tau_x}{4} \Vert \nabla_y f(x_k,y_k) \Vert^2 \right] \\
&\quad - \mathbb{E}\left[ \frac{\tau_x}{8} \Vert G_x(x_k,y_k) \Vert^2 + \frac{\lambda \tau_x}{8} \Vert G_y(x_k,y_k) \Vert^2 \right] \\
&\quad +  \mathbb{E}\left[  \frac{5\tau_x}{2} \Vert G_x(x_k,y_k )- \nabla_x f(x_k,y_k) \Vert^2 + \frac{9\lambda \tau_x}{8} \Vert G_y(x_k,y_k) - \nabla_y f(x_k,y_k) \Vert^2\right].
\end{align*}
Above, the first inequality follows from Lemma \ref{lem: new-bound-Bk} and the second inequality uses Young's inequality that $ \BE [\Vert a - b \Vert^2] \le \BE [ \Vert a \Vert^2 + \Vert b \Vert^2]  $ and  $\lambda = 32L^2 / \mu_y^2$.

Plug in the variance bound by Lemma \ref{lem: spider-var} and  $B  = M$, we have
\begin{align*}
\mathbb{E}[\mathcal{V}_{k+1}] &\le \mathbb{E}\left[ \mathcal{V}_k - \frac{\tau_x}{2} \Vert \nabla g(x_k) \Vert^2 - \frac{\lambda \tau_x}{4} \Vert \nabla_y f(x_k,y_k) \Vert^2 \right] \\
&\quad + \mathbb{E} \left[ \left(\frac{5 }{2} + \frac{9\lambda }{8}\right) \frac{\tau_x^3 L^2}{M} \sum_{j=(n_k-1)M}^k \Vert G_x(x_j,y_j) \Vert^2  - \frac{\tau_x}{8} \Vert G_x(x_k,y_k) \Vert^2\right] \\
&\quad + \mathbb{E}\left[ \left(\frac{5 }{2} + \frac{9\lambda }{8}\right) \frac{\tau_x \tau_y^2 L^2}{M} \sum_{j=(n_k-1)M}^k\Vert G_y(x_j,y_j) \Vert^2  - \frac{\lambda \tau_x}{8} \Vert G_y(x_k,y_k) \Vert^2\right] \\
&\le   \mathbb{E}\left[ \mathcal{V}_k - \frac{\tau_x}{2} \Vert \nabla g(x_k) \Vert^2 -\frac{\lambda \tau_x}{4} \Vert \nabla_y f(x_k,y_k) \Vert^2 \right] \\
&\quad+ \mathbb{E} \left[ \frac{5\lambda \tau_x^3  L^2}{4M} \sum_{j=(n_k-1)M}^k  \Vert G_x(x_j,y_j) \Vert^2  - \frac{\tau_x}{8} \Vert G_x(x_k,y_k) \Vert^2\right] \\
&\quad + \mathbb{E}\left[ \frac{5\lambda \tau_x \tau_y^2 L^2}{4M} \sum_{j=(n_k-1)M}^k \Vert G_y(x_j,y_j) \Vert^2  - \frac{\lambda \tau_x}{8} \Vert G_y(x_k,y_k) \Vert^2\right].
\end{align*}

Now we telescope for $i = (n_k-1)M, \cdots, k$.
\begin{align*}
\mathbb{E} [ \fV_{k+1}] &\le \mathbb{E}\left[ \mathcal{V}_{(n_k-1)M} - \frac{\tau_x}{2} \Vert \nabla g(x_k) \Vert^2 -\frac{\lambda \tau_x}{4} \Vert \nabla_y f(x_k,y_k) \Vert^2 \right] \\
&\quad+ \BE \left[ \sum_{i= (n_k-1)M}^k \left( \sum_{j=(n_k-1)M}^i \frac{5\lambda \tau_x^3  L^2}{4M}  \Vert G_x(x_j,y_j) \Vert^2  - \frac{\tau_x}{8} \Vert G_x(x_i,y_i) \Vert^2\right)\right] \\
&\quad + \mathbb{E}\left[  \sum_{i= (n_k-1)M}^k \left(  \sum_{j=(n_k-1)M}^i \frac{5\lambda \tau_x \tau_y^2 L^2}{4M} \Vert G_y(x_j,y_j) \Vert^2  - \frac{\lambda \tau_x}{8} \Vert G_y(x_i,y_i) \Vert^2\right)\right] \\
&\le \mathbb{E}\left[ \mathcal{V}_{(n_k-1)M} - \frac{\tau_x}{2} \Vert \nabla g(x_k) \Vert^2 -\frac{\lambda \tau_x}{4} \Vert \nabla_y f(x_k,y_k) \Vert^2 \right] \\
&\quad + \mathbb{E}  \left[\sum_{j=(n_k-1)M}^k  \left( \frac{5\lambda \tau_x^3  L^2}{4}  \Vert G_x(x_j,y_j) \Vert^2  - \frac{\tau_x}{8} \Vert G_x(x_j,y_j) \Vert^2\right)\right] \\
&\quad + \mathbb{E}\left[   \sum_{j=(n_k-1)M}^k  \left( \frac{5\lambda \tau_x \tau_y^2 L^2}{4}\Vert G_y(x_j,y_j) \Vert^2  - \frac{\lambda \tau_x}{8} \Vert G_y(x_j,y_j) \Vert^2\right)\right] \\
&\le \BE \left[\fV_{(n_k-1)M} - \frac{\tau_x}{2} \sum_{j=(n_k-1)M}^k \Vert \nabla g(x_k) \Vert^2 -\frac{\lambda \tau_x}{4} \sum_{j=(n_k-1)M}^k \Vert \nabla_y f(x_k,y_k) \Vert^2 \right] \\
&\quad 
- \mathbb{E} \left[\frac{\tau_x}{16} \sum_{j=(n_k-1)M}^k \Vert G_x(x_j,y_j) \Vert^2 +\frac{\lambda \tau_x}{16} \sum_{j=(n_k-1)M}^k \Vert G_y(x_j,y_j) \Vert^2 \right],
\end{align*}
where we use $ \lambda \tau_x^2L^2 \le 1/20 $ and $\tau_y^2L^2 \le 1/20$ in the last inequality.

From now on, we need to write down the subscripts with respect to $t$. Telescope for $k = 0,\cdots,K-1$ and drop the negative terms containing $ \Vert G_x (x_j,y_j) \Vert^2$ and $ \Vert G_y (x_j,y_j) \Vert^2$, we have
\begin{align*}
    \BE[\fV_{t,K}] \le \mathbb{E}\left[ \mathcal{V}_{t,0}   - \frac{\tau_x}{2} \sum_{k=0}^{K-1} \Vert \nabla g(x_k) \Vert^2 - \frac{\lambda \tau_x}{4} \sum_{k=0}^{K-1} \Vert \nabla_y f(x_k,y_k) \Vert^2\right].
\end{align*}
Rearranging this inequality concludes the proof.

\end{proof}
\end{lem}

Equipped with the above lemma, we can easily prove Theorem \ref{thm: SPIDER-GDA}.

\subsection{Proof of Theorem \ref{thm: SPIDER-GDA}}

\begin{proof}

Recall that $g(\,\cdot\,)$ is $\mu_x$-PL and $-f(x,\,\cdot\,)$ is $\mu_y$-PL, it holds that
\begin{align*}
    \Vert \nabla g(x) \Vert^2 &\ge 2 \mu_x ( g(x) - g(x^*) ) \\
    \Vert \nabla_y f(x,y) \Vert^2 &\ge 2 \mu_y (g(x) - f(x,y) ).
\end{align*}
Then Lemma \ref{lem: key-lem-for-Spider} implies
\begin{align*}
    \mu_x \tau_x  \BE \left[\tilde \fA_{t+1} +  \frac{\lambda \mu_y}{2 \mu_x} \tilde \fB_{t+1}\right]  \le \frac{1}{K} \BE\left[\tilde \fA_t + \frac{\lambda\tau_x}{\tau_y} \tilde \fB_t\right].
\end{align*}
By $ \mu_y \tau_y \ge 2 \mu_x \tau_x$, we further have
\begin{align*}
      \BE \left[\tilde \fA_{t+1} +  \frac{\lambda \tau_x}{\tau_y} \tilde \fB_{t+1}\right] \le \frac{1}{\mu_x \tau_x K} \BE \left[ \tilde \fA_t + \frac{\lambda \tau_x}{\tau_y} \fB_t \right] \le \frac{1}{2}  \BE \left[ \tilde \fA_t + \frac{\lambda \tau_x}{\tau_y} \fB_t \right].
\end{align*}

Therefore, to find $\hat x$ such that $g(\hat x) - g(x^{\ast}) \le \epsilon$ and $ g(\hat x) - f(\hat x,\hat y) \le 24\epsilon$ in expectation,  the complexity is
\begin{align*}
    \mathcal{O}\left( (n+MK ) \log \left(\frac{1}{\epsilon}\right)\right) = \mathcal{O} \left( (n+K \sqrt{n})\log\left( \frac{1}{\epsilon}\right)\right) = \mathcal{O} \left((n + \sqrt{n} \kappa_x \kappa_y^2) \log \left(\frac{1}{\epsilon}\right) \right)
\end{align*}
\end{proof}

\section{Proof of Section \ref{sec: Catalyst-for-acc}}

In this section, we present the convergence results of AccSPIDER-GDA when $\gamma = 0$, now  $F_k$ can be written as: 
\begin{align*}
    \min_{x \in \BR^{d_x}} \max_{y \in \BR^{d_y}} F_k(x,y) \triangleq f(x,y) + \frac{\beta}{2} \Vert x - x_k \Vert^2.
\end{align*}
For convenience, we also define the associated primal and dual functions as:
\begin{align*}
    G_k(x) \triangleq \max_{y \in \BR^{d_y}} F_k(x,y) \quad {\rm and} \quad H_k \triangleq \min_{x \in \BR^{d_x}} F_k(x,y).
\end{align*}
First of all, we take a closer look at $F_k$. The regularization term $\beta$ transforms the condition number of the problem.
\begin{lem} \label{lem: sub-probelm}
Given $\beta > L$, if $f$ is $L$-smooth and $\mu_y$-PL in $y$, then the sub-problem  $F_k(x,y)$ is $(\beta-L)$-strongly convex in $x$ , $\mu_y$-PL in $y$ and $(\beta+L)$-smooth.
\end{lem}

Strong duality also holds for the sub-problem $F_k(x,y)$ since its saddle point exist.
\begin{lem} \label{lem: dual-Fk}
Under Assumption \ref{asm: L-smooth} and \ref{asm: one-PL}, given $\beta > L$, the sub-problem $F_k(x,y)$ exists at least one saddle point $(\tilde x_{k+1}, \tilde y_{k+1})$.
\end{lem}

\begin{proof}
According to Assumption \ref{asm: one-PL}, the maximization problem $\max_{\vy \in \BR^{d_y}} f(x,y)$ has a nonempty solution set. 
It is clear that $F_k(x,y)$ is strongly convex in $x$ for $\beta >L$. Hence, we know that $G_k(x) = \max_{y \in \BR^{d_y}} F_k(x,y)$ is strongly convex since taking the supremum is an operation that preserves (strong) convexity. Therefore, for any $\tilde y_{k+1} \in \arg \max_{y \in \BR^{d_y}} f(x,y)$, there exist a unique $\tilde x_{k+1} = \arg \min_{\vx \in \BR^{d_x}} G_k(x)$. 
Above, the point $(x^*,y^*)$ is a unique global minimax point of $F_k(x,y)$. And the global minimax point of $F_k(x,y)$ is equivalent to a saddle point of $F_k(x,y)$ by Lemma \ref{lem: three-points-equal}.
\end{proof}

From now on, throughout this section, we always let $(\tilde x_{k+1}, \tilde y_{k+1})$ be the saddle point of $F_{k}$ for all $k \ge 0$. Next, we study the error brought by the inexact solution to the sub-problem. The idea is that when we can control the precision of $F_k$ with a global constant $\delta$, then the algorithm will be close to the exact proximal point algorithm. In the following analysis, we measure the precision of sub-problems using its duality gap:
\begin{align*}
    {\rm Gap}_{k+1} \triangleq \max_{y \in \BR^{d_y}} F_k(x_{k+1}, y) - \min_{x \in \BR^{d_x}} F_k(x, y_{k+1}). 
\end{align*}
First, let us present two consequences when the duality of of sub-problem is small.
\begin{lem} \label{lem: approx-func}
Suppose $x_{k+1}$ satisfies $(x_{k+1},y_{k+1})$ satisfies $\BE[{\rm Gap}_{k+1}]\le \delta $ for any saddle point $(\tilde x_{k+1}, \tilde y_{k+1})$ of $F_k$, then it holds that 
\begin{align*}
   \BE[G_k(x_{k+1}) - G_k(\tilde x_{k+1})] \le \delta.
\end{align*}
\end{lem}

\begin{proof}
 By the definition of the duality gap, we have
\begin{align*}
    {\rm Gap}_{k+1} =& \max_{y \in \BR^{d_y}} F_k(x_{k+1}, y) 
    - F_k(\tilde x_{k+1},\tilde y_{k+1}) + F_k(\tilde x_{k+1},\tilde y_{k+1})
    - \min_{x \in \BR^{d_x}} F_k(x, y_{k+1}) \\
    \ge & \max_{y \in \BR^{d_y}} F_k(x_{k+1}, y) 
    - F_k(\tilde x_{k+1},\tilde y_{k+1}) = G_k(x_{k+1}) - G_k(\tilde x_{k+1}).
\end{align*}
Hence, $\BE[{\rm Gap}_{k+1}]\le \delta $ implies that $\BE[G_k(x_{k+1}) - G_k(\tilde x_{k+1})] \le \delta$.
\end{proof}


\begin{lem} \label{lem: approx-dist}
Under Assumption \ref{asm: L-smooth} and \ref{asm: one-PL}, given $\beta > L$, 
suppose $x_{k+1}$ satisfies $(x_{k+1},y_{k+1})$ satisfies $\BE[{\rm Gap}_{k+1}]\le \delta $ for any saddle point $(\tilde x_{k+1}, \tilde y_{k+1})$ of $F_k$, then it holds that 
\begin{align*}
   \Vert x_{k+1} - \tilde x_{k+1} \Vert^2 \le \frac{2\delta}{\beta-L}.
\end{align*}

\end{lem}

\begin{proof}
Since $G_k(\vx)$ is $(\beta-L)$-strongly convex, we have
\begin{align*}
    \Vert x_{k+1} - \tilde x_{k+1} \Vert^2 \le \frac{2}{\beta-L} \left( G_k(x_{k+1}) - G_k(\tilde x_{k+1})  \right) \le \frac{2 \delta}{\beta-L},
\end{align*}
where the last step sues Lemma \ref{lem: approx-func}.


\end{proof}

When the sub-problem is solved precisely enough, we can show that $g(x)$ decreases in each iteration.

\begin{lem} \label{lem:g(x)-decrease}
Under Assumption \ref{asm: L-smooth} and \ref{asm: one-PL}, given $\beta > L$, 
suppose $(x_{k+1},y_{k+1})$ satisfies $\BE[{\rm Gap}_{k+1}]\le \delta $ for any saddle point $(\tilde x_{k+1}, \tilde y_{k+1})$ of $F_k$, then it holds that 
\begin{align*} 
\BE[g(x_{k+1}) - g(x^{\ast})]  &\le \BE \left[g(x_k) - g(x^{\ast}) -  \frac{\beta}{2} \Vert x_{k+1} - x_k \Vert^2 \right] + \delta.
\end{align*}
\end{lem}

\begin{proof}
Consider the following inequalities:
\begin{align*} 
\BE[g(x_{k+1}) - g(x^{\ast})] &=  \BE \left[G_k(x_{k+1}) - g(x^{\ast}) - \frac{\beta}{2} \Vert x_{k+1} - x_k \Vert^2\right]  \\
&\le \BE \left[G_k(\tilde x_{k+1}) - g(x^{\ast})- \frac{\beta}{2} \Vert x_{k+1} - x_k \Vert^2  \right] + \delta \\
&\le \BE \left[g(x_k) - g(x^{\ast}) -  \frac{\beta}{2} \Vert x_{k+1} - x_k \Vert^2 \right] + \delta, 
\end{align*}
where in the first inequality we use Lemma \ref{lem: approx-func}, the second inequality is because we know it holds that 
$G_k(\tilde x_{k+1}) \le G_k(x_k) = g(x_k)$.
\end{proof}
Now, we consider how $g(x_k)$ converges to $g(x^{\ast})$ when precision $\delta$ is obtained.

\begin{lem}  \label{lem: g(x)-convergence}
Under Assumption \ref{asm: L-smooth} and \ref{asm: one-PL}, given $\beta > L$, suppose $(x_{k+1},y_{k+1})$ satisfies $\BE[{\rm Gap}_{k+1}]\le \delta $ for any saddle point $(\tilde x_{k+1}, \tilde y_{k+1})$ of $F_k$, then it holds that 
\begin{align*}
    \BE[g(x_{k+1}) - g(x^{\ast})] \le \BE \left[g(x_k) - g(x^{\ast}) - \frac{1}{4 \beta} \Vert \nabla g(\tilde x_{k+1}) \Vert^2 \right] +  \left( 
    \frac{2\beta}{\beta -L} + 1
    \right)\delta.
\end{align*}
\end{lem}

\begin{proof}
The proof is based on Lemma \ref{lem:g(x)-decrease}:
\begin{align*}
    \BE[g(x_{k+1}) - g(x^{\ast})] 
    &\le \BE \left[g(x_k) - g(x^{\ast}) -  \frac{\beta}{2} \Vert x_{k+1} - x_k \Vert^2\right] + \delta \\
    &\le \BE \left[g(x_k) - g(x^{\ast}) - \frac{\beta}{4} \Vert \tilde x_{k+1} - x_k \Vert^2 + \frac{\beta}{2} \Vert x_{k+1} - \tilde  x_{k+1}\Vert^2\right] + \delta \\
    &= \BE \left[g(x_k) - g(x^{\ast}) - \frac{1}{4 \beta} \Vert \nabla g(\tilde x_{k+1}) \Vert^2\right] + \left( 
    \frac{2\beta}{\beta -L} + 1
    \right)\delta, 
\end{align*}
where the second inequality relies on the fact that $- \Vert a - b\Vert^2 \le - \Vert a \Vert^2/2 + \Vert b \Vert^2$,
In the last inequality we use Lemma \ref{lem: approx-dist} and the fact that $(\tilde x_{k+1}, \tilde y_{k+1})$ is also a stationary point by Lemma \ref{lem: three-points-equal}, which implies that $ \nabla g( \tilde x_{k+1}) + \beta ( \tilde x_{k+1} - x_k ) = 0$. 
\end{proof}

\subsection{Proof of Lemma \ref{lem: outer-converge}}

\begin{proof}
Noting that $g(x)$ satisfies $\mu_x$-PL by Lemma \ref{lem: g(x)-PL} and using the result of Lemma \ref{lem: g(x)-convergence}, we obtain 
\begin{align*}
    &\quad \BE[g(x_{k+1})] \\
    &\le \BE \left[g(x_k) - \frac{1}{4 \beta} \Vert \nabla g(\tilde x_{k+1}) \Vert^2 \right] + \left( 
    \frac{2\beta}{\beta -L} + 1
    \right)\delta \\
    &\le \BE \left[g(x_k) - \frac{1}{8\beta} \Vert \nabla g(x_{k+1}) \Vert^2  +
    \frac{1}{4 \beta} \Vert \nabla g(\tilde x_{k+1}) - \nabla g(x_{k+1}) \Vert^2  \right] +
    \left( 
    \frac{2\beta}{\beta -L} + 1
    \right)\delta \\
    &\le \BE \left[g(x_k) - \frac{\mu_x}{ 4\beta} (g(x_{k+1}) - g(x^{\ast})) +  \frac{1}{4 \beta} \Vert \nabla g(\tilde x_{k+1}) - \nabla g(x_{k+1}) \Vert^2   \right] + \left( 
    \frac{2\beta}{\beta -L} + 1
    \right)\delta. 
\end{align*}
Recalling Lemma \ref{lem: g(x)-L-smooth} that $g(\vx)$ is $(2L^2/\mu_y)$-smooth and Lemma~\ref{lem: approx-dist} that $ \BE \Vert x_{k+1} - \tilde x_{k+1} \Vert^2 \le 2 \delta/(\beta-L)$, we obtain 
\begin{align*}
    \BE[g(x_{k+1})] \le \BE \left[g(x_k) - \frac{\mu_x}{4 \beta}(g(x_{k+1}) - g(x^{\ast}))\right] + \left( \frac{2}{\beta -L} \left( \frac{L^4}{\beta \mu_y^2} +\beta \right) +1 \right) \delta.
\end{align*}
Subtracting $g(x^\ast)$ from both sides and rearranging the above inequality, we obtain
\begin{align*}
    &\quad \BE [g(x_{k+1}) - g(x^{\ast})] \\ &\le \BE \left[\left(1 - \frac{\mu_x}{4 \beta + \mu_x}\right) (g(x_k) - g(x^{\ast}))\right] + \left(1 - \frac{\mu_x}{4 \beta + \mu_x}\right) \left( \frac{2}{\beta -L} \left( \frac{L^4}{\beta \mu_y^2} +\beta \right) +1 \right) \delta. 
\end{align*}
Let $q \triangleq {\mu_x}/{(4\beta+ \mu_x)}$ and telescope, then we can obtain that
\begin{align*}
    &\quad \BE[g(x_k) - g(x^{\ast})] \\
    &\le (1- q)^k (g(x_0) - g(x^{\ast})) + q \left( \frac{2}{\beta -L} \left( \frac{L^4}{\beta \mu_y^2} +\beta \right) +1 \right) \delta  \sum_{i=0}^{k-1}(1-q)^i \\
    &\le (1-q)^k (g(x_0) - g(x^{\ast})) + \left( \frac{2}{\beta -L} \left( \frac{L^4}{\beta \mu_y^2} +\beta \right) +1 \right) \delta.
\end{align*}

Plugging in $\beta, \delta$ and using $\mu_y \le L$ yields the desired statement,
\end{proof}

Now, we show how to control the precision of the sub-problem $\delta$ recursively to satisfy the condition of Lemma \ref{lem: outer-converge} that $\BE[{\rm Gap}_{k+1}] \le \delta$ holds for all $k \ge 1$.

Before that, we need to following lemma to use ${\rm Gap}_k$ to upper-bound the duality gap at the point $(x_k,y_k)$ with respect to the sub-problem $F_{k}$:
\begin{align*}
    {\rm Gap}'_{k} \triangleq \max_{y \in \BR^{d_x}} F_k(x_k,y) - \min_{x \in \BR^{d_x}} F_k(x,y_k).
\end{align*}

\begin{lem} \label{lem: bound-tilde-xy}
Under Assumption \ref{asm: L-smooth} and \ref{asm: one-PL}, if we let $\beta > L$, then it holds that
\begin{align*}
    {\rm Gap}_k' \le 24 {\rm Gap}_k + \frac{16 \beta^2 L^2}{\mu_y^2 (\beta-L)^2} \Vert x_{k} - x_{k-1} \Vert^2.
\end{align*}
\end{lem}

\begin{proof}
We start with the following identity for ${\rm Gap}'_{k}$: 
\begin{align*}
    {\rm Gap}'_{k} =& \max_{y \in \BR^{d_x}} F_k(x_k,y) - \min_{x \in \BR^{d_x}} F_k(x,y_k) \\
    =&  \max_{y \in \BR^{d_x}} F_k(x_k,y) - \min_{x \in \BR^{d_x}} \max_{y \in \BR^{d_y}} F_k(x,y) +  \max_{y \in \BR^{d_y}} \min_{x \in \BR^{d_x}} F_k(x,y)  - \min_{x \in \BR^{d_x}} F_k(x,y_k) \\
    =& G_k(x_k) - \min_{x \in \BR^{d_x}} G_k(x) + \max_{y \in \BR^{d_y}} H_k(y) - H_k(y_k),
\end{align*}
where the strong duality is guaranteed by Lemma \ref{lem: dual-Fk}. Next, Lemma \ref{lem: g(x)-L-smooth} indicates that $G_k(x)$ is $(2 L^2/ \mu_y + \beta)$-smooth, and similarly, $H_k(y)$ is $ 2 \beta (\beta+L) / (\beta-L)$-smooth. Therefore, for any saddle point $(\tilde x_{k+1}, \tilde y_{k+1})$ of the sub-problem $F_k$, we have
\begin{align} \label{eq:ub-gap-k-prime}
    {\rm Gap}'_k \le  \left( \frac{L^2}{\mu_y} + \frac{\beta}{2} \right) \Vert x_k - \tilde x_{k+1} \Vert^2 +  \frac{\beta (\beta +L)}{\beta - L} \Vert y_k - \tilde y_{k+1} \Vert^2.
\end{align}
Hence, using Young's inequality, for any saddle point $(\tilde x_k, \tilde y_k)$ of the sub-problem $F_{k-1}$,  we obtain
\begin{align*}
     {\rm Gap}'_k \le &\left( \frac{2L^2}{\mu_y} + \beta \right) (\Vert x_k - \tilde x_{k} \Vert^2 + \Vert \tilde x_{k+1} - \tilde x_k \Vert^2) \\
     +  & \frac{2\beta (\beta +L)}{\beta - L} (\Vert y_k - \tilde y_{k} \Vert^2 + \Vert \tilde y_{k+1} - \tilde y_k \Vert^2)
\end{align*}
Next, we let $\tilde y_k$ be the projection of~$y_k$ onto the set $\arg \max_{y \in \BR^{d_y}} H_{k-1}(y)$ and let $\tilde x_k = \arg \min_{x \in \BR^{d_x}} H_{k-1}(x,\tilde y_k)$ be the unique solution to the associated strongly convex sub-problem. Since $-H_{k-1}(y) $ is $\mu_y$-PL by Lemma \ref{lem: g(x)-PL} and similarly $G_{k-1}(x)$ is $(\beta-L)$-strongly convex, we can apply the quadratic growth condition by Lemma~\ref{lem: quadratic-growth} to obtain 
\begin{align*}
    \frac{\beta - L}{2} \Vert x_k - \tilde x_k \Vert^2 \le & G_{k-1}(x_k) - G_{k-1}(\tilde x_k), \\
    \frac{\mu_y}{2} \Vert y_k - \tilde y_k \Vert^2 \le & H_{k-1}(\tilde y_k) - H_{k-1}(y_k).
\end{align*}
Since the strong duality indicates that for the saddle point $(\tilde x_k,\tilde y_k)$, it holds that $G_{k-1}(\tilde x_k) = H_{k-1}(\tilde y_k)$. Hence, summing up the above inequalities yields
\begin{align*}
     \frac{\beta - L}{2} \Vert x_k - \tilde x_k \Vert^2 +  \frac{\mu_y}{2} \Vert y_k - \tilde y_k \Vert^2 \le {\rm Gap}_k.
\end{align*}
Then, for $\beta =2L$, for any saddle point $(\tilde x_{k+1}, \tilde y_{k+1})$ of the sub-problem $F_k$, we obtain
\begin{align*}
    {\rm Gap}'_k \le & \frac{4L^2}{\mu_y}(\Vert x_k - \tilde x_{k} \Vert^2 + \Vert \tilde x_{k+1} - \tilde x_k \Vert^2) +   12 L (\Vert y_k - \tilde y_{k} \Vert^2 + \Vert \tilde y_{k+1} - \tilde y_k \Vert^2) \\
     \le & 24 {\rm Gap}_k + \frac{4L^2}{\mu_y} \Vert \tilde x_{k+1} - \tilde x_k \Vert^2 + 12L \Vert \tilde y_{k+1} - \tilde y_k \Vert^2.
\end{align*}
For $\tilde y_k$ as the projection onto the set $\arg \max_{y \in \BR^{d_y}} H_{k-1}(y)$ as defined previously and $\tilde x_k = \arg \min_{x \in \BR^{d_x}} H_{k-1}(x)$, we have 
\begin{align*}
 \tilde y_k \in \arg \max_{y \in \BR^{d_y}} H_{k-1}(y) = \arg \max_{y \in \BR^{d_y}} F_{k-1}(\tilde x_k, y) =  \arg \max_{y \in \BR^{d_y}} f(\tilde x_k, y)   
\end{align*}
Then,
we know from Lemma \ref{lem: prox-argmax} that there exists 
\begin{align*}
 \tilde y_{k+1} \in \arg \max_{y \in \BR^d} H_k(x,y) = \arg \max_{y \in \BR^d} F_k(\tilde x_{k+1},y) = \arg \max_{y \in \BR^{d}} f(\tilde x_{k+1}, y)   
\end{align*}
for $\tilde x_{k+1} = \arg \min_{x \in \BR^{d_x}} H_k(x,y)$ such that 
\begin{align*}
 \Vert \tilde y_{k+1} - \tilde y_k \Vert \le \frac{L^2}{\mu_y^2} \Vert \tilde x_{k+1} - \tilde x_k \Vert.
\end{align*}
Therefore, we have
\begin{align*}
    {\rm Gap}'_k \le 24 {\rm Gap}_k + \frac{16L^2}{\mu_y^2} \Vert \tilde x_{k+1}- \tilde x_k \Vert^2.
\end{align*}
By the first-order optimality conditions, we have 
\begin{align*}
    \nabla f(\tilde x_{k+1}) + \beta (\tilde x_{k+1} - x_k) =&0, \\
    \nabla f(\tilde x_k) + \beta (\tilde x_k - x_{k-1}) =&0,
\end{align*}
which means that 
\begin{align*}
    \Vert \tilde x_{k+1} - \tilde x_k \Vert \le  
    \Vert x_k - x_{k-1} \Vert + \frac{1}{\beta} \Vert \nabla f(\tilde x_{k+1}) - \nabla f(\tilde x_k) \Vert \le  \Vert x_k - x_{k-1} \Vert + \frac{L}{\beta} \Vert \tilde x_{k+1} - \tilde x_k \Vert.
\end{align*}
For $\beta > L$, we have $\Vert \tilde x_{k+1} -\tilde x_k \Vert \le \beta \Vert x_k - x_{k-1} \Vert / (\beta -L)$. Therefore, 
\begin{align*}
    {\rm Gap}_k' \le 24 {\rm Gap}_k + \frac{16 \beta^2 L^2}{\mu_y^2 (\beta-L)^2 } \Vert x_{k} - x_{k-1} \Vert^2.
\end{align*}
\end{proof}

An additional bound for the first iteration is required.

\begin{lem} \label{lem:induction-base}
Under Assumption \ref{asm: L-smooth} and \ref{asm: one-PL}, if we let $\beta > L$, then it holds that
\begin{align*}
    {\rm Gap}_0' \le & \left( \left( \frac{L^2}{\mu_y} + \frac{\beta}{2} \right) +   \frac{4L^2\beta (\beta +L)}{\mu_y^2 (2 \beta -L)(\beta - L)} \right) (g(x_0) - g^*) + \frac{4 \beta (\beta+L)}{\mu_y(\beta-L)} ( g(x_0) - f(x_0,y_0) ),
\end{align*}
where $g^* = \inf_{x \in \BR^{d_x}} g(x)$.
\end{lem}

\begin{proof}
Recalling (\ref{eq:ub-gap-k-prime}), for any saddle point $(\tilde x_1, \tilde y_1)$ of $F_0$, we have
\begin{align*}
     {\rm Gap}'_0 \le  \left( \frac{L^2}{\mu_y} + \frac{\beta}{2} \right) \Vert x_0 - \tilde x_{1} \Vert^2 +  \frac{\beta (\beta +L)}{\beta - L} \Vert y_0 - \tilde y_{1} \Vert^2.
\end{align*}
Next, we separately upper-bound $\Vert x_0 - \tilde x_{1} \Vert^2$ and $\Vert y_0 - \tilde y_{1} \Vert^2$ as follows. 

On the one hand, since $G_0(x)$ is $(\beta-L)$-strongly convex, we have
\begin{align*}
    \Vert x_0 - \tilde x_1 \Vert^2 \le& \frac{2}{\beta -L} \left( G_0(x_0) - G_0(\tilde x_1) \right) \\
    = & \frac{2}{\beta -L} \left( g(x_0) - g(\tilde x_1) - \frac{\beta}{2} \Vert \tilde x_1 - x_0 \Vert^2 \right) \\
    \le & \frac{2}{\beta-L} \left( g(x_0) - g^* \right) - \frac{\beta}{\beta -L} \Vert x_0  -\tilde x_1 \Vert^2,
\end{align*}
which gives the following upper bound of $\Vert x_0 - \tilde x_{1} \Vert^2$ after rearranging:
\begin{align*}
    \Vert x_0 - \tilde x_1 \Vert^2 \le \frac{2}{2 \beta -L} ( g(x_0) - g^*).
\end{align*}
On the other hand, for $y^*(x_0)$ as the projection of $y_0$ to the set $\arg \max_{y \in \BR^{d_y}} f(x_0,y)$, we have
\begin{align*}
    \Vert y_0 - \tilde y_1 \Vert^2 \le & 2 \Vert y_0 - y^*(x_0) \Vert^2 + 2 \Vert y^*(x_0) - \tilde y_1 \Vert^2 \\
    \le& 2 \Vert y_0 - y^*(x_0) \Vert^2 + \frac{2L^2}{\mu_y^2} \Vert x_0 - \tilde x_1 \Vert^2,
\end{align*}
where we use Lemma \ref{lem: prox-argmax} and the fact that $\tilde y_1 \in \arg \max_{y \in \BR^{d_y}} f(\tilde x_1,y)$. Next, by the upper bound of $\Vert x_0 - \tilde x_{1} \Vert^2$ and Lemma \ref{lem: quadratic-growth}, we have
\begin{align*}
    \Vert y_0 - \tilde y_1 \Vert^2 \le & \frac{4}{\mu_y} ( g(x_0) - f(x_0,y_0) ) + \frac{4L^2}{\mu_y^2(2 \beta-L)} (g(x_0) - g^*).
\end{align*}
Therefore, combining the upper bounds of $\Vert x_0 - \tilde x_{1} \Vert^2$ and $\Vert y_0 - \tilde y_{1} \Vert^2$, we obtain 
\begin{align*}
    {\rm Gap}_0' \le & \left( \left( \frac{L^2}{\mu_y} + \frac{\beta}{2} \right) +   \frac{4L^2\beta (\beta +L)}{\mu_y^2 (2 \beta -L)(\beta - L)} \right) (g(x_0) - g^*) + \frac{4 \beta (\beta+L)}{\mu_y(\beta-L)} ( g(x_0) - f(x_0,y_0) ).
\end{align*}
\end{proof}

\subsection{The Proof of Lemma \ref{sub:eps-saddle}}

\begin{proof}
Recall we solve the sub-problem:
\begin{align*}
    \max_{y \in \BR^{d_y}} \min_{x \in \BR^{d_x}} F_k(x,y) = - \min_{x \in \BR^{d_x}} \max_{y \in \BR^{d_y}} \{ - F_k(x,y) \}.
\end{align*}
It is $\mu_y$-PL in $y$ and $L$-strongly-convex in $x$ and thus clearly satisfies $L$-PL in $x$.

We define  $H_k(y) = \min_{x \in \BR_{d_x}} F_k(x,y)$. By Lemma \ref{lem: g(x)-PL} and \ref{lem: g(x)-L-smooth}, we know that $H_k(y)$ is also $\mu_y$-PL in $y$ and it is $12L$-smooth since $F_k$ is $3L$-smooth.

According to Theorem \ref{thm: SPIDER-GDA}, we know that SPIDER-GDA with precision $\delta_k$ makes sure
\begin{align*}
&\quad \underbrace{\mathbb{E}\left[  H_{k}(\tilde y_{k+1}) - H_{k}( y_{k+1}) + \frac{1}{24} \left( F_{k}(x_{k+1},y_{k+1}) -H_{k}(y_{k+1}) \right) \right]}_{\rm LHS} \\
&\le \delta_k \underbrace{\mathbb{E} \left[H_{k}(\tilde y_{k+1}) - H_{k}( y_{k}) + \frac{1}{24} \left( F_{k}(x_{k},y_{k}) -H_{k}(y_{k}) \right)\right]}_{\rm RHS}
\end{align*}
for any $\tilde y_{k+1} \in \arg \max_{y \in \BR^{d_y}} H_k(y)$.

Next, we use the duality gaps ${\rm Gap}_{k+1}$ and ${\rm Gap}_k'$ to give a lower bound of the left-hand side (LHS) and an upper bound of the right-hand side (RHS).

On the one hand, 
since $G_k(x)$ is $(4L^2 / \mu_y)$-smooth,
for $\tilde x_{k+1} = \arg \min_{x \in \BR^{d_x}} G_k(x)$,
we have
\begin{align*}
    {\rm Gap}_{k+1} =& G_{k}(x_{k+1}) - \min_{x \in \BR^{d_x}} G_k(x) + \max_{y \in \BR^{d_y}} H_k(y) - H_k(y_{k+1}) \\
    \le & \frac{2 L^2}{\mu_y} \Vert x_{k+1} - \tilde x_{k+1} \Vert^2 + \max_{y \in \BR^{d_y}} H_k(y) - H_k(y_{k+1})
\end{align*}
Let $x_k^*(y) = \arg \min_{x \in \BR^{d_x}} F_k(x,y)$, which is unique due to the strong convexity of $F_k(\,\cdot\,,y)$. Then, for any saddle point $(\tilde x_{k+1}, \tilde y_{k+1})$ of $F_k$, we have
\begin{align*}
&\quad \Vert x_{k+1} - \tilde x_{k+1} \Vert^2  \\
&\le 2 \Vert x_{k+1} - x_k^*(y_{k+1}) \Vert^2 +2  \Vert x_k^*(y_{k+1}) - \tilde x_{k+1} \Vert^2  \\
&\le 2\Vert x_{k+1} - x^*(y_{k+1}) \Vert^2  + 18 \mathbb{E}\Vert y_{k+1} - \tilde y_{k+1} \Vert^2  \\
&\le 
\frac{4}{L} (F_{k} (x_{k+1},y_{k+1}) - H_{k}(y_{k+1}) ) + \frac{36}{\mu_y} ( H_{k}(\tilde y_{k+1}) - H_{k}(y_{k+1})).
\end{align*}
Above,
the first inequality is simply the Young's inequality; the second inequality uses Lemma \ref{lem: prox-argmax}, the uniqueness of $x_k^*(y)$, and $\tilde x_{k+1} = x_k^*(\tilde y_{k+1})$; the last inequality uses the fact that $F_k(x,y)$ is $L$-strongly convex in $x$, $H_k(y)$ is $\mu_y$-PL in $y$, and Lemma \ref{lem: quadratic-growth} for $\tilde y_{k+1}$ be the projection of $y_{k+1}$ onto the set $\arg \min_{y \in \BR^{d_y}} H_k(y)$.

Therefore, by simple calculations, we can lower-bound of LHS as
\begin{align*}
    {\rm Gap}_{k+1} \le \frac{192 L^2}{\mu_y^2} {\rm LHS}.
\end{align*}

On the other side, for the right-hand side (RHS), we can similarly upper-bound it using 
\begin{align*}
{\rm RHS} =& H_{k}(\tilde y_{k+1}) - H_{k}( y_{k}) + \frac{1}{24} \left( F_{k}(x_{k},y_{k}) -H_{k}(y_{k}) \right) \\
\le& H_{k}(\tilde y_{k+1}) - H_{k}( y_{k}) + \frac{L}{48} \Vert x_k - \tilde x_{k+1} \Vert^2 \\
\le& H_{k}(\tilde y_{k+1}) - H_{k}( y_{k}) + \frac{1}{24} (G_k(x_k) - G_k(\tilde x_{k+1}))
\end{align*}
Above, the first inequality uses the $L$-smoothness of $F_k(x,y)$ and holds for any saddle point $(\tilde x_{k+1}, \tilde y_{k+1})$ of $F_k$; the second inequality follows from Young's inequality; the third inequality holds because $G_k(x)$ is $L$-strongly convex.

Therefore, after algebraic calculations, we can simplify the above bound to
\begin{align*}
    {\rm RHS} \le \frac{25}{24} {\rm Gap}_k'.
\end{align*}
Combining both bounds for LHS and RHS and taking the expectation, we now obtain
\begin{align*}
    \BE[{\rm Gap}_{k+1}] \le \frac{192 L^2}{\mu_y^2} \BE[{\rm LHS}] \le \frac{192 L^2 \delta_k}{\mu_y^2} \BE[{\rm RHS}] \le \frac{200 L^2 \delta_k}{\mu_y^2} \BE[{\rm Gap}_k'].
\end{align*}
\end{proof}

Now it is sufficient to control $\delta$ recursively.

\begin{lem} \label{lem: two-side-error}
Under the same setting of Theorem \ref{thm: Catalyst-GDA},
if we solve each sub-problem $F_k$ with precision~$\delta_k$ as defined in Theorem \ref{thm: Catalyst-GDA}, then for all $k$ it holds that $\BE[{\rm Gap}_k] \le \delta$.
\begin{proof}

We prove by induction. Suppose we the following statement holds true for all $1 \le k' \le k$ that we have $\BE [{\rm Gap}_{k}'] \le \delta$.
Then, by Lemma \ref{sub:eps-saddle}, for $\delta_k' = 200L^2 \delta_k/\mu_y^2$, we have
\begin{align*}
&\quad \BE [{\rm Gap}_{k+1}] \le \delta_k' \BE[{\rm Gap}_k'] \\
&\le 24 \delta_k' \BE[ {\rm Gap}_k] + \frac{64 L^2 \delta_k'}{\mu_y^2} \Vert x_k - x_{k-1} \Vert^2 \\
&\le  24 \delta_k' \delta + \frac{64 L^2 \delta_k'}{\mu_y^2} \Vert x_k - x_{k-1} \Vert^2,
\end{align*}
where the second line uses Lemma \ref{lem: bound-tilde-xy} with $\beta =2L$ and the last line uses the induction hypothesis and.
Note that our choice of $\delta_k$ satisfies
\begin{align*}
    \max \left\{24 \delta_k' \delta, \frac{64 L^2 \delta_k'}{\mu_y^2} \Vert x_k - x_{k-1} \Vert^2\right \} \le \frac{\delta}{2}.
\end{align*}
Therefore, we have $\BE[{\rm Gap}_{k+1}] \le \delta$ and complete the induction from $k$ to $k+1$. For the induction base, we use Lemma \ref{lem:induction-base} with $\beta = 2L$ to obtain
\begin{align*}
    &\BE[{\rm Gap_1}] \le \delta_k' \BE[{\rm Gap}_0' ] \\
    \le &  \frac{10 \delta_k' L^2}{\mu_y^2} (g(x_0) - g(x^*)) + \frac{12 \delta_k'L}{\mu_y} ( g(x_0) - f(x_0,y_0)). 
\end{align*}
Therefore, the induction base $\BE[{\rm Gap}_1]$ can be satisfied by setting $\delta_k'$ ensuring that the right-hand side above is no larger than $\delta$.

%
\end{proof}

\end{lem}

\subsection{Proof of Theorem \ref{thm: Catalyst-GDA}}

Combing Lemma \ref{lem: outer-converge}, Lemma \ref{sub:eps-saddle} and Lemma \ref{lem: two-side-error}, we can easily prove Theorem \ref{thm: Catalyst-GDA}.

\begin{proof}
Note that each sub-problem $F_k$ is $3L$-smooth, $L$-PL in $x$ and $\mu_y$-PL in $y$ for $\beta  =2L$. Now if we choose  
\begin{align*}
    K = \left \lceil ((2 \beta+ \mu_x)/ \mu_x) \log (2/\epsilon) \right \rceil = \fO( \kappa_x \log (1/\epsilon) ),
\end{align*} 
then by Lemma \ref{lem: outer-converge} it is sufficient to guarantee that $\BE[g(x_K) - g(x^*)] \le \epsilon$,  while solving each sub-problem $F_k$ requires no more than $T_k \le a (n+ \sqrt{n } \kappa_y )  \log({\kappa_y}/{\delta_k})$ first-order oracle calls in expectation by Lemma \ref{sub:eps-saddle}, where $a$ is an independent positive constant. 

Now we telescope the inequality in Lemma \ref{lem:g(x)-decrease} and we can obtain
\begin{align} \label{eq: one-side-tele}
    \sum_{k=0}^{K-1} \BE \left[\Vert x_{k+1} - x_k \Vert^2\right] &\le \frac{2}{\beta} (g(x_0) - g^{\ast}) + \frac{2\delta}{\beta}.
\end{align}
Note that we have
\begin{align*}
    \frac{1}{\delta_k} \le \omega \times \max\left\{ 24, \frac{64 \kappa_y^2 \Vert x_k - x_{k-1} \Vert^2}{\delta } \right\} \le \omega \times \left(24 + \frac{64 \kappa_y \Vert x_k - x_{k-1} \Vert^2}{\delta }\right),
\end{align*}
by the choice of $\delta_k$ for all $k \ge 1$ in (\ref{dfn:delta_k}), where  $\omega = 400 \kappa_y^2$. Denote $C = a(n+ \sqrt{n}\kappa)$ for a numerical constant $a>0$ and $\fV_0 = g(x_0) - g(x^*) + g(x_0) - f(x,y_0)$, then we have
\begin{align}
\begin{split} \label{SFO:xk12}
&\quad \sum_{k=0}^{K} \BE[T_k]  = \BE[T_0] + \sum_{k=1}^{K} \BE[T_k] \\
 &\le C  \log \left( \omega \times \frac{6 \kappa_y^2 \fV_0}{ \delta}\right ) + C \sum_{k=1}^{K} \BE \left[\log \left( \frac{\kappa_y}{\delta_k} \right) \right] \\
&\le C  \log \left(  \omega \times \frac{6 \kappa_y^2 \fV_0}{ \delta} \right ) + C \sum_{k=1}^{K} \BE \left[\log \left( \omega \times\left(  24 \kappa_y + \frac{64 \kappa_y^3 \Vert x_k - x_{k-1} \Vert^2}{\delta} \right)\right) \right]\\
&\le  C  \log \left( \omega \times \frac{6 \kappa_y^2 \fV_0}{ \delta} \right ) +    C K  \log \left( \omega \times\sum_{k=1}^K \left(  24 \kappa_y + \frac{64 \kappa_y^3 \BE \Vert x_k - x_{k-1} \Vert^2}{\delta} \right)\right)   \\
&= C  \log \left(  \omega \times \frac{6 \kappa_y^2 \fV_0}{ \delta} \right ) +    C K  \log \left(\omega \times \sum_{k=0}^{K-1} \left(  24 \kappa_y + \frac{64 \kappa_y^3 \BE \Vert x_{k+1} - x_{k} \Vert^2}{\delta} \right)\right),
\end{split}
\end{align}
where the second inequality relies on the choice of $\delta_k$; the third inequality is  due to Jensen's inequality and the convexity of $\log(\,\cdot\,)$ function.

 
Lastly, we use (\ref{eq: one-side-tele}) and notice that $\delta$ defined in (\ref{dfn:delta-two}) is dependent on both $\epsilon$ and $\kappa_y$, while $\omega$ is dependent on $\kappa_y$ to show the SFO complexity (in expectation) of the order
\begin{align*}
    \fO ( (n \kappa_x + \sqrt{n} \kappa_x \kappa_y) \log (1/\epsilon) \log(\kappa_y / \epsilon)).
\end{align*}

\end{proof}

\section{Proof of Section \ref{sec: extension}} \label{apx: one-side}

First of all, we show the convergence of SVRG-GDA and SPIDER-GDA under one-sided PL condition as studied in Section \ref{sec: extension}. We reuse the lemmas under the two-sided PL condition. It is worth noticing that we can discard the outermost loop with respect to the restart strategy for both SVRG-GDA and SPIDER-GDA, i.e we set $T=1$ in this setting.

\subsection{SVRG-GDA under one-sided PL condition}

For SVRG-GDA, we have the following theorem. 
\begin{thm} \label{thm: SVRG-GDA-one}
Under Assumption \ref{asm: one-PL} and \ref{asm: L-smooth}, let $T = 1$ and $M,\tau_x,\tau_y,\lambda$  defined in Lemma \ref{lem: SVRG-GDA-V_k}; $\alpha = {2}/{3}$, $ SM = \lceil {8 }/{(\tau_x \epsilon^2)} \rceil$. Algorithm \ref{alg: Catalyst-GDA} can guarantee the output $\hat x$ to satisfy  $\Vert \nabla g(\hat x) \Vert^2 \le \epsilon$ in expectation with no more than $\fO(n+ { n^{2/3} \kappa_y^2 L}{\epsilon^{-2}})$ stochastic first-order oracle calls.

\begin{proof}

Telescoping for $k = 0,\dots, M-1$ and $s= 0,\dots, S-1$  for the inequality in Lemma \ref{lem: SVRG-GDA-V_k}:
\begin{align*}
    \frac{1}{SM} \sum_{s=0}^{S-1} \sum_{k=0}^{M-1} \BE [\Vert \nabla g(x_{s,k}) \Vert^2] \le \frac{8 \fV_{0,0}}{\tau_x SM}.
\end{align*}
Note that we have $M = \fO(n^{3\alpha/2})$ and if we let $ SM = \lceil {8 }/{(\tau_x \epsilon^2)} \rceil$, then $S = \fO( {L \kappa_y^2 }/({n^{\alpha/2}\epsilon^2)})$, so the complexity is
\begin{align*}
    \fO( n+ SM + Sn) = \fO\left( n+ \frac{\kappa_y^2 L ( n^{\alpha} + n^{1- \alpha/2})}{\epsilon^2} \right).
\end{align*}
Plugging in $\alpha = 2/3$ yields the desired complexity.
\end{proof}

\end{thm}

\subsection{Proof of Theorem \ref{thm: SPIDER-GDA-one}}

Similarly to SVRG-GDA, we can also analyze the convergence of SPIDER-GDA.
\begin{proof}
By Lemma \ref{lem: key-lem-for-Spider}, the output satisfies $ \BE[ \Vert \nabla g(\hat x) \Vert] \le \epsilon$ under our choice of parameters.

Since $\tau_x = \fO(1/ (\kappa_y^2L) )$ and $M = B = \sqrt{n}$, the complexity becomes:
\begin{align*}
\fO\left( n + \frac{\sqrt{n}}{\tau_x \epsilon^2}\right)    = \fO\left( n + \frac{\sqrt{n} \kappa_y^2L}{\epsilon^2}\right).
\end{align*}
\end{proof}

Above, we have show that the complexity of SVRG-GDA is $ \fO( n+ n^{2/3} \kappa_y^2 L \epsilon^{-2}) $ and the complexity of SPIDER-GDA is $\fO(n + \sqrt{n} \kappa_y^2 L \epsilon^{-2})$ \footnote{ To be more precise, our theorem only suits the case when $ 1/ (\kappa_y^2 L \epsilon^2 ) > \sqrt{n}$ for SPIDER-GDA. If not, we can directly set $K = 2M$ to achieve the same convergence result.}. Thus, we can come to the conclusion that SPIDER-GDA strictly outperforms SVRG-GDA under both the two-sided and one-sided PL conditions. In the rest of this section, we mainly focus on the complexity of AccSPIDER-GDA under the one-sided PL condition. 

In the following lemma, we show that AccSPIDER-GDA converges when we can control the precision of solving each sub-problem with a global constant $\delta$.

\subsection{Proof of Lemma \ref{lem: outer-convergence-one-side}}
\begin{proof}
Similar to the proof under the two-sided PL condition, we begin our proof with Lemma \ref{lem: g(x)-convergence}. We can see that
\begin{align*}
    &\quad \BE[g(x_{k+1})] \\
    &\le \BE \left[g(x_k)  - \frac{1}{4 \beta} \Vert \nabla g(\tilde x_{k+1}) \Vert^2 + \left( \frac{2 \beta}{\beta -L} +1 \right) \delta \right]\\ 
    &\le \BE \left[g(x_k) - \frac{1}{8\beta} \Vert \nabla g(x_{k+1}) \Vert^2  +
    \frac{1}{4 \beta} \Vert \nabla g(\tilde x_{k+1}) - \nabla g(x_{k+1}) \Vert^2  \right] +
    \left( 
    \frac{2\beta}{\beta -L} + 1
    \right)\delta \\
    &\le  \BE \left[g(x_k) - \frac{1}{8\beta} \Vert \nabla g(x_{k+1}) \Vert^2 \right] +
    \left( \frac{2}{\beta -L} \left( \frac{L^4}{\beta \mu_y^2} +\beta \right) +1 \right) \delta,
\end{align*}
where we use the fact that $-\Vert a - b \Vert^2 \le - \Vert a \Vert^2/2 +\Vert b \Vert^2$ in the second inequality, $g(x)$ is $({2L^2}/{\mu_y})$-smooth by Lemma \ref{lem: g(x)-L-smooth} in the third one and $ \BE \Vert x_{k+1} - \tilde x_{k+1} \Vert^2 \le 2 \delta/(\beta-L)$ by Lemma~\ref{lem: approx-dist} in the last one.
Telescoping for $k = 0,1,2,...K-1$, we can see that 
\begin{align*}
    \frac{1}{8 \beta} \sum_{k=0}^{K-1} \BE \left[\Vert \nabla g(x_k) \Vert^2\right] &\le \BE \left[g(x_0) - g(x_K) \right] + \left( \frac{2}{\beta -L} \left( \frac{L^4}{\beta \mu_y^2} +\beta \right) +1 \right) K\delta \\
    &\le  \BE \left[g(x_0) - g^{\ast}\right] + \left( \frac{2}{\beta -L} \left( \frac{L^4}{\beta \mu_y^2} +\beta \right) +1 \right) K\delta.
\end{align*}
Divide  both sides by $K/ (8 \beta)$, then
\begin{align*}
    \frac{1}{K} \sum_{k=0}^{K-1} \BE \left[ \Vert \nabla g(x_k) \Vert^2\right] 
    &\le  \BE \left[\frac{8 \beta(g(x_0) - g^{\ast})}{K}\right] + \left( \frac{16 \beta}{\beta -L} \left( \frac{L^4}{\beta \mu_y^2} +\beta \right) +8 \beta \right) \delta.
\end{align*}
Plugging the choice of $\delta$ and $\beta =2L$ yields the desired inequality.
\end{proof}

\subsection{Proof of Theorem \ref{thm: Catalyst-GDA-one-side}}

Combing Lemma \ref{lem: outer-convergence-one-side}, Lemma \ref{sub:eps-saddle} and Lemma \ref{lem: two-side-error}, we can easily prove Theorem \ref{thm: Catalyst-GDA-one-side}. We remark that both the proof of Lemma \ref{sub:eps-saddle} and Lemma \ref{lem: two-side-error} only use the PL property in the direction of $y$, so they can both be directly applied to the one-sided PL case.

\begin{proof}
Note that each sub-problem $F_k$ is $3L$-smooth, $L$-PL in $x$ and $\mu_y$-PL in $y$ for $\beta  =2L$. Now, if we choose  
\begin{align*}
    K = \left \lceil 16 \beta (g(x_0) - g^*)/ \epsilon^2 \right \rceil = \fO(L \epsilon^{-2} ),
\end{align*} 
then by Lemma \ref{lem: outer-convergence-one-side} it is sufficient to guarantee that $\BE[ \Vert g(\hat x) \Vert] \le \epsilon$,  while solving each sub-problem $F_k$ requires no more than $T_k \le a (n+ \sqrt{n } \kappa_y )  \log({\kappa_y}/{\delta_k})$ first-order oracle calls in expectation by Lemma \ref{sub:eps-saddle}, where $a$ is an independent positive constant. 

Therefore,  using (\ref{eq: one-side-tele}), (\ref{SFO:xk12}) and noticing that $\delta$ defined in (\ref{dfn:delta-one}) is dependent on both $\epsilon$ and $\kappa_y$, awhile $\omega$ in (\ref{SFO:xk12}) is dependent on $\kappa_y$ to show the SFO complexity of the order
\begin{align*}
    \fO ( (n  + \sqrt{n}  \kappa_y) L \epsilon^{-2} \log(\kappa_y / \epsilon)).
\end{align*}

\end{proof}



%% file: reference.bib
@article{han2021lower,
  title={Lower Complexity Bounds of Finite-Sum Optimization Problems: The Results and Construction},
  author={Han, Yuze and Xie, Guangzeng and Zhang, Zhihua},
  journal={arXiv preprint arXiv:2103.08280},
  year={2021}
}

@inproceedings{yang2020global,
  title={Global convergence and variance reduction for a class of nonconvex-nonconcave minimax problems},
  author={Yang, Junchi and Kiyavash, Negar and He, Niao},
  booktitle={NeurIPS},
  year={2020}
}

@inproceedings{lin2015universal,
  title={A universal catalyst for first-order optimization},
  author={Lin, Hongzhou and Mairal, Julien and Harchaoui, Zaid},
  booktitle={NIPS},
  year={2015}
}

@book{nesterov2018lectures,
  title={Lectures on convex optimization},
  author={Nesterov, Yurii },
  volume={137},
  year={2018},
  publisher={Springer}
}

@inproceedings{nouiehed2019solving,
  title={Solving a class of non-convex min-max games using iterative first order methods},
  author={Nouiehed, Maher and Sanjabi, Maziar and Huang, Tianjian and Lee, Jason D. and Razaviyayn, Meisam},
  booktitle={NeurIPS},
  year={2019}
}

@inproceedings{xian2021faster,
  title={A faster decentralized algorithm for nonconvex minimax problems},
  author={Xian, Wenhan and Huang, Feihu and Zhang, Yanfu and Huang, Heng},
  booktitle={NeurIPS},
  year={2021}
}

@article{liu2022loss,
  title={Loss landscapes and optimization in over-parameterized non-linear systems and neural networks},
  author={Liu, Chaoyue and Zhu, Libin and Belkin, Mikhail},
  journal={Applied and Computational Harmonic Analysis},
  year={2022},
  publisher={Elsevier}
}

@article{sun2018geometric,
  title={A geometric analysis of phase retrieval},
  author={Sun, Ju and Qu, Qing and Wright, John},
  journal={Foundations of Computational Mathematics},
  volume={18},
  number={5},
  pages={1131--1198},
  year={2018},
  publisher={Springer}
}

@article{cai2019global,
  title={On the global convergence of imitation learning: A case for linear quadratic regulator},
  author={Cai, Qi and Hong, Mingyi and Chen, Yongxin and Wang, Zhaoran},
  journal={arXiv preprint arXiv:1901.03674},
  year={2019}
}

@inproceedings{guo2020communication,
  title={Communication-efficient distributed stochastic auc maximization with deep neural networks},
  author={Guo, Zhishuai and Liu, Mingrui and Yuan, Zhuoning and Shen, Li and Liu, Wei and Yang, Tianbao},
  booktitle={ICML},
  year={2020},
}

@article{duchi2019variance,
  title={Variance-based regularization with convex objectives},
  author={Duchi, John and Namkoong, Hongseok},
  journal={The Journal of Machine Learning Research},
  volume={20},
  number={1},
  pages={2450--2504},
  year={2019},
}

@inproceedings{carmon2019variance,
  title={Variance reduction for matrix games},
  author={Carmon, Yair and Jin, Yujia and Sidford, Aaron and Tian, Kevin},
  booktitle={NeurIPS},
  year={2019}
}

@inproceedings{du2017stochastic,
  title={Stochastic variance reduction methods for policy evaluation},
  author={Du, Simon S. and Chen, Jianshu and Li, Lihong and Xiao, Lin and Zhou, Dengyong},
  booktitle={ICML},
  year={2017}
}

@inproceedings{wai2018multi,
  title={Multi-agent reinforcement learning via double averaging primal-dual optimization},
  author={Wai, Hoi-To and Yang, Zhuoran and Wang, Zhaoran and Hong, Mingyi},
  booktitle={NeurIPS},
  year={2018}
}

@article{polyak1963gradient,
  title={Gradient methods for minimizing functionals},
  author={Polyak, Boris Teodorovich},
  journal={Zhurnal Vychislitel'noi Matematiki i Matematicheskoi Fiziki},
  volume={3},
  number={4},
  pages={643--653},
  year={1963},
  publisher={Russian Academy of Sciences, Branch of Mathematical Sciences}
}

@inproceedings{ying2016stochastic,
  title={Stochastic online {AUC} maximization},
  author={Ying, Yiming and Wen, Longyin and Lyu, Siwei},
  booktitle={NIPS},
  year={2016}
}

@article{liu2019stochastic,
  title={Stochastic auc maximization with deep neural networks},
  author={Liu, Mingrui and Yuan, Zhuoning and Ying, Yiming and Yang, Tianbao},
  journal={arXiv preprint arXiv:1908.10831},
  year={2019}
}

@inproceedings{karimi2016linear,
  title={Linear convergence of gradient and proximal-gradient methods under the polyak-{\l}ojasiewicz condition},
  author={Karimi, Hamed and Nutini, Julie and Schmidt, Mark},
  booktitle={Joint European Conference on Machine Learning and Knowledge Discovery in Databases},
  year={2016},
  organization={Springer}
}

@inproceedings{meinshausen2018causality,
  title={Causality from a distributional robustness point of view},
  author={Meinshausen, Nicolai},
  booktitle={DSW},
  year={2018},
  organization={IEEE}
}

@inproceedings{zhang2013linear,
  title={Linear convergence with condition number independent access of full gradients},
  author={Zhang, Lijun and Mahdavi, Mehrdad and Jin, Rong},
  booktitle={NIPS},
  year={2013}
}

@inproceedings{johnson2013accelerating,
  title={Accelerating stochastic gradient descent using predictive variance reduction},
  author={Johnson, Rie and Zhang, Tong},
  booktitle={NIPS},
  year={2013}
}

@inproceedings{allen2016variance,
  title={Variance reduction for faster non-convex optimization},
  author={Allen-Zhu, Zeyuan and Hazan, Elad},
  booktitle={ICML},
  year={2016}
}

@inproceedings{reddi2016stochastic,
  title={Stochastic variance reduction for nonconvex optimization},
  author={Reddi, Sashank J. and Hefny, Ahmed and Sra, Suvrit and Poczos, Barnabas and Smola, Alex},
  booktitle={ICML},
  year={2016}
}

@inproceedings{allen2016improved,
  title={Improved SVRG for non-strongly-convex or sum-of-non-convex objectives},
  author={Allen-Zhu, Zeyuan and Yuan, Yang},
  booktitle={ICML},
  year={2016}
}

@inproceedings{defazio2014saga,
  title={{SAGA}: A fast incremental gradient method with support for non-strongly convex composite objectives},
  author={Defazio, Aaron and Bach, Francis and Lacoste-Julien, Simon},
  booktitle={NIPS},
  year={2014}
}

@article{schmidt2017minimizing,
  title={Minimizing finite sums with the stochastic average gradient},
  author={Schmidt, Mark and Le Roux, Nicolas and Bach, Francis},
  journal={Mathematical Programming},
  volume={162},
  number={1},
  pages={83--112},
  year={2017},
  publisher={Springer}
}

@article{shalev2013stochastic,
  title={Stochastic dual coordinate ascent methods for regularized loss minimization.},
  author={Shalev-Shwartz, Shai and Zhang, Tong},
  journal={Journal of Machine Learning Research},
  volume={14},
  number={2},
  year={2013}
}

@article{mairal2015incremental,
  title={Incremental majorization-minimization optimization with application to large-scale machine learning},
  author={Mairal, Julien},
  journal={SIAM Journal on Optimization},
  volume={25},
  number={2},
  pages={829--855},
  year={2015},
  publisher={SIAM}
}

@inproceedings{fang2018spider,
  title={Spider: Near-optimal non-convex optimization via stochastic path-integrated differential estimator},
  author={Fang, Cong and Li, Chris Junchi and Lin, Zhouchen and Zhang, Tong},
  booktitle={NeurIPS},
  year={2018}
}

@inproceedings{wang2019spiderboost,
  title={Spiderboost and momentum: Faster variance reduction algorithms},
  author={Wang, Zhe and Ji, Kaiyi and Zhou, Yi and Liang, Yingbin and Tarokh, Vahid},
  booktitle={NeurIPS},
  year={2019}
}

@inproceedings{zhou2018stochastic,
  title={Stochastic nested variance reduction for nonconvex optimization},
  author={Zhou, Dongruo and Xu, Pan and Gu, Quanquan},
  booktitle={NeurIPS},
  year={2018}
}

@article{JMLR:v21:19-248,
  author  = {Nhan H. Pham and Lam M. Nguyen and Dzung T. Phan and Quoc Tran-Dinh},
  title   = {ProxSARAH: An Efficient Algorithmic Framework for Stochastic Composite Nonconvex Optimization},
  journal = {Journal of Machine Learning Research},
  year    = {2020},
  volume  = {21},
  number  = {110},
  pages   = {1-48}
}

@inproceedings{nguyen2017sarah,
  title={SARAH: A novel method for machine learning problems using stochastic recursive gradient},
  author={Nguyen, Lam M and Liu, Jie and Scheinberg, Katya and Tak{\'a}{\v{c}}, Martin},
  booktitle={ICML},
  year={2017},
}

@inproceedings{zhou2019faster,
  title={Faster first-order methods for stochastic non-convex optimization on Riemannian manifolds},
  author={Zhou, Pan and Yuan, Xiao-Tong and Feng, Jiashi},
  booktitle={AISTATS},
  year={2019},
}

@inproceedings{luo2020stochastic,
  title={Stochastic recursive gradient descent ascent for stochastic nonconvex-strongly-concave minimax problems},
  author={Luo, Luo and Ye, Haishan and Huang, Zhichao and Zhang, Tong},
  booktitle={NeurIPS},
  year={2020}
}

@article{huang2022accelerated,
  title={Accelerated zeroth-order and first-order momentum methods from mini to minimax optimization},
  author={Huang, Feihu and Gao, Shangqian and Pei, Jian and Huang, Heng},
  journal={Journal of Machine Learning Research},
  volume={23},
  number={36},
  pages={1--70},
  year={2022}
}

@inproceedings{palaniappan2016stochastic,
  title={Stochastic variance reduction methods for saddle-point problems},
  author={Palaniappan, Balamurugan and Bach, Francis},
  booktitle={NIPS},
  year={2016}
}

@inproceedings{chavdarova2019reducing,
  title={Reducing noise in GAN training with variance reduced extragradient},
  author={Chavdarova, Tatjana and Gidel, Gauthier and Fleuret, Fran{\c{c}}ois and Lacoste-Julien, Simon},
  booktitle={NeurIPS},
  year={2019}
}

@article{alacaoglu2021stochastic,
  title={Stochastic variance reduction for variational inequality methods},
  author={Alacaoglu, Ahmet and Malitsky, Yura},
  journal={arXiv preprint arXiv:2102.08352},
  year={2021}
}

@article{vladislav2021accelerated,
  title={On Accelerated Methods for Saddle-Point Problems with Composite Structure},
  author={Tominin, Vladislav and Tominin, Yaroslav  and Borodich, Ekaterina  and Kovalev, Dmitry  and Gasnikov, Alexander and Dvurechensky, Pavel},
  journal={arXiv preprint arXiv:2103.09344},
  year={2021}
}

@article{luo2021near,
  title={Near Optimal Stochastic Algorithms for Finite-Sum Unbalanced Convex-Concave Minimax Optimization},
  author={Luo, Luo and Xie, Guangzeng and Zhang, Tong and Zhang, Zhihua},
  journal={arXiv preprint arXiv:2106.01761},
  year={2021}
}

@inproceedings{yang2020catalyst,
  title={A catalyst framework for minimax optimization},
  author={Yang, Junchi and Zhang, Siqi and Kiyavash, Negar and He, Niao},
  booktitle={NeurIPS},
  year={2020}
}

@inproceedings{lin2020near,
  title={Near-optimal algorithms for minimax optimization},
  author={Lin, Tianyi and Jin, Chi and Jordan, Michael I.},
  booktitle={COLT},
  year={2020},
}

@inproceedings{lin2020gradient,
  title={On gradient descent ascent for nonconvex-concave minimax problems},
  author={Lin, Tianyi and Jin, Chi and Jordan, Michael I.},
  booktitle={ICML},
  year={2020}
}

@article{rafique2018non,
  title={Non-convex min-max optimization: Provable algorithms and applications in machine learning},
  author={Rafique, Hassan and Liu, Mingrui and Lin, Qihang and Yang, Tianbao},
  journal={arXiv preprint:1810.02060},
  year={2018}
}

@article{nash1953two,
  title={Two-person cooperative games},
  author={Nash, John},
  journal={Econometrica: Journal of the Econometric Society},
  pages={128--140},
  year={1953}
}

@inproceedings{yang2022faster,
  title={Faster Single-loop Algorithms for Minimax Optimization without Strong Concavity},
  author={Yang, Junchi and Orvieto, Antonio and Lucchi, Aurelien and He, Niao},
  booktitle={AISTATS},
  year={2022}
}

@inproceedings{li2021page,
  title={PAGE: A simple and optimal probabilistic gradient estimator for nonconvex optimization},
  author={Li, Zhize and Bao, Hongyan and Zhang, Xiangliang and Richt{\'a}rik, Peter},
  booktitle={ICML},
  year={2021}
}

@inproceedings{yue2023lower,
  title={On the lower bound of minimizing {P}olyak--{\L}ojasiewicz functions},
  author={Yue, Pengyun and Fang, Cong and Lin, Zhouchen},
  booktitle={COLT},
  year={2023}
}

@inproceedings{doan2022convergence,
  title={Convergence rates of two-time-scale gradient descent-ascent dynamics for solving nonconvex min-max problems},
  author={Doan, Thinh},
  booktitle={Learning for Dynamics and Control Conference},
  year={2022}
}

@inproceedings{allen2018katyusha,
  title={Katyusha {X}: Simple momentum method for stochastic sum-of-nonconvex optimization},
  author={Allen-Zhu, Zeyuan},
  booktitle={ICML},
  year={2018}
}
